\documentclass[11pt,reqno,oneside,a4paper]{article}
\usepackage[a4paper,includeheadfoot,left=25mm,right=25mm,top=00mm,bottom=20mm,headheight=20mm]{geometry}
\usepackage{amssymb,amsmath,amsthm}
\usepackage{xcolor,graphicx}
\usepackage{verbatim}
\usepackage{mathtools}
\usepackage{hyperref}
\usepackage{titling}
\usepackage{fancyhdr}
\newcommand{\runningtitle}{Running Title}
\newtheorem{thm}{Theorem}
\newtheorem{lem}[thm]{Lemma}

\newtheorem{prop}[thm]{Proposition}

\theoremstyle{definition}
\newtheorem{defn}[thm]{Definition}

\theoremstyle{remark}
\newtheorem{rmk}[thm]{Remark}

\usepackage{scalerel}
\usepackage{stackengine}
\stackMath
\newcommand\reallywidecheck[1]{%
\savestack{\tmpbox}{\stretchto{%
  \scaleto{%
    \scalerel*[\widthof{\ensuremath{#1}}]{\kern-.6pt\bigwedge\kern-.6pt}%
    {\rule[-\textheight/2]{1ex}{\textheight}}%WIDTH-LIMITED BIG WEDGE
  }{\textheight}%
}{0.5ex}}%
\stackon[1pt]{#1}{\scalebox{-1}{\tmpbox}}%
}

\newcommand\reallywidehat[1]{%
\savestack{\tmpbox}{\stretchto{%
  \scaleto{%
    \scalerel*[\widthof{\ensuremath{#1}}]{\kern-.6pt\bigwedge\kern-.6pt}%
    {\rule[-\textheight/2]{1ex}{\textheight}}%WIDTH-LIMITED BIG WEDGE
  }{\textheight}%
}{0.5ex}}%
\stackon[1pt]{#1}{\tmpbox}%
}

\NewDocumentEnvironment{MLME}{O{c}m}{%
  \multlined[#1][#2]
}{\endmultlined}
\newcommand{\NN}{\mathbb N}              % The set of naturals
\newcommand{\RR}{\mathbb R}              % The set of reals
\newcommand{\CC}{\mathbb C}              % The set of complex numbers
\renewcommand{\Re}{\operatorname*{Re}} \renewcommand{\Im}{\operatorname*{Im}}
\newcommand{\D}{\ensuremath{\,\mathrm{d}}}
\newcommand{\ri}{\ensuremath{\mathrm{i}}}
\newcommand{\re}{\ensuremath{\mathrm{e}}}
\newcommand{\la}{\ensuremath{\lambda}}
\renewcommand{\epsilon}{\varepsilon}
\renewcommand{\geq}{\geqslant}
\renewcommand{\leq}{\leqslant}
\providecommand{\BVec}[1]{\mathbf{#1}}
\providecommand{\clos}{\operatorname{clos}}
\newcommand{\abs}[1]{\left\lvert#1\right\rvert}
\newcommand{\normp}[2]{\left\lVert#1\right\rVert_{#2}}

\newcommand{\sgn}{\operatorname{sgn}}

\newcommand{\Mspacer}{\;} %Spacer for below Matrix display functions
\newcommand{\M}[3]{#1_{#2\Mspacer#3}} %Print a symbol with two subscripts eg a matrix entry
\newcommand{\Msup}[4]{#1_{#2\Mspacer#3}^{#4}} %Print a symbol with two subscripts and a superscript eg a matrix entry
\providecommand{\bigoh}[1]{\mathcal{O}\left(#1\right)}
\providecommand{\lindecayla}{\bigoh{\la^{-1}}}

\providecommand{\PV}{\mathrm{PV}\hspace{-0.4em}}

\providecommand{\argdot}{{}\cdot{}}

\newcommand{\AC}{\mathrm{AC}} % Absolutely continuous functions
\providecommand{\Lebesgue}{\mathrm{L}} % Upright L for Lebesgue spaces
\providecommand{\ContinuousSpace}{\mathrm{C}} % Upright C for Continuous spaces
\author{D A Smith}
\title{The linearized Korteweg-de Vries equation on the line with metric graph defects}
\renewcommand{\runningtitle}{LKdV line defect}
\date{\today}

\begin{document}
\maketitle
\thispagestyle{fancy}

\begin{abstract}
    We study the small amplitude linearization of the Korteweg de Vries equation on the line with a local defect scattering waves represented by a metric graph domain adjoined at one point.
    For a representative collection of examples, we derive explicit solution formulae expressed as contour integrals and obtain existence and unicity results for piecewise absolutely continuous data.
    In so doing, we implement the unified transform method on metric graphs comprising both bonds and leads for a third order differential operator.
\end{abstract}

% \tableofcontents

\section{Introduction} \label{sec:Introduction}

If waves in a channel encounter a defect, or an obstacle partially obscuring the wave transmission, then the incoming waves are scattered by the obstacle, in a possibly complicated way.
In situations which may be represented by a spatial differential operator that is selfadjoint, or at least symmetric, the theory is highly developed.
The scattering theory of the wave equation was well understood over 50 years ago~\cite{LP1967}, but continues to yield valuable results in applications; see, for example,~\cite{ZMGI2020a}.
The scattering theory of differential operators on metric graph domains has been of particlar interest recently because of potential applications to the study of metamaterials~\cite{LTC2022a}.
Linear Schr\"odinger operators on metric graph domains have also seen significant interest, not least because their scattering theory, like that for the wave equation, admits extension from a simple metric graph to periodic assemblies of such graphs~\cite{BK2013a}.
Both as a dispersive wave equation in its own right, and approximately as the small amplitude linearization of the nonlinear Schr\"odinger equation, the linear Schr\"odinger equation can be used to model physical effects from quantum interactions to optics.
In these studies, a defect in the material, the special properties of a metamaterial, or the effects of a boundary layer are modelled by adjoining to the physical domain some kind of interface region represented by a metric graph which causes some complex scattering effects.

The Korteweg-de Vries (KdV) equation is another completely integrable dispersive wave equation, with applications to unidirectional shallow water waves and blood pressure waves in elastic channels.
Its wellposedness and solitons have been studied on domains with a defect~\cite{AP2015a,CP2017a}, star graph domains~\cite{Cav2018a,AC2018a,CCM2020a,PC2021a,MNS2024a,AC2025a}, and on more complex graphs~\cite{AM2024a}.
Similar results were obtained earlier for the Benjamin Bona Mahony equation~\cite{BC2008a}.

It may be possible to adapt the inverse scattering transform approach of~\cite{Cau2015a} from the nonlinear Schr\"odinger to KdV equation (see also~\cite{Fok2002b} for KdV on the half line), and obtain solutions for specific ``linearizable'' vertex conditions, but the nonlinear nature of a certain step in the method prevents a full solution for general vertex conditions.
This will be more difficult still for metric graph domains including one or more bonds (edges of finite length), as the characterization of linearizable boundary conditions of a nonlinear equation is complicated even on a domain composed of a single finite interval.

Nevertheless, it is possible to solve in full the small amplitude linearization of the KdV equation on a metric graph domain, via the unified transform method (UTM), or Fokas transform method.
The simplest example of this, the degenerate case of a metric graph consisting of but the single semiinfinite interval $[0,\infty)$ has been solved analytically via the unified transform method~\cite{Fok2008a,DTV2014a}.
The corresponding single finite interval $[0,1]$ problem has also been solved~\cite{FP2001a,Pel2004a}, and is particularly instructive for the spectral theory it reveals.
Indeed, it has been shown that, although for certain boundary conditions the spatial differential operator may have spectral representation similar to the familiar Sturm-Liouville theory of selfadjoint second order operators, in general this is not the case.
Rather, because the (generalized) eigenfunctions do not form a complete system, an additional ``eigenvalue at infinity'' is required to represent the initial data, resulting in a solution respresentation via contour integrals~\cite{Jac1915a,Hop1919a,Pel2005a,Pap2011a,PS2013a,FS2016a,ABS2022a}.

The unified transform method has been used to solve the linearized Korteweg-de Vries equation (lKdV) and its transportless simplification (often known as the Airy equation) on more general metric graph domains.
Specifically, all metric graph problems consisting of two half line domains were studied in~\cite{DSS2016a}, where full criteria on the interface conditions for success of the method were also derived.
Note that, because the partial differential equation is unidirectional, the orientation of the half lines categorically changes the problem, leading to four such problems, of which only one pair are equivalent via change of variables.
On a star graph consisting of three half lines all of the same orientation, the Airy equation was solved via the unified transform method~\cite{SE2020a}.
On a star graph with equally many incoming and outgoing edges with a specific set of vertex conditions, the ``$\delta$-type'' vertex conditions, the lKdV was solved in~\cite[\S{}A]{PC2021a}.
Applications of the unified transform method to interface problems on metric graphs for other equations began with the heat equation on line graphs~\cite{DPS2014a,APSS2015a,MPSS2016a}, and was extended, principally by Sheils, to heat and linear Schr\"odinger problems on line and loop graphs~\cite{DS2014a,DS2016a,She2017a,DS2020a}.
A more abstract formulation of the method for arbitrary constant coefficient linear partial differential equations on metric graphs composed of finite bonds only was presented in~\cite{ABS2022a}.
The works of Farkas and Deconinck~\cite{DF2023a,DF2024a,DF2025a} return to second order equations on line graphs, but provide a means for taking a limit, via a Riemann sum, as the number of interfaces increases.

Recently, Chatziafratis has provided rigourous wellposedness results in classical spaces based on the unified transform method and solution formulae.
Although this approach has not yet been applied to problems posed on graph domains, notable results on simpler domains include~\cite{CM2022a,CKS2023a,COT2024a,ACF2024a,ACCF2024a,ACFKM2025a}.
Rigourous wellposedness results in Sobolev and Bourgain spaces based on unified transform method formulae were obtained earlier for the nonlinear Schr\"odinger, KdV and other semilinear equations, and for their small amplitude linearizations, principally by Mantzavinos and Himonas; see~\cite{FHM2016a,FHM2017a,HMY2019c,HMY2019b,HMY2019a,HM2020a,HM2021a,HM2022a} and the survey~\cite{Man2022a}.
Asymptotics of linear problems via unified transform method solution formulae have also been fruitful~\cite{CK2024a,BCCK2024a,CGK2024a,CO2024a}.

The problem of wellposedness for lKdV with various vertex conditions was decided in much greater abstraction and generality, given countably many half lines oriented in each direction, adjoined at a single vertex to form a star graph~\cite{MNS2018a}.
See also~\cite{SSVW2015a} for much of the abstract construction on which the former relies, and treatment of a first order problem in the same domain.
However, solutions of such problems have not previously been presented except in the cases where the metric graph domain is a star graph comprising just two~\cite{DSS2016a,STV2019a} or three~\cite{SE2020a} half lines, and two of these three works are only for the Airy equation, not the full lKdV.

In this work, we study lKdV, with general transport term, on certain metric graph domains, with selected vertex conditions.
The class of problems we consider is stated explicitly below, and is designed to represent the domain of a full line with, attached at $x=0$, various types of compact metric graphs, that may describe defects in the domain.
In~\S\ref{sec:UTM}, we describe the overall method, and implement the part of the method common to all of the problems.
The four classes of examples studied are then presented in detail in~\S\ref{sec:mismatch}--\ref{sec:sink}, with full existence, unicity, and solution representation results.
The appendices~\S\ref{sec:piecewiseACfunctions} and~\S\ref{sec:technicalLemmata} provide a precise definition of the function spaces in which we work, and much of the analytic details required to obtain the existence results, respectively.
The latter is an adaptation and slight extension of the recent results of Chatziafratis and collaborators~\cite{CKS2023a,COT2024a}, which give a broadly applicable framework for proving existence results using contour integral formulae obtained from the unified transform method.

\subsection*{Problem}

Consider a metric directed graph with loops allowed (henceforth, a graph) in which each directed edge $e$ has an associated spatial domain $\Omega_e$ with $x\in\Omega_e$ increasing from the source vertex of $e$ to its target.
On each edge we specify also $a_e\in\RR$.
We study the problem
\begin{subequations} \label{eqn:generalproblem}
\begin{align}
    \label{eqn:generalproblem.PDE} \tag{\theparentequation.PDE}
    \partial_t u_e(x,t) &= [\partial_x^3+a_e\partial_x] u_e(x,t) & (x,t) &\in \Omega_e \times (0,T), \\
    \label{eqn:generalproblem.IC} \tag{\theparentequation.IC}
    u_e(x,0) &= U_e(x) & x &\in\Omega_e,
\end{align}
\end{subequations}
together with some vertex conditions to be specified later.
We see that, by scaling $x$ on each directed edge separately, the coefficient of $u_{xxx}$ may also be arbitrary, at the cost of changing the vertex conditions.
Changing the sign of the $u_{xxx}$ term on an edge has the effect of reversing the orientation of the edge and changing the sign of $a_e$.
Therefore, there is no loss of generality in assuming, per~\eqref{eqn:generalproblem.PDE}, that the coefficient of $u_{xxx}$ is always $1$.
Also without loss of generality, we specify that each $\Omega_e$ is one of $(-\infty,0],[0,\infty)$, or $[0,\eta_e]$ for some $\eta_e>0$.
We shall refer to these types of directed edges as \emph{incoming leads}, \emph{outgoing leads}, and \emph{bonds}, respectively, and partition the directed edge set into $\mathcal{L}^-$, $\mathcal{L}^+$, and $\mathcal{B}$, accordingly.
In~\S\ref{sec:mismatch}--\ref{sec:sink} of this work, all problems we study have $\abs{\mathcal{L}^-}=1=\abs{\mathcal{L}^+}$, so that we may ``zoom out'' and consider the domain of equation~\eqref{eqn:generalproblem.PDE} to be the real line, but with some kind of defect at $x=0$ described by the vertex conditions and the compact component of the metric graph.

\section{UTM for lKdV on metric graphs} \label{sec:UTM}

\subsection{Overview of UTM} \label{ssec:UTM.Overview}

In this work, we apply the unified transform method to initial boundary value problems posed on metric graphs.
Due to the notational complexity necessary to even state such problems, some explanation of the general procedure is warranted before the full argument is presented.
An excellent pedagogical introduction to the unified transform method on the simplest possible metric graphs, half lines and finite intervals, is provided by~\cite{DTV2014a}.
An implementation of the method for problems on metric graphs without leads has already been provided in~\cite[\S7]{ABS2022a} but, as our graphs always have two leads, that work is not applicable.
Moreover, both these papers omit the last part of the method described below: stage three.
We provide in the following a brief explanation of the three stages of the unified transform method, with descriptions of the necessary adaptations for the problems on the metric graphs which we study.

\subsubsection*{Stage 1}
The first stage of the method is to derive a pair of equations:
\begin{description}
    \item[Ehrenpreis form.]{
        This equation gives a formula for the solution $u(x,t)$ in terms of complex contour integrals of the Fourier transform of the initial datum, and certain temporal transforms of the boundary values, which are denoted by $\M fej$ in this work.
        In the case of the heat equation on a finite interval, temporal transforms of all four of the Dirichlet and Neumann boundary values at the left and the right of the interval would be included in the integrands of this expression, yet only two (linear combinations) of these may be explicitly specified in a wellposed problem, so this is not yet an effective representation of the solution.

        One way to derive the Ehrenpreis form, and the way we follow below, is to begin with the global relation, apply the inverse Fourier transform, and then make contour deformations to certain parts of the resulting formula, so that the inverse Fourier transform becomes an integral along another contour extending from $\infty$ to $\infty$.
        The purpose of the contour deformations may be mysterious at first glance, but is clarified in stages~2 and~3.
        An older but more complicated derivation of the Ehrenpreis form via a Riemann-Hilbert problem provides good a priori motivation for the particular contours of integration~\cite{Fok2008a}.

        In the adaptation to problems on metric graphs, one must obtain an Ehrenpreis form for each of the edges, so that $u_e(x,t)$ is given for each $e$ by complex contour integrals of these transforms of initial data and boundary values.
        As the derivation of the global relation relies only upon the PDE and initial condition, not on the boundary/vertex conditions, the adaptation required for the metric graph setting are minimal.
        However, to reduce the complexity of calculations later in the paper, we make a change of spectral variable that is not usually necessary for problems on an interval, and define temporal transforms of the boundary values in a way convenient for this change of variables.
        The change of variables for each edge $e$ is given by the function~$\nu_e$, which is defined above lemma~\ref{lem:generalNu}.
    }
    \item[Global relation.]{
        This equation relates those same temporal transforms of the boundary values to the Fourier transforms of the initial datum and the solution at time $t$.
        There are several ways to derive the global relation, but the approach followed in this paper is to take an ordinary Fourier transform of the PDE, integrate by parts to obtain a very simple temporal ODE in the Fourier transform of $u$ with the boundary values considered as inhomogeneities, then solve that ODE via an integrating factor.

        Again, this derivation is performed without reference to the boundary conditions, so it may be done in general for all problems like~\eqref{eqn:generalproblem}, and is a straightforward generalization of the result on an interval.
        As with the Ehrenpreis form, we make a change of spectral variable compared to the typical presentation of the global relation for problems on the interval.
    }
\end{description}
Because it is independent of the vertex conditions, it is possible to present stage~1 for problem~\eqref{eqn:generalproblem} in full generality.
This is done in~\S\ref{ssec:UTM.Overview}.

\subsubsection*{Stage 2}
In the second stage of the method, one completes the derivation of an effective contour integral representation of the solution of the problem.
To do so, one uses the boundary conditions (or, when working on a metric graph, vertex conditions) and the global relation to derive a system of linear equations for the temporal transforms of the boundary values, solves the system, and substitutes the solution into the Ehrenpreis form.
There are two significant complications with this process.

The first issue is that (for each edge) there is but one global relation, and the set of boundary (interface) conditions is typically short by more than one linear equation of a full rank system in the transformed boundary values.
However, the temporal transforms of the boundary values depend on a power (the cube, in our problems) of the spectral variable $\la$ rather than on $\la$ itself, so, for $\alpha$ a primitive cube root of unity, the maps $\la\mapsto\alpha\la$ and $\la\mapsto\alpha^2\la$ give extra global relations in the same temporal transforms of the boundary values.
Related to this issue is that the original global relation (for a given edge) may be valid on only one of the closed halfplanes $\clos\CC^+$ or $\clos\CC^-$, but we require formulae for the temporal transforms of the boundary values valid on some other sector of $\CC$.
The solution to this issue is the same: select the ``right'' scalings $\la\mapsto\alpha^k\la$ so that the resultant versions of the global relation all become valid in the sector of interest.

The second impediment is that the global relation, involves not only transforms of the boundary values (for which we plan to solve) and initial datum (which is known), but also the Fourier transform of the solution at time $t$ (which is certainly unknown).
Therefore, this unknown term necessarily appears in the formulae obtained by solving the linear system, and will be substituted into the Ehrenpreis form.
It is subsequently shown via asymptotic analysis and the tools of complex analysis, that this term contributes zero to the solution formula once it is inside the carefully chosen contour integrals of the Ehrenpreis form.
It was in anticipation of this part of the argument that we made contour deformations in the construction of the Ehrenpreis form.

The main added difficulty for problems on metric graphs is that the vertex conditions do not typically specify a boundary value explicitly, or even give a Robin type condition in which a linear combination of only the boundary values corresponding to a single edge is specified; rather a typical vertex condition specifies an equality between boundary values of several edges joined at the vertex.
Therefore, one must use all the global relations, from all the edges, each with maps $\la\mapsto\alpha^k\la$ applied to the spectral variable.
The resulting linear system is rather larger.

The plan for stage~2 is described again in~\S\ref{ssec:stage2} with the benefit of notation that is developed in the intervening.
The full implementation of stage~2 is presented in~\S\S\ref{ssec:mismatchUnicity},~\ref{ssec:loopUnicity}, and~\ref{ssec:sourceUnicity}.

\subsubsection*{Stage 3}
At the conclusion of stage~2, a solution representation has been determined, and its derivation required only assumption of the existence of a solution.
Therefore, in particular, unicity of the solution is established under the same assumption.
Stage three of the method justifies that existence assumption, thereby bootstrapping the whole argument.

We begin with the formula obtained in stage two, and show directly that it satisfies the PDE, initial condition, and all boundary (vertex) conditions.
This is the most analytically involved section of the argument, has not previously been implemented in detail for problems on metric graphs, and is often omitted even for boundary value problems on intervals.
However, as it is required for a full rigourous implementation of the unified transform method, stage~3 is included in this work.

The methods for completing stage~3 are less standarized than those for stages~1 and~2, but some progress has been made in recent works by Chatziafratis and collaborators.
The main analytical difficulty comes from the fact that the complex contour integrals derived in stage~2 are typically not uniformly convergent, so differentiation under the integral is not possible.
Parts of the argument are also algebraically nontrivial; identities~\eqref{eqn:generalNu.vietaSymmetries} are crucial to establishing cancellation of various quantities.

A rough explanation of the analytical tools fors tage~3 is given in~\S\ref{ssec:definitions}.
This stage is presented in greatest detail for the first problem in~\S\ref{ssec:mismatchExistence}.
Its rigourous implementation relies on the sequence of technical lemmata in the appendicial~\S\ref{sec:technicalLemmata}.
The necessary further generalizations for the other problems appear in the sections dedicated to each problem.

\subsection{UTM stage 1 for lKdV on metric graphs} \label{ssec:UTM.Overview}

For each edge $e$, let $\nu_e$ be the biholomorphic function defined outside a disc of radius $R$ having $\la^3=\nu(\la)^3-a_e\nu(\la)$ and $\nu(\la)/\la=1+\lindecayla$ as $\la\to\infty$.
The existence of this function is guaranteed in the proof below.
Its unicity follows by unicity of analytic continuation from large $\abs\la$, where $a_e\nu(\la)$ is dominated by the other terms, to smaller $\abs\la$.

\begin{lem} \label{lem:generalNu}
    As $\la\to\infty$, $\nu(\la) = \la+\lindecayla$, and the asymptotic term has a control uniform in $\arg(\la)$.
    Outside $B(0,R)$, $\nu(\RR^\pm),\nu^{-1}(\RR^\pm)\subset\RR^\pm$.
    Outside $B(0,R)$, $\nu(\CC^\pm),\nu^{-1}(\CC^\pm)\subset\CC^\pm$.
    For $\alpha=\exp(2\pi\ri/3)$, the function $\nu$ has the symmetries
    \begin{subequations} \label{eqn:generalNu.vietaSymmetries}
    \begin{align} \label{eqn:generalNu.vietaSymmetries.1}
        \sum_{j=0}^2 \alpha^j\nu(\alpha^j\la)^k \nu'(\alpha^j\la) &= \begin{cases} 0 & \mbox{if } k \in \{0,1,3\}, \\ 3\la^2 & \mbox{if } k = 2, \end{cases}
        \\
         \label{eqn:generalNu.vietaSymmetries.2}
        \sum_{j=0}^2 \alpha^j\nu(\alpha^j\la)^k\left(a-\nu(\alpha^j\la)^2\right)\nu'(\alpha^j\la) &= \begin{cases} 0 & \mbox{if } k \in \{1,2\}, \\ -3\la^2 & \mbox{if } k = 0. \end{cases}
    \end{align}
    \end{subequations}
\end{lem}

\begin{proof}[Proof of lemma~\ref{lem:generalNu}]
    The existence of this biholomorphism is established in~\cite[proposition~1]{ABS2022a}.
    As $\nu$ is analytic in the exterior of $B(0,R)$, it has a Laurent series at $0$, which we express as $\nu(\la) = \la+\sum_{j=0}^\infty c_j\la^{-j}$.
    But then
    \[
        \la^3 = \nu(\la)^3 - a\nu(\la) = \la^3 + 3\la^2c_0 + \la[3c_0^2+3c_1-a] + \bigoh{1},
    \]
    so $c_0=0$ and $c_1=a/3$, which establishes the asymptotic expression $\nu(\la)=\la+\lindecayla$.
    By complete induction, all $c_j$ are real numbers, which means $\nu(\RR)\subset\RR$ and, by ensuring $R>0$ is sufficiently large, the separation of the two half lines $\nu(\{\la\in\RR:\pm\la>R\})\subset\RR^\pm$ follows from $\nu(\la)/\la\to1$.
    Bijectivity implies the corresponding results for $\nu^{-1}$.
    The fact that $\RR$ divides $\CC^+$ from $\CC^-$, together with injectivity and continuity of $\nu$, implies that $\nu(\CC^\pm)\subset\CC^\pm$, and then bijectivity implies the corresponding results for $\nu^{-1}$.

    We claim that
    \begin{equation} \label{eqn:generalNu.vietaSymmetries.proof}
        \sum_{j=0}^2 \alpha^j\nu'(\alpha^j\la) = 0,
        \qquad
        \sum_{j=0}^2 \alpha^j\nu'(\alpha^j\la)\nu(\alpha^j\la) = 0,
        \qquad
        \sum_{j=0}^2 \alpha^j\nu'(\alpha^j\la)\nu(\alpha^j\la)^2 = 3\la^2.
    \end{equation}
    The first of symmetries~\eqref{eqn:generalNu.vietaSymmetries.proof} follows from Vieta's first formula~\cite[(I.1.1.4)]{Coo1931a} applied to the polynomial $\nu(\la)^3-a\nu(\la)-\la^3$ in $\nu(\la)$, by differentiation with respect to $\la$.
    Combining Vieta's first and second formulae, one obtains $\sum_{j=0}^2 \nu(\alpha^j\la)^2 = 2a$.
    Differentiating the latter gives the second of symmetries~\eqref{eqn:generalNu.vietaSymmetries.proof}.
    The third of symmetries~\eqref{eqn:generalNu.vietaSymmetries.proof} follows similarly from Vieta's first, second, and third formulae.

    Symmetries~\eqref{eqn:generalNu.vietaSymmetries.proof} combine to give symmetry~\eqref{eqn:generalNu.vietaSymmetries.1} in all cases except $k=3$.
    But the $k=3$ case reduces to the sum of the $k=0$ and $k=1$ cases because, by definition, $\nu(\alpha^j\la)^3 = \la^3+a\nu(\alpha^j\la)$.

    Now, using $\nu(\alpha^j\la)(a-\nu(\alpha^j\la)^2)=-\la^3$, the $k=1,2$ cases of symmetry~\eqref{eqn:generalNu.vietaSymmetries.2} reduce to the $k=0,1$ cases of symmetry~\eqref{eqn:generalNu.vietaSymmetries.1}, respectively.
    The $k=0$ case of symmetry~\eqref{eqn:generalNu.vietaSymmetries.2} is a linear combination of the $k=0$ and $k=2$ cases of symmetry~\eqref{eqn:generalNu.vietaSymmetries.1}.
\end{proof}

In order to state the main proposition of this section, we define the set
\begin{equation} \label{eqn:defnD}
    D = \{\la\in\CC^-:\abs\la>R,\,\Re(\ri\la^3)<0\},
\end{equation}
and its positively oriented boundary $\partial D$, as shown in figure~\ref{fig:bdryD3}.
\begin{figure}
    \centering
    \includegraphics{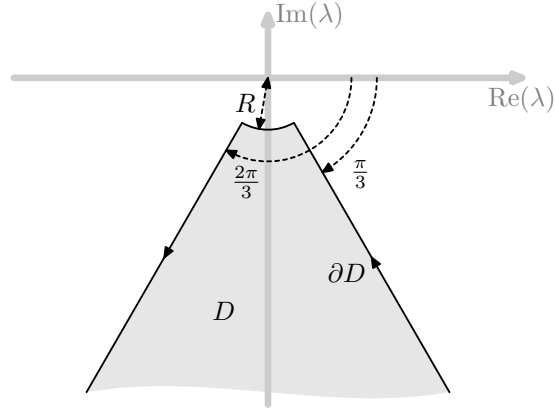}
    \caption{The contour $\partial D$.}
    \label{fig:bdryD3}
\end{figure}
Several rotations of this set appear in the proof of the proposition, but are unnecessary for its statement.
Specifically, $\alpha D$ and $\alpha^2D$ are both in $\CC^+$ together with $E=-D$, displayed in figure~\ref{fig:vertexRewritingRealMinus}, while $\alpha E$ and $\alpha^2E$ are in $\CC^-$ and shown in figure~\ref{fig:vertexRewritingRealPlus}.

\begin{prop} \label{prop:UTM}
    Suppose that, for all edges $e$, $u_e$ satisfies problem~\ref{eqn:generalproblem}.
    Suppose also that for all $x\in\Omega_e$ and all $k\in\{0,1,2\}$, $\partial_x^ku_e(x,\argdot)\in\AC_{\mathrm p}^0[0,T]$, while for all $t\in[0,T]$, $u_e(\argdot,t)\in\AC_{\mathrm p}^2(\clos\Omega_e)$; see appendicial~\S\ref{sec:piecewiseACfunctions} for a definition of these spaces.
    Let $\alpha$ be the primitive cube root of unity $\re^{2\pi\ri/3}$, $\partial D$ be the positively oriented boundary of the set defined in equation~\eqref{eqn:defnD}, and
    \begin{equation} \label{eqn:generalDefnF}
        F_e(\la;x,t) = \M fe2(\la;x,t) + \ri\nu_e(\la)\M fe1(\la;x,t) + \left(a_e-\nu_e(\la)^2\right) \M fe0(\la;x,t),
    \end{equation}
    for
    \[
        \M fej(\la;x,t) = \int_0^t \re^{\ri\la^3s} \partial_x^j u_e(x,s) \D s.
    \]
    For any integrable function $\theta$ defined on the real interval $\Omega$, we denote its Fourier transform by
    \[
        \widehat \theta(\la) = \int_\Omega \re^{-\ri\la x} \theta(x)\D\la.
    \]
    Then, for all $\la\in\clos(\CC^\mp)\setminus B(0,R)$ if $e\in\mathcal L^\pm$, and for all $\la\in\CC\setminus B(0,R)$ if $e\in\mathcal B$,
    \begin{equation} \label{eqn:generalGRe}
        0 = \re^{\ri\la^3t} \hat u_e(\nu_e(\la);t) - \widehat U_e(\nu_e(\la))
        + \chi_{\mathcal B \cup \mathcal L^+}(e) F_e(\la;0,t)
        - \chi_{\mathcal B \cup \mathcal L^-}(e) \re^{-\ri\eta_e\nu_e(\la)}F_e(\la;\eta_e,t),
    \end{equation}
    where, for notational convenience, $\eta_e=0$ for $e\in\mathcal L^-$.
    Also, for all $(x,t)\in\Omega_e\times(0,T)$,
    \begin{multline} \label{eqn:generalEFe}
        2\pi u_e(x,t) = \int_{-\infty}^\infty \re^{\ri\la x - \ri(\la^3-a_e\la)t} \widehat U_e(\la) \D\la \\
        - \chi_{\mathcal B \cup \mathcal L^+}(e) \int_{\partial D} \re^{-\ri\la^3t} \left[\alpha\re^{\ri\nu_e(\alpha\la)x}F_e(\alpha\la;0,t)\nu'_e(\alpha\la)+\alpha^2\re^{\ri\nu_e(\alpha^2\la)x}F_e(\alpha^2\la;0,t)\nu'_e(\alpha^2\la)\right] \D\la \\
        - \chi_{\mathcal B \cup \mathcal L^-}(e) \int_{\partial D} \re^{\ri\nu_e(\la)(x-\eta_e)-\ri\la^3t} F_e(\la;\eta_e,t)\nu'_e(\la) \D\la.
    \end{multline}
    Moreover, $R>0$ may be increased an arbitrary finite amount, consequently deforming $\partial D$ further from $0$, without affecting the validity of equation~\eqref{eqn:generalEFe}.
\end{prop}

\begin{proof}[Proof of proposition~\ref{prop:UTM}]
    Applying the Fourier transform to equation~\eqref{eqn:generalproblem.PDE} for each edge $e$ separately, and integrating by parts in space thrice, we obtain
    \begin{multline*}
        \reallywidehat{\partial_t u_e}(\la;t) + (-\ri\la^3+a_e\ri\la)\hat u_e(\la;t)
        \\
        =
        - \left[\re^{-\ri\la x}\left(
            \partial_{xx}u_e(x,t)
            + \ri\la \partial_{x}u_e(x,t)
            - \la^2 u_e(x,t)
            + a_e u_e(x,t)
        \right)\right]_{x\to\inf\Omega_e}^{x\to\sup\Omega_e}.
    \end{multline*}
    The regularity in time hypothesis is sufficient to interchange the time derivative and Fourier transform.
    Treating the right side as if it were a known inhomogeneity, we can solve this temporal ODE to obtain
    \begin{equation} \label{generalGRe.la}
        \re^{(\ri\la^3-a_e\ri\la)t} \hat u_e(\la;t)
        =
        \widehat U_e(\la)
        -
        \chi_{\mathcal B \cup \mathcal L^+}(e) F_e(\nu_e^{-1}(\la);0,t)
        +
        \chi_{\mathcal B \cup \mathcal L^-}(e) \re^{-\ri\eta_e\la}F_e(\nu_e^{-1}(\la);\eta_e,t).
    \end{equation}
    In equation~\eqref{generalGRe.la}, the composition of $\nu_e^{-1}$ with the first argument of $F_e$ should be understood in the formal sense of ``replace every $\nu_e(\la)$ with $\la$ in the definition of $F_e(\la)$ and the $\la^3$ with $\la^3-a_e\la$ in the definition of $\M fej$'', so that the resulting function is continuously defined on $\clos(\CC^\mp)$ for $e\in\mathcal L^\pm$, and $\CC$ for $e\in\mathcal B$.
    The spatial regularity of $u_e$ then implies that this function is analytic on the interior of its domain of definition.
    Equation~\eqref{eqn:generalGRe} follows by a change of variables $\la\mapsto\nu(\la)$.

    Dividing by the exponential factor and applying the inverse Fourier transform to equation~\eqref{generalGRe.la}, we arrive at
    \begin{multline*}
        2\pi u_e(x,t) = \int_{-\infty}^\infty \re^{\ri\la x - \ri(\la^3-a_e\la)t} \widehat U_e(\la) \D\la
        - \chi_{\mathcal B \cup \mathcal L^+}(e) \int_{-\infty}^\infty \re^{\ri\la x - \ri(\la^3-a_e\la)t} F_e(\nu_e^{-1}(\la);0,t) \D\la \\
        + \chi_{\mathcal B \cup \mathcal L^-}(e) \int_{-\infty}^\infty \re^{\ri\la(x-\eta_e) - \ri(\la^3-a_e\la)t} F_e(\nu_e^{-1}(\la);\eta_e,t) \D\la.
    \end{multline*}
    The first integral already matches that in equation~\eqref{eqn:generalEFe}.
    We make finite contour deformations so that the contour of the second integral lies in $\clos(\CC^+)$ and the third lies in $\clos(\CC^-)$ and the new contours lie in the domain of biholomorphicity of $\nu_e^{-1}$.
    The change of variables $\la\mapsto\nu_e(\la)$ in the second and third integrals results in
    \begin{multline} \label{eqn:generalEFe.beforeDeformation}
        2\pi u_e(x,t) = \int_{-\infty}^\infty \re^{\ri\la x - \ri(\la^3-a_e\la)t} \widehat U_e(\la) \D\la
        - \chi_{\mathcal B \cup \mathcal L^+}(e) \int_{\Gamma_-} \re^{\ri\nu_e(\la)x-\ri\la^3t} F_e(\la;0,t)\nu'_e(\la) \D\la \\
        + \chi_{\mathcal B \cup \mathcal L^-}(e) \int_{\Gamma_+} \re^{\ri\nu_e(\la)(x-\eta_e)-\ri\la^3t} F_e(\la;\eta_e,t)\nu'_e(\la) \D\la,
    \end{multline}
    where $\Gamma_\pm$ is the concatenation $(-\infty,-R)\cup\{\re^{\mp\ri(\pi-\theta)}:\theta\in[0,\pi]\}\cup(R,\infty)$; see figure~\ref{fig:Gammapm}.

    Suppose $e$ is a bond or an outgoing lead and consider $\la\to\infty$ in the closed sector with argument ranging from $\frac\pi3$ to $\frac{2\pi}3$.
    Let $0=\tau_0<\tau_1<\ldots<\tau_n=t$ be the partition associated with $\partial_x^j u_e(x,\argdot)\rvert_{[0,t]}$, per definition~\ref{defn:piecewiseACfunctions}.
    Integrating by parts, we see that
    \begin{align*}
        \abs{\re^{-\ri\la^3t}\M fej(\la;x,t)}
        &= \frac1{\abs\la^3} \abs{ \sum_{k=1}^n\left[\re^{-\ri\la^3\tau_k}\partial_x^ju_e(x,\tau_k^-) - \re^{-\ri\la^3\tau_{k-1}}\partial_x^ju_e(x,\tau_{k-1}^+)\right] - \int_0^t \re^{\ri\la^3(s-t)}\partial_t\partial_x^ju_e(x,s) \D s } \\
        &\leq \frac1{\abs\la^3} \left[
            2 n \max_{k\in\{0,1,\ldots,n\}}\abs{\partial_x^ju_e(x,\tau_k^\pm)}
            + \normp{\partial_t\partial_x^ju_e(x,\argdot)}{\Lebesgue^1}
        \right].
    \end{align*}
    Therefore, by lemma~\ref{lem:generalNu}, $\re^{-\ri\la^3t}F_e(\la;0,t)\nu'_e(\la)=\lindecayla$ here, and this decay is uniform in the argument of $\la$ within the sector.
    Hence, by Cauchy's theorem and Jordan's lemma,
    \[
        0 = \int_{\partial E} \re^{\ri\nu_e(\la)x-\ri\la^3t} F_e(\la;0,t)\nu'_e(\la) \D\la.
    \]
    Subtracting this from the second term in equation~\eqref{eqn:generalEFe.beforeDeformation}, we obtain
    \[
        - \int_{\Gamma_-} \re^{\ri\nu_e(\la)x-\ri\la^3t} F_e(\la;0,t)\nu'_e(\la) \D\la
        =
        - \int_{\partial\alpha D \cup \partial\alpha^2 D} \re^{\ri\nu_e(\la)x-\ri\la^3t} F_e(\la;0,t)\nu'_e(\la) \D\la,
    \]
    in which the contour encloses the two sectors with arguments from $0$ to $\frac\pi3$ and from $\frac{2\pi}3$ to $\pi$, except deformed away from $0$.
    Changes of variables $\la\mapsto\alpha\la$, $\la\mapsto\alpha^2\la$ for the two connected components of the contour map both to $\partial D$, whence the second integral of equation~\eqref{eqn:generalEFe}.

    By a similar argument, though applying Jordan's lemma in the two sectors $\alpha E$ and $\alpha^2 E$ shown in figure~\ref{fig:vertexRewritingRealPlus}, rather than in $E$, the contour $\Gamma_-$ may be deformed down to $\partial D$ but traversed in the opposite direction.
    The sign change to arrive at the third integral of equation~\eqref{eqn:generalEFe} arises from the switch in orientation; no change of variables is necessary in this case.
    The final claim on increasing $R$ follows from the analyticity of the integrands.
\end{proof}

We call equation~\eqref{eqn:generalGRe} the \emph{global relation} and equation~\eqref{eqn:generalEFe} the \emph{Ehrenpreis form}.
These are slightly different from the usual equations of the same name in the literature on the unified transform method.
The Ehrenpreis form has a change of variables in the integrals about $\partial D$, from $\la$ to $\nu(\la)$.
This is common in the literature on the unified transform method for interface problems (see, for example,~\cite{DS2020a,ABS2022a}), as it provides a significant algebraic simplification to the D-to-N maps that must be constructed later in the method.
Further, following~\cite{DS2020a}, we have mapped all three contour integrals about $\partial\alpha^kD$ to the single contour $\partial D$.
This is done not because $\partial D$ has any special merit over the others, but because having all on the same contour provides a further algebraic simplification of the D-to-N maps which is relevant only for problems of spatial order at least 3 in which the domain includes a lead.
The global relation differs only by a change of variables $\la\mapsto\nu_e(\la)$, in order to accommodate the corresponding change in the Ehrenpreis form; the usual form of the global relation would be equation~\eqref{generalGRe.la}.

\subsection{Sketch of UTM stage~2} \label{ssec:stage2}

From equation~\eqref{eqn:generalEFe}, it remains to implement stage~1, to determine $F_e$ for a particular set of vertex conditions so that an explicit solution formula may be provided.
In~\S\ref{sec:mismatch}--\ref{sec:sink}, we shall consider several metric graphs with associated classes of vertex conditions and, in each case, produce explicit formulae that may substituted for $F_e$, solving the problem.

To find formulae for $F_e$, we will use the global relations~\eqref{eqn:generalGRe} under appropriate choices of map $\la\mapsto\alpha^k\la$ for integer $k$.
Because $\M fej$ is invariant under this map, this allows us to obtain from the global relations sufficiently many equations in the $\M fej$ that, together with the vertex conditions, formulae for all $\M fej$ may be found.
Since the global relations include the term $\re^{\ri\la^3t} \hat u_e(\nu_e(\la);t)$, the resulting expressions for the $\M fej$ depend also upon this term, so substituting these expressions into the Ehrenpreis form~\eqref{eqn:generalEFe} does not prima facie yield a computable representation of the solution of problem~\eqref{eqn:generalproblem}.
Nevertheless, as we shall show in each case, having been so integrated, those problematic terms contribute $0$ to the solution representation, at least for $R$ chosen sufficiently large.

\subsection{Definitions of contours for UTM stage~3} \label{ssec:definitions}

There are a number of technical arguments necessary for the existence theorems in this work, which require the development of some weight of notation.
As much as possible, the arguments and associated notation are confined to the progression of lemmata in the appendicial~\S\ref{sec:technicalLemmata}.
However, in establishing that those results are applicable to each of the problems we study, it is necessary to use some of this notation in the main body of the work.
That subset of the notation is collected here.

Based on $\partial D$, we define two perturbations:
\begin{align}
    \label{eqn:defn.Gamma}
    \Gamma &= \alpha^2(-\epsilon\ri+(-\infty,-R])\cup\ri[-\epsilon,0]\cup\{R\re^{-\ri k}:k\in[\tfrac\pi3,\tfrac{2\pi}3]\}\cup-\ri[0,\epsilon]\cup\alpha(-\epsilon\ri+[R,\infty)), \\
    \gamma &= \re^{-\ri\left(\tfrac{2\pi}3-\epsilon'\right)}(-\infty,-R] \cup \{R\re^{-\ri \theta}:\theta\in[\tfrac\pi3+\epsilon',\tfrac{2\pi}3-\epsilon']\} \cup \re^{-\ri\left(\tfrac\pi3+\epsilon'\right)}[R,\infty),
\end{align}
where $\epsilon\in(0,1/\sqrt3)$ and $\epsilon'\in(0,\tfrac\pi6)$ are arbitrary and unimportant.
The closed set bounded between $\partial D$ and $\Gamma$, including the circular arc they share, is called $S$.
Similarly, $s$ is the closed set between $\partial D$ and $\gamma$, including their common circular arc.
These sets and contours are displayed in figure~\ref{fig:gammaGamma}.
\begin{figure}
    \centering
    \includegraphics{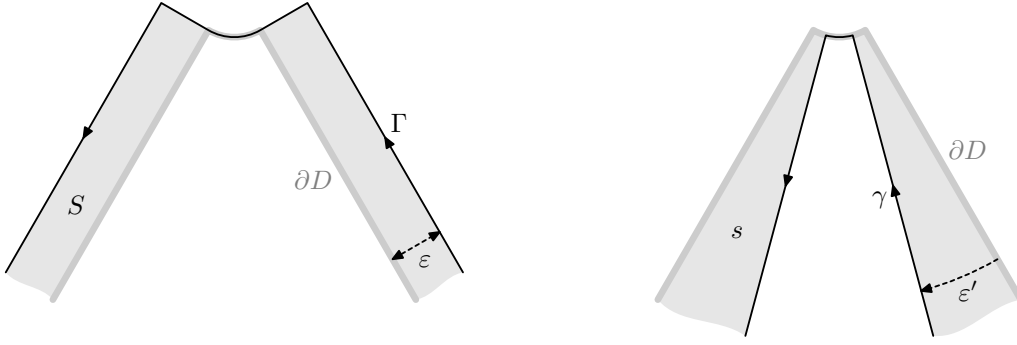}
    \caption{The contours $\Gamma$ and $\gamma$ and the closed sets $S$ and $s$.}
    \label{fig:gammaGamma}
\end{figure}

The utility of the contours $\Gamma$ and $\gamma$ is as follows.
Given an integral over $\partial D$ with integrand $\theta=\phi+\varphi+\psi$, we will deform the $\phi$ part to $\Gamma$ and the $\varphi$ part to $\gamma$ while the remainder $\psi$ is left on $\partial D$.
To make such an argument, the contour must be deformed over $S$ and $s$, and analyticity and decay properties of $\phi$ and $\varphi$ in these sets must be established.
Having done this splitting and deformation, because each of $\phi,\varphi,\psi$ will have been carefully selected to have appropriate decay on their new contour, the resulting integrals converge uniformly, so various derivatives and limits may be evaluated directly.

\section{Mismatch defect} \label{sec:mismatch}

The simplest metric graph among the class we consider is that with $\mathcal B=\emptyset$, so that the metric graph is composed of just two leads, as displayed in figure~\ref{fig:graph-mismatch}.
\begin{figure}
    \centering
    \includegraphics{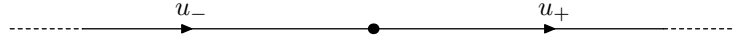}
    \caption{The metric graph domain with a mismatch defect.}
    \label{fig:graph-mismatch}
\end{figure}
We study that case here, labelling the edges with the symbols $-$ and $+$ so that $\mathcal L^-=\{-\}$ and $\mathcal L^+=\{+\}$, so $\Omega_-=(-\infty,0]$ and $\Omega_+=[0,\infty)$.
If $a_+=a_-$ and we were to specify vertex conditions enforcing $\ContinuousSpace^2$ continuity at $x=0$, we would simply have a full line problem for the lKdV equation, which may be solved via the ordinary Fourier transform method.
If instead $a_+=a_-=0$, then the full set of admissible vertex conditions that can specify well posed problems has been investigated previously in~\cite{DSS2016a}.

Guided by these results, we expect it is reasonable to ask for three vertex conditions to be imposed, and pick the simple form that the left and right limits at $x=0$ of $u$ differ by a given scalar factor, and the same holds for each of $u_x$ and $u_{xx}$, albeit with possibly different scalars.
This specifies the following problem with a \emph{mismatch} defect.
\begin{subequations} \label{eqn:mismatchproblem}
\begin{align}
    \label{eqn:mismatchproblem.PDE} \tag{\theparentequation.PDE}
    \partial_t u_\pm(x,t) &= [\partial_x^3+a_\pm\partial_x] u_{\pm}(x,t) & (x,t) &\in \Omega_\pm \times (0,T), \\
    \label{eqn:mismatchproblem.IC} \tag{\theparentequation.IC}
    u_\pm(x,0) &= U_\pm(x) & x &\in\Omega_\pm, \\
    \label{eqn:mismatchproblem.VC} \tag{\theparentequation.VC}
    \partial_x^k u_-(0,t) &= B_k^{-1} \partial_x^k u_+(0,t), & t &\in [0,T], \qquad k\in\{0,1,2\},
\end{align}
\end{subequations}
with coefficients $B_k\in\CC\setminus\{0\}$.

We do not mean to suggest that the vertex conditions for wellposed problems on this metric graph are necessarily of type~\eqref{eqn:mismatchproblem.VC}.
Indeed, it is known that this false~\cite{DSS2016a,MNS2018a}.
Rather, this class of vertex conditions has been selected because it is both restrictive enough that many formulae simplify considerably, and broad enough to demonstrate how wellposedness criteria (such as the requirement $\sum B_k\neq0$ in the below proposition) arise from the method.
We aim with these vertex conditions, as with those in~\S\ref{sec:loop}--\ref{sec:sink}, to balance clarity and generality of exposition, rather than provide a comprehensive theorem.
Vertex conditions~\eqref{eqn:mismatchproblem.VC} are not motivated by particular applications, but they may be seen as a mathematically natural generalization of perfect thermal contact conditions~\cite{CJ1959a} from the heat equation to the lKdV equation.

\begin{prop} \label{prop:mismatchSolution}
    Suppose $u$ is sufficiently smooth and satisfies problem~\eqref{eqn:mismatchproblem}, in which $\sum B_k\neq0$.
    Let
    \begin{equation} \label{eqn:mismatch.A}
        \mathcal A(\la) =
        \begin{pmatrix}
            1 & \M\nu-{2} & \left(a_--\Msup\nu-22\right) \\
            1 & \M\nu-{1} & \left(a_--\Msup\nu-12\right) \\
            B_2 & B_1\M\nu+0 & B_0\left(a_+-\Msup\nu+02\right)
        \end{pmatrix},
        \qquad
        \Delta(\la)=\det\mathcal A(\la),
    \end{equation}
    and define $\delta_k(\la)$ to be the determinant of $\mathcal A(\la)$ but with the $k$\textsuperscript{th} column replaced by
    \begin{equation} \label{eqn:mismatchYForSolution}
        Y=
        \begin{pmatrix}
            -\widehat U_-(\M\nu-2) \\ -\widehat U_-(\M\nu-1) \\ \widehat U_+(\M\nu+0)
        \end{pmatrix}.
    \end{equation}
    Then
    \begin{subequations} \label{eqn:mismatchSolution}
    \begin{align}
    \notag
        2\pi u_-(x,t)
        &= \int_{-\infty}^\infty \re^{\ri\la x - \ri(\la^3-a_-\la)t} \widehat U_-(\la) \D \la \\
    \label{eqn:mismatchSolution.-} \tag{\theparentequation.\(-\)}
        &\hspace{3em} - \int_{\partial D} \re^{\ri\nu_-(\la)x-\ri\la^3t} \left[ \frac{\delta_1(\la) + \nu_-(\la)\delta_2(\la) + (a_--\nu_-(\la)^2)\delta_3(\la)}{\Delta(\la)} \right] \nu'_-(\la) \D\la, \displaybreak[0] \\
    \notag
        2\pi u_+(x,t)
        &= \int_{-\infty}^\infty \re^{\ri\la x - \ri(\la^3-a_+\la)t} \widehat U_+(\la) \D \la \\
    \notag
        &\hspace{3em} - \int_{\partial D} \frac{\re^{-\ri\la^3t}}{\Delta(\la)} \Bigg[
            B_2\delta_1(\la) \left( \alpha \re^{\ri\nu_+(\alpha\la)x} \nu'_+(\alpha\la) + \alpha^2 \re^{\ri\nu_+(\alpha^2\la)x} \nu'_+(\alpha^2\la) \right) \\
    \notag
            &\hspace{6em} + B_1\delta_2(\la) \left( \alpha \nu_+(\alpha\la)\re^{\ri\nu_+(\alpha\la)x} \nu'_+(\alpha\la) + \alpha^2 \nu_+(\alpha^2\la) \re^{\ri\nu_+(\alpha^2\la)x} \nu'_+(\alpha^2\la) \right) \\
    \label{eqn:mismatchSolution.+} \tag{\theparentequation.\(+\)}
            &\hspace{-4em} + B_0\delta_3(\la) \left( \alpha \left(a_+-\nu_+(\alpha\la)^2\right) \re^{\ri\nu_+(\alpha\la)x} \nu'_+(\alpha\la) + \alpha^2 \left(a_+-\nu_+(\alpha^2\la)^2\right) \re^{\ri\nu_+(\alpha^2\la)x} \nu'_+(\alpha^2\la) \right)
        \Bigg] \D\la.
    \end{align}
    \end{subequations}
\end{prop}

It follows immediately from proposition~\ref{prop:mismatchSolution} that the solution is unique.
Indeed, any solutions $(u_-,u_+)$ that exist must be given by explicit formulae~\eqref{eqn:mismatchSolution}.
The following proposition establishes existence of a solution.

\begin{prop} \label{prop:mismatchExistence}
    Assume $U_e\in\AC^4_{\mathrm p}(\Omega_e)$, per definition~\ref{defn:piecewiseACfunctions}.
    For all $x\in\Omega_e$ and $t\geq0$, let $u_e(x,t)$ be defined by equations~\eqref{eqn:loopSoln}.
    Then the functions $u_e$ satisfy problem~\eqref{eqn:mismatchproblem}.
\end{prop}

\begin{figure}
    \centering
    \includegraphics[width=0.65\textwidth]{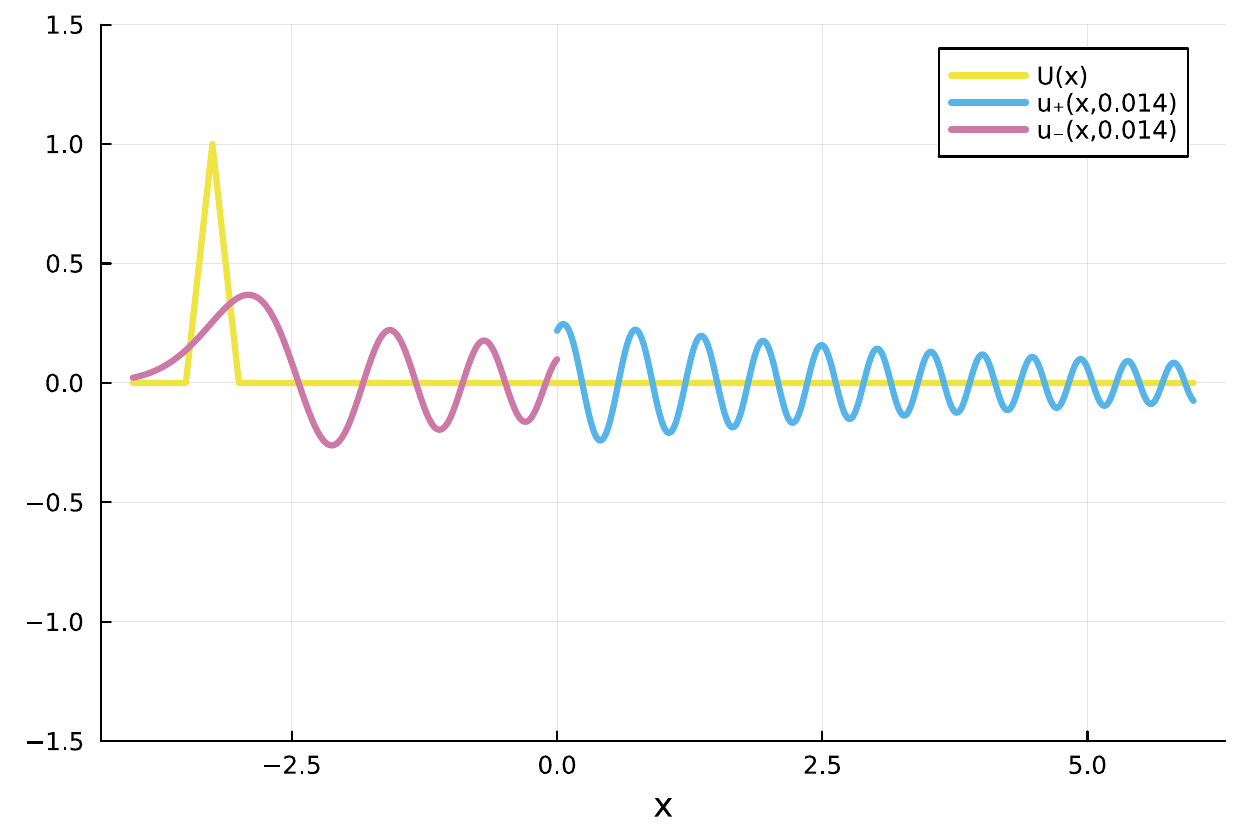}
    \caption{
        The solution of the mismatch problem at time $t=0.014$ with vertex coefficients $(B_0,B_1,B_2)=(2.2,2.0,1.1)$, with initial data shown in yellow.
    }
    \label{fig:plot-mismatch-eg}
\end{figure}

The solution of this problem for a particular initial datum and set of boundary coefficients is displayed in figure~\ref{fig:plot-mismatch-eg}.
The initial shock disperses into a wave of decreasing amplitude and slowly decreasing wavelength as $x$ increases.
The vertex coefficients cause the jump in $u$ and its spatial derivatives at the interface.

\subsection{Solution representation: proof of proposition~\protect\ref{prop:mismatchSolution}} \label{ssec:mismatchUnicity}

Imposing the boundary conditions on the global relations~\eqref{eqn:generalGRe}, we obtain
\begin{subequations} \label{eqn:mismatchGR}
\begin{align}
    \label{eqn:mismatchGR.-} \tag{\theparentequation.\(-\)}
    0 &= \re^{\ri\la^3t} \hat u_-(\M\nu-0;t) - \widehat U_-(\M\nu-0) - \left[ \M f-2 + \ri \M\nu-0 \M f-1 + \left( a_- - \M\nu-0^2 \right) \M f-0 \right] \\
    \label{eqn:mismatchGR.+} \tag{\theparentequation.\(+\)}
    0 &= \re^{\ri\la^3t} \hat u_+(\M\nu+0;t) - \widehat U_+(\M\nu+0) + \left[ B_2\M f-2 + \ri \M\nu+0 B_1 \M f-1 + \left( a_+ - \M\nu+0^2 \right) B_0 \M f-0 \right],
\end{align}
\end{subequations}
where we use the notation
\[
    \M\nu ej = \nu_e(\alpha^j\la),
    \qquad
    \M fek = \M fek(\la;0,t).
\]
The global relation~\eqref{eqn:mismatchGR.-} is valid on $\clos(\CC^+)$ and \eqref{eqn:mismatchGR.+} is valid on $\clos(\CC^-)$.
We wish to solve for $\M f-2,\M f-1,\M f-0$ on $\clos(D)$, so these global relations can be useful if we apply $\la\mapsto\alpha^j\la$ for appropriate choice of $j$.
Specifically, we use equations
\[
    \mbox{\eqref{eqn:mismatchGR.-}}\Big\rvert_{\la\mapsto\alpha^{2}\la},
    \qquad
    \mbox{\eqref{eqn:mismatchGR.-}}\Big\rvert_{\la\mapsto\alpha\la},
    \qquad
    \mbox{\eqref{eqn:mismatchGR.+}}\Big\rvert_{\la\mapsto\la}.
\]
This yields the linear system, for $\la\in\clos(D)$,
\begin{equation} \label{eqn:mismatchLinearSystem}
    \mathcal A(\la) \begin{pmatrix} \M f-2 \\ \ri \M f-1 \\ \M f-0 \end{pmatrix} = Y(\la),
\end{equation}
where
$\mathcal A(\la)$ is defined by equation~\eqref{eqn:mismatch.A}
and
\begin{equation}
    Y(\la) =
    \re^{\ri\la^3t}
    \begin{pmatrix}
        \hat u_-(\M\nu-2;t) \\ \hat u_-(\M\nu-1;t) \\ -\hat u_+(\M\nu+0;t)
    \end{pmatrix}
    -
    \begin{pmatrix}
        \widehat U_-(\M\nu-2) \\ \widehat U_-(\M\nu-1) \\ -\widehat U_+(\M\nu+0)
    \end{pmatrix}.
\end{equation}

Solving system~\eqref{eqn:mismatchLinearSystem} by Cramer's rule, we find that
\[
    \M f-2 = \frac{\delta_1(\la)}{\Delta(\la)}, \qquad
    \M f-1 = \frac{\delta_2(\la)}{\Delta(\la)}, \qquad
    \M f-0 = \frac{\delta_3(\la)}{\Delta(\la)},
\]
where $\Delta=\det\mathcal A$, and $\delta_k$ is the determinant of $\mathcal A$ but with the $k$th column replaced by $Y$.
We aim to substitute these formulae into the Ehrenpreis form~\eqref{eqn:generalEFe} to obtain a solution representation for $u_\pm$, and doing so would yield equations~\eqref{eqn:mismatchSolution}, but with $\delta_k$ as defined in this paragraph instead of as defined in the statement of proposition~\eqref{prop:mismatchSolution}.
However, to obtain an effective solution representation, we have to remove the effects of the first column vector in the definition of $Y$.
Unless this is done, the formulae for $u_\pm(x,t)$ would each depend on both $\hat u_-(\argdot;t)$ and $\hat u_+(\argdot;t)$.

Because of the asymptotic form for $\nu_\pm(\la)$,
\[
    \Delta(\la)=-\sqrt3\ri\la^3(B_0+B_1+B_2)+\bigoh{\la^2} \mbox{ as } \la\to\infty \mbox{ within } \clos(D).
\]
Similarly, using just the $\hat u_\pm(\argdot;t)$ factor of the time dependent term in $Y$ to define the $\delta_k$, it follows from integration by parts that $\delta_1(\la)=\bigoh{\la^2}$, $\delta_2(\la)=\bigoh{\la}$, and $\delta_3(\la)=\bigoh{1}$, as $\la\to\infty$ within $\clos(D)$.
Indeed, for example, because $\M\nu ej = \alpha^j\la + \lindecayla$,
\begin{align*}
    \delta_3(\la)
    &= \hat u_-(\M\nu-2;t) \left( B_1\M\nu+0-B_2\M\nu-1 \right)
    + \hat u_-(\M\nu-1;t) \left( B_2\M\nu-2-B_1\M\nu+0 \right) \\ &\hspace{20em}
    - \hat u_+(\M\nu+0;t) \left( \M\nu-1-\M\nu-2 \right) \\
    &= (B_1-B_2\alpha)\la \int_{-\infty}^0 \re^{-\ri\alpha^2\la x}u_-(x,t) \D x \\ &\hspace{14em}
    + (B_2\alpha^2-B_1)\la \int_{-\infty}^0 \re^{-\ri\alpha\la x}u_-(x,t) \D x \\ &\hspace{14em}
    - (\alpha-\alpha^2)\la \int_0^\infty \re^{-\ri\la x}u_+(x,t) \D x + \lindecayla \\
    &= (B_1-B_2\alpha)\ri\alpha \left[ u_-(0,t) - \int_{-\infty}^0 \re^{-\ri\alpha^2\la x}\partial_xu_-(x,t) \D x \right] \\ &\hspace{8.5em}
    + (B_2\alpha^2-B_1)\ri\alpha^2 \left[ u_-(0,t) - \int_{-\infty}^0 \re^{-\ri\alpha\la x}\partial_xu_-(x,t) \D x \right] \\ &\hspace{8.5em}
    - (\alpha-\alpha^2)\ri \left[ u_+(0,t) - \int_0^\infty \re^{-\ri\la x}\partial_xu_+(x,t) \D x \right] + \lindecayla
\end{align*}
and, for all $\la\in\clos(D)$, both $\Re(-i\alpha^2\la x)\leq0$ and $\Re(-i\alpha\la x)\leq0$ for $x\in(-\infty,0]$, while $\Re(-i\la x)\leq0$ for $x\in[0,\infty)$, so the integrals are bounded in $\la$ and the whole of $\delta_3(\la)=\bigoh1$.
Thus, still with only that part of $Y$ in the definitions of $\delta_k$, and assuming that $R$ is taken sufficiently large that not only are both $\nu_\pm$ biholomorphic but also $\Delta$ has no zeros, Jordan's lemma implies that
\begin{subequations} \label{eqn:mismatchWithoutTimeD3IntegralsZero}
\begin{align}
    \label{eqn:mismatchWithoutTimeD3IntegralsZero.12}
    0 &= \int_{D} \frac{\re^{\ri\nu_+(\alpha^j\la)x}}{\Delta(\la)} \left[ B_2\delta_1(\la) + B_1\delta_2(\la) + \left( a_+ - \nu_+(\alpha^j\la)^2 \right) B_0\delta_3(\la) \right] \D \la, \qquad j=1,2, \\
    \label{eqn:mismatchWithoutTimeD3IntegralsZero.3}
    0 &= \int_{D} \frac{\re^{\ri\nu_-(\la)x}}{\Delta(\la)} \left[ \delta_1(\la) + \delta_2(\la) + \left( a_- - \nu_-(\la)^2 \right) \delta_3(\la) \right] \D \la;
\end{align}
\end{subequations}
the $\hat u(\argdot;t)$ terms make no contribution to the solution representation.
Therefore, redefining $Y$ according to equation~\eqref{eqn:mismatchYForSolution} and $\delta_k$ to be the determinants of $\mathcal A$ but with the $k$\textsuperscript{th} column replaced by this new $Y$, as in the statement of proposition~\ref{prop:mismatchSolution}, we obtain the effective solution representation~\eqref{eqn:mismatchSolution}.

\subsection{Solution satisfies the problem: proof of proposition~\protect\ref{prop:mismatchExistence}} \label{ssec:mismatchExistence}

We begin with a lemma describing the asymptotic behaviour of the integrands for the integrals about $\partial D$ in solution representation~\eqref{eqn:mismatchSolution}.
The corresponding result for the integrands of the real integrals is stated as lemma~\ref{lem:FTLeadingOrder}.
The character of both these results is that, within appropriately chosen sectors, the integrands are decaying.
More than that, the integrands may be decomposed into a sum of a part with $\bigoh{\la^{-5}}$ decay and another part with only $\lindecayla$ decay, but the latter part may be analytically extended a little outside the sector, onto semistrips parallel to the boundaries of the sector, and still maintain this $\lindecayla$ behaviour.
Specifically, in the notation of lemma~\ref{lem:mismatchConvergenceZetaLeadingOrder}, equations~\eqref{eqn:mismatchSolution} can be expressed as
\begin{subequations} \label{eqn:mismatchSolutionRewrite}
\begin{align}
\label{eqn:mismatchSolutionRewrite.-} \tag{\theparentequation.\(-\)}
    2\pi u_-(x,t)
    &= \int_{-\infty}^\infty \re^{\ri\la x - \ri(\la^3-a_-\la)t} \widehat U_-(\la) \D \la
    - \int_{\partial D} \re^{\ri\nu_-(\la)x-\ri\la^3t} \frac{\M\zeta-3(\la)}{\Delta(\la)} \D\la, \\
\notag
    2\pi u_+(x,t)
    &= \int_{-\infty}^\infty \re^{\ri\la x - \ri(\la^3-a_+\la)t} \widehat U_+(\la) \D \la \\
\label{eqn:mismatchSolutionRewrite.+} \tag{\theparentequation.\(+\)}
    &\hspace{3em} - \int_{\partial D} \re^{-\ri\la^3t} \Bigg[
        \alpha \re^{\ri\nu_+(\alpha\la)x} \frac{\M\zeta+1(\la)}{\Delta(\la)}
        + \alpha^2 \re^{\ri\nu_+(\alpha^2\la)x} \frac{\M\zeta+2(\la)}{\Delta(\la)}
    \Bigg] \D\la,
\end{align}
\end{subequations}
and then lemmata~\ref{lem:FTLeadingOrder} and~\ref{lem:mismatchConvergenceZetaLeadingOrder} permit separation of each of $\M\zeta ej(\la)/\Delta(\la)$ and $\widehat U(\la)$ into components $\psi(\la)=\bigoh{\la^{-5}}$ and $\phi=\lindecayla$ and extensible.
This decomposition and analytic continuation of the less quickly decaying part will be crucial for evaluating $x\to0^\pm$ and $t\to0^+$ limits and derivatives of the various integrals in equations~\eqref{eqn:mismatchSolution}, per~\cite{CKS2023a,COT2024a} and~\S\ref{sec:technicalLemmata}.

\begin{lem} \label{lem:mismatchConvergenceZetaLeadingOrder}
    For $U_-$ and $U_+$ as in proposition~\ref{prop:mismatchExistence}, and $\delta_k$ the determinant of $\mathcal A$ with the $k$th column replaced by $Y$ from equation~\eqref{eqn:mismatchYForSolution}, let
    \begin{align*}
        \M\zeta-3(\la) &= \left[\delta_1(\la) + \nu_-(\la)\delta_2(\la) + (a_--\nu_-(\la)^2)\delta_3(\la)\right]\nu'_-(\la), \\
        \M\zeta+1(\la) &= \left[\delta_1(\la) + \nu_+(\alpha\la)\delta_2(\la) + (a_+-\nu_+(\alpha\la)^2)\delta_3(\la)\right]\nu'_+(\alpha\la), \\
        \M\zeta+2(\la) &= \left[\delta_1(\la) + \nu_+(\alpha^2\la)\delta_2(\la) + (a_+-\nu_+(\alpha^2\la)^2)\delta_3(\la)\right]\nu'_+(\alpha^2\la).
    \end{align*}
    Then we can represent
    \[
        \M\zeta ej(\la)/\Delta(\la) = \M\phi ej(\la) + \M\psi ej(\la),
    \]
    such that $\M\psi ej(\la) = \bigoh{\la^{-5}}$ as $\la\to\infty$ along $\partial D$ and $\M\phi ej$ is holomorphically extensible to closed semistrips
    \[
        S = \{\alpha^2(z-\ri y): z>R,\; 0\leq y \leq \epsilon\} \cup \{\alpha(z-\ri y): z<-R,\; 0\leq y \leq \epsilon\},
    \]
    as shown in figure~\ref{fig:gammaGamma}, and obeys
    \(
        \M\phi ej(\la)=\bigoh{\la^{-1}}
    \)
    as $\la\to\infty$ within $S$.
    More precisely, if $X_e$ is the partition for $U_e$, then
    \[
        \M\phi ej(\la) = \frac1\la\sum_{\xi\in X_e} \M ce\xi \, \re^{-\ri\la\xi} + \bigoh{\la^{-2}},
    \]
    for some constants $\M ce\xi\in\CC$.
\end{lem}

\begin{proof}
    By construction,
    \[
        \delta_3(\la)
        =
        \widehat U_+(\M\nu+0)(\M\nu-1-\M\nu-2)
        + \widehat U_-(\M\nu-1)(B_1\M\nu+0-B_2\M\nu-2)
        - \widehat U_-(\M\nu-2)(B_1\M\nu+0-B_2\M\nu-1).
    \]
    By lemma~\ref{lem:FTLeadingOrder} with $N=4$, $\delta_3(\la)=\phi(\la)+\psi(\la)$, with $\psi(\la)=\bigoh{\la^{-5}}$ as $\la\to\infty$ within $\clos(D)$, and $\phi$ having holomorphic extension to the appropriate semistrips and $\bigoh{1}$ there.
    Multiplication by $(a_\pm-\nu_\pm(\alpha^j\la)^2)/\Delta(\la)$ gives the right decay of $\psi$ and $\phi$.
    Lemma~\ref{lem:generalNu} implies that multiplication by $\Msup\nu\pm j\prime$ does not affect the asymptotic result.
    The same arguments can be made for $\nu_\pm(\alpha^j\la)\delta_2(\la)\Msup\nu\pm j\prime/\Delta(\la)$ and $\delta_1(\la)\Msup\nu\pm j\prime/\Delta(\la)$.
    To obtain the final leading order asymptotic expression, we use the guarantee that $\nu(\la)=\la+\lindecayla$ from lemma~\ref{lem:generalNu} in the above construction with the explicit expression for $\phi(\la)$ given in lemma~\ref{lem:FTLeadingOrder}.
\end{proof}

Na\"\i{}vely differentiating under the integral in $x$ and $t$, the form of the exponential factor immediately guarantees that~\eqref{eqn:generalproblem.PDE} holds.
However, because some of the integrals as expressed in formulae~\eqref{eqn:mismatchSolutionRewrite} are not uniformly convergent, nor even convergent, if the integrand were multiplied by $\la^3$, we must first reexpress the integrals in forms that are amenable to differentiation.
In the following proof, we use lemmata~\ref{lem:FTLeadingOrder} and~\ref{lem:mismatchConvergenceZetaLeadingOrder} to deform the contour of integration for the slowly decaying ``$\phi$'' part of the integrand away from $\RR$ and $\partial D$, so that all of the resulting integrals are uniformly convergent.
Then we can differentiate under the integral in $x$ and $t$ to see that the PDE holds.
The contour deformation and uniform convergence results are contained in lemmata~\ref{lem:ConvergenceRealIntegrals} and~\ref{lem:ConvergenceD3Integrals}, whose criteria are that $\widehat U_\pm$ and $\M\zeta ej/\Delta$ may each be expressed in the $\phi+\psi$ form discussed above.

\begin{proof}[Proof that each $u_e$ satisfies its PDE]
    By lemma~\ref{lem:mismatchConvergenceZetaLeadingOrder} for $\M\zeta-3(\la;\BVec v)/\Delta(\la)$ and the fact that, as $\la\to\infty$ along $\partial D$, $\abs{\re^{\ri\nu_-(\la) x-\ri\la^3t}}\leq\re^{-\ri\abs\la x/2}\to0$ exponentially, the second integral on the right of equation~\eqref{eqn:mismatchSolutionRewrite.-} is absolutely convergent, uniformly in $(x,t)\in(-\infty,x_0]\times[0,\infty)$, for any $x_0<0$.
    Moreover, for any $\kappa$ continuous on $\partial D$ having at most polynomial growth as $\la\to\infty$ along $\partial D$, we may multiply the integrand by $\kappa(\la)$, and still get uniform absolute convergence.
    Hence, in the second integral on the right of equation~\eqref{eqn:mismatchSolutionRewrite.-}, we may differentiate under the integral in both $x$ and $t$ as many times as we desire without sacrificing the stated uniform convergence.
    It follows from the definition of $\nu_-$ that this integral satisfies~\eqref{eqn:generalproblem.PDE}.

    By lemma~\ref{lem:FTLeadingOrder}, we may represent $\widehat U_-$ in a form convenient for analysis with lemma~\ref{lem:ConvergenceRealIntegrals}.
    The latter then implies that we can reexpress the first integral on the right of equation~\eqref{eqn:mismatchSolutionRewrite.-} using equation~\eqref{eqn:ConvergenceRealIntegrals.integral} and that we may differentiate this representation under the integral up to $3$ times in $x$ or up to once in $t$.
    Hence the first integral on the right of equation~\eqref{eqn:mismatchSolutionRewrite.-} also satisfies~\eqref{eqn:generalproblem.PDE}.
    Similarly, the first integral on the right of equation~\eqref{eqn:mismatchSolution.+} satisfies~\eqref{eqn:generalproblem.PDE}.

    Similarly, lemma~\ref{lem:mismatchConvergenceZetaLeadingOrder} implies that the second integral on the right of equation~\eqref{eqn:mismatchSolution.+} can be analysed via lemma~\ref{lem:ConvergenceD3Integrals}, to see that this integral also obeys~\eqref{eqn:generalproblem.PDE}.
\end{proof}

Next we aim to show that the solution formulae~\eqref{eqn:mismatchSolutionRewrite} satisfy~\eqref{eqn:generalproblem.IC}.
The bulk of the argument in the proof below is to justify that the $t\to0^+$ limit is continuous for all the relevant integrals.
The contour deformation lemmata~\ref{lem:ConvergenceRealIntegrals} and~\ref{lem:ConvergenceD3Integrals} that were used above fail here, because they only give uniform convergence of the integrals for $t\in[\tau,\tau']$ for $\tau'>\tau>0$; they do not facilitate a $t\to0^+$ limit.
We appeal instead to lemma~\ref{lem:tLimitMainLemma}, to show continuity of the integrals.

\begin{proof}[Proof that each $u_e$ satisfies its initial condition]
    Evaluating at $t=0$ the expressions on the right of equations~\eqref{eqn:mismatchSolutionRewrite}, the integrals about $\partial D$ are of the form~\eqref{eqn:mismatchWithoutTimeD3IntegralsZero}, so evaluate to $0$.
    Therefore, by the Fourier inversion theorem, equation~(\ref{eqn:mismatchSolutionRewrite}.$\pm$) simplifies to $2\pi u_\pm(x,0) = 2\pi U_\pm(x)$, which is~\eqref{eqn:generalproblem.IC}.
    It remains only to check that substituting $t=0$ is equivalent to taking the $t\to0^+$ limit in each of these integrals.

    The first paragraph of the proof that each $u_e$ satisfies its PDE argues that the last integral in equation~\eqref{eqn:mismatchSolutionRewrite.-} converges uniformly in $t\in[0,\infty)$.
    It follows that it is continuous in the limit $t\to0^+$.

    We express the first integral on the right of each of equations~\eqref{eqn:mismatchSolutionRewrite} as
    \begin{equation} \label{eqn:mismatchExistence.ICproof.Real}
        \left\{ \int_{-\infty}^{-B} + \int_{-B}^{B} + \int_{B}^{\infty} \right\} \re^{\ri\la x - \ri(\la^3-a_-\la)t} \widehat U_\pm(\la) \D\la,
    \end{equation}
    for $B$ as defined in lemma~\ref{lem:tLimitMainLemma}.
    The second integral is continuous in $t$ because the integrand is uniformly bounded in $t$ and it is a finite integral.
    The first integral can be reexpressed as
    \[
        \int_{B}^{\infty} \re^{\ri\la (-x) - \ri(\la^3-a_-\la)(-t)} \widehat U_-(-\la) \D\la.
    \]
    By lemma~\ref{lem:FTLeadingOrder},
    \[
        \widehat U_\pm(\la)
        =
        \frac1{\ri\la} \sum_{\xi\in X_\pm} \re^{-\ri\la\xi} [U_\pm(\xi^+)-U_\pm(\xi^-)] + \psi(\la),
    \]
    where $\psi(\la)=\bigoh{\la^{-2}}$, and $X_\pm$ is the partition set of $U_\pm$.
    Therefore, the first integral of expression~\eqref{eqn:mismatchExistence.ICproof.Real} can be represented as
    \[
        -\ri \sum_{\xi\in X_\pm} [U_\pm(\xi^+)-U_\pm(\xi^-)] \int_B^\infty \re^{\ri\la (-x-\xi) - \ri(\la^3-a_-\la)(-t)} \frac1\la \D\la + \int_B^\infty \re^{\ri\la (-x) - \ri(\la^3-a_-\la)(-t)} \psi(-\la) \D\la.
    \]
    The final integral above is absolutely convergent uniformly in $t\in\RR$, so is continuous in $t$.
    Each of the integrals in the sum is continuous in the limit $t\to0$ by lemma~\ref{lem:tLimitMainLemma}.
    Similarly, the third integral of expression~\eqref{eqn:mismatchExistence.ICproof.Real} is continuous in the limit $t\to0$.

    For the second integral in equation~\eqref{eqn:mismatchSolutionRewrite.+}, the proof mirrors the relevant section of the proof that it satisfies the PDE: some parts are continuous in $t$ because they are absolutely uniformly convergent, and the rest because of the second statement of lemma~\ref{lem:tLimitMainLemma}.
\end{proof}

Before giving the full proof that $u_\pm$ defined by~\eqref{eqn:mismatchSolution} satisfy~\eqref{eqn:mismatchproblem.VC}, we give a proof of the zeroth order vertex condition.
The aim of this is to present the argument in the simplest case, in which the technicality of the analytic considerations may be greatly reduced, so that the algebraic structure can be emphasized.
With this aim in mind, we do not give a detailed justification of the analytic claims in the below argument.
The generalization one might guess to the higher order vertex conditions turns out to be false; the correct generalization is given in the full proof, which follows this one.

\begin{proof}[Proof that the $u_e$ satisfy the zeroth order vertex condition]
    We aim to show that
    \begin{equation} \label{eqn:mismatchVertexSimple.toProve}
        B_0 \lim_{x\to0^-} \left[ 2\pi u_-(x,t) \right] - \lim_{x\to0^+} \left[ 2\pi u_+(x,t) \right]=0,
    \end{equation}
    in which the functions $2\pi u_\pm(x,t)$ are defined according to formulae~\eqref{eqn:mismatchSolution}.
    Using lemmata~\ref{lem:ConvergenceRealIntegrals} and~\ref{lem:ConvergenceD3Integrals} to partially deform, take the limit, then deform back, it can be shown that formulae~\eqref{eqn:mismatchSolution} are continuous in the relevant limits $x\to0^\pm$, so we may evaluate the limits by simply setting $x=0$.

    We change variables $\la\mapsto\nu_\pm(\la)$ in the first integrals of equations~\eqref{eqn:mismatchSolution}, so that the new contours are $\Gamma_\pm$, the concatenation $(-\infty,-R)\cup\{\re^{\mp\ri(\pi-\theta)}:\theta\in[0,\pi]\}\cup(R,\infty)$ depicted in figure~\ref{fig:Gammapm}.
    The arc corresponding to the integrand involving $\widehat U_\pm(\nu_\pm(\la))$ lies in $\CC^\mp$.
    For the region
    \[
        E = \{ \la\in\CC: \abs\la>R,\; \tfrac\pi3<\arg(\la)\tfrac{2\pi}3 \},
    \]
    appearing in figure~\ref{fig:vertexRewritingRealMinus}, because we already set $x=0$, Jordan's lemma may be used to justify that
    \[
        \int_{\partial E} \re^{-\ri\la^3t} \widehat U_-(\nu_-(\la)) \nu'_-(\la) \D\la = 0.
    \]
    After subtracting this zero integral, we have shown that the $x\to0^-$ limit of the first integral on the right of~\eqref{eqn:mismatchSolution.-} is equal to
    \[
        \int_{\partial (\alpha D) \cup \partial (\alpha^2 D)} \re^{-\ri\la^3t} \widehat U_-(\nu_-(\la)) \nu'_-(\la) \D\la.
    \]
    We change variables $\la\mapsto\alpha\la$ and $\la\mapsto\alpha^2\la$ to map each of $\partial (\alpha D)$ and $ \partial (\alpha^2 D)$ to $\partial D$.
    Similarly,
    \[
        \int_{\partial (\alpha E) \cup (\alpha^2 E)} \re^{-\ri\la^3t} \widehat U_+(\nu_+(\la)) \nu'_+(\la) \D\la = 0,
    \]
    so the contour arising from the first integral of equations~\eqref{eqn:mismatchSolution.+} may be deformed to $\partial D$, negating the integrand to ensure the orientation is respected.
    We have shown that the left side of equation~\eqref{eqn:mismatchVertexSimple.toProve} is equal to
    \begin{multline} \label{eqn:mismatchVertexSimple.allD3}
        \int_{\partial D}\re^{-\ri\la^3t} \bigg\{
            B_0 \left[ \widehat U_-(\M\nu-1) \alpha \Msup\nu-1\prime + \widehat U_-(\M\nu-2) \alpha^2 \Msup\nu-2\prime - \frac{\delta_1+\M\nu-0\delta_2 + (a_--\Msup\nu-02)\delta_3}{\Delta}\Msup\nu-0\prime \right]
        \\
            + \bigg[ \widehat U_+(\M\nu+0)\Msup\nu+0\prime + \frac1\Delta\Big( B_2\{\alpha\Msup\nu+1\prime+\alpha^2\Msup\nu+2\prime\}\delta_1 + B_1\{\alpha\M\nu+1\Msup\nu+1\prime+\alpha^2\M\nu+2\Msup\nu+2\prime\}\delta_2
        \\
            + B_0\{\alpha(a_+-\Msup\nu+12)\Msup\nu+1\prime+\alpha^2(a_+-\Msup\nu+22)\Msup\nu+2\prime\}\delta_3 \Big) \bigg]
        \bigg\} \D\la.
    \end{multline}

    Within the second bracket of expression~\eqref{eqn:mismatchVertexSimple.allD3}, we use the identities in lemma~\ref{lem:generalNu} to evaluate the braced quantities.
    Indeed, the first braced quantity evaluates to $-\Msup\nu+0\prime$, the second to $-\M\nu+0\Msup\nu+0\prime$, and the third to $-(a_+-\Msup\nu+02)\Msup\nu+0\prime-3\la^2$.
    Therefore, the second bracket of expression~\eqref{eqn:mismatchVertexSimple.allD3} simplifies to
    \begin{equation} \label{eqn:mismatchVertexSimple.secondBracket}
        \left[ \widehat U_+(\M\nu+0) - \left( B_2 \frac{\delta_1}{\Delta} + B_1\M\nu+0 \frac{\delta_2}{\Delta} + B_0 (a_+-\Msup\nu+02) \frac{\delta_3}{\Delta} \right) \right] \Msup\nu+0\prime
        - 3B_0\la^2\frac{\delta_3(\la)}{\Delta(\la)}.
    \end{equation}
    By definition, for $\la\in\partial D$,
    \begin{equation} \label{eqn:mismatchVertexSimple.matrix}
        \begin{pmatrix}
            1 & \M\nu-{2} & \left(a_--\Msup\nu-22\right) \\
            1 & \M\nu-{1} & \left(a_--\Msup\nu-12\right) \\
            B_2 & B_1\M\nu+0 & B_0\left(a_+-\Msup\nu+02\right)
        \end{pmatrix}
        \begin{pmatrix}
            \delta_1 / \Delta \\ \delta_2 / \Delta \\ \delta_3 / \Delta
        \end{pmatrix}
        =
        \begin{pmatrix}
            -\widehat U_-(\M\nu-2) \\ -\widehat U_-(\M\nu-1) \\ \widehat U_+(\M\nu+0)
        \end{pmatrix},
    \end{equation}
    the third line of which establishes that the bracket in expression~\eqref{eqn:mismatchVertexSimple.secondBracket} evaluates to $0$.

    Multiplying the first line of equation~\eqref{eqn:mismatchVertexSimple.matrix} by $-\alpha^2\Msup\nu-2\prime$ and the second by $-\alpha\Msup\nu-1\prime$ and summing the results, we obtain
    \begin{align*}
        \widehat U_-(\M\nu-1) \alpha \Msup\nu-1\prime + \widehat U_-(\M\nu-2) \alpha^2 \Msup\nu-2\prime
        &=
        - \{\alpha \Msup\nu-1\prime + \alpha^2 \Msup\nu-2\prime\} \frac{\delta_1}\Delta
        - \{\alpha \M\nu-1\Msup\nu-1\prime + \alpha^2 \M\nu-2\Msup\nu-2\prime\} \frac{\delta_2}\Delta \\
        &\hspace{6em} - \{\alpha (a_--\Msup\nu-12)\Msup\nu-1\prime + \alpha^2 (a_--\Msup\nu-22)\Msup\nu-2\prime\} \frac{\delta_3}\Delta \\
        &= \left[ \frac{\delta_1}\Delta + \M\nu-0 \frac{\delta_2}\Delta + (a_--\Msup\nu-02) \frac{\delta_3}\Delta \right] \Msup\nu-0\prime + 3\la^2 \frac{\delta_3(\la)}{\Delta(\la)},
    \end{align*}
    where the second equality again uses the identities from lemma~\ref{lem:generalNu}.
    Substituting this in the first bracket in expression~\eqref{eqn:mismatchVertexSimple.allD3}, we find that that bracket also almost completely cancels.
    Indeed, expression~\eqref{eqn:mismatchVertexSimple.allD3} has now simplified to
    \[
        \int_{\partial D}\re^{-\ri\la^3t} \left\{
            B_0 \left[ 3\la^2 \frac{\delta_3(\la)}{\Delta(\la)} \right]
            + \left[ - 3B_0\la^2\frac{\delta_3(\la)}{\Delta(\la)} \right]
        \right\} \D\la
        =
        \int_{\partial D}\re^{-\ri\la^3t} \left\{ 0 \right\} \D\la
        =
        0.
        \qedhere
    \]
\end{proof}

In the above argument, we first provided an analytic justification that the claim was equivalent to expression~\eqref{eqn:mismatchVertexSimple.allD3} evaluating to $0$, then used an algebraic argument to prove that.
For the first order vertex condition, bearing in mind that $x\to0^\pm$ limits have been taken of functions having the simple $x$ dependence $\exp(\ri\M\nu\pm j x)$, one might hope that expression~\eqref{eqn:mismatchVertexSimple.allD3} may be simply modified by replacing the first $B_0$ with $B_1$ and multiplying terms by $\ri\M\nu\pm j$ as appropriate.
Unfortunately, such an integral would diverge, so the $x\to0^\pm$ limits of the first order boundary values necessarily have a more complicated expression.
Instead one must separate out the leading order terms in the integrands and deform them to separate parallel contours before taking the $x$ derivative, but then the rest of the analytic argument is essentially valid.
We derive instead expression~\eqref{eqn:mismatchVertexFull.allD3Gamma3}, which is similar to~\eqref{eqn:mismatchVertexSimple.allD3}, but with two contour integrals, one on a contour parallel to $\partial D$ but a little outside $D$, whose integrand is the leading order part, and one on $\partial D$ itself, whose integrand is the faster decaying remainder.
Because the integrands are each expressions similar to the integrand in~\eqref{eqn:mismatchVertexSimple.allD3}, we can still use a similar argument to the algebraic argument in the latter part of the above proof, albeit applying the algebraic argument separately to each of the two integrands.

\begin{proof}[Full proof that the $u_e$ satisfy the vertex conditions]
    We aim to show that, for $k\in\{0,1,2\}$.
    \begin{equation} \label{eqn:mismatchVertexFull.toProve}
        (-\ri)^kB_k \lim_{x\to0^-} 2\pi \partial_x^k u_-(x,t) - (-\ri)^k\lim_{x\to0^+} 2\pi \partial_x^k u_+(x,t) = 0.
    \end{equation}

    We change variables in the first integral on the right of~\eqref{eqn:mismatchSolution.-}, to yield
    \[
        \int_{\Gamma_-} \re^{\ri\nu_-(\la)x-\ri\la^3t}\widehat U_-(\nu_-(\la)) \nu'_-(\la) \D\la,
    \]
    where $\Gamma_-$ is the contour that consists of the positively oriented real line except that the part from $-R$ to $R$ is deformed along a semicircular arc in $\CC^+$; see figure~\ref{fig:Gammapm}.
    By lemma~\ref{lem:vertexRewritingLemma.Real-},
    \begin{multline*} %\label{eqn:vertexRewritingLemma.Real-.Step3}
        (-\ri)^kB_k \lim_{x\to0^-} \partial_x^k \left[\begin{matrix}\text{first term on right}\\\text{of equation~\eqref{eqn:mismatchSolution.-}}\end{matrix}\right]
        \\
        =
        B_k \,\PV\,\left[
            \int_{\partial D} \re^{-\ri\la^3t} \left[ \alpha \Msup\nu-1k\Msup\nu-1\prime \M\psi\RR-(\alpha\la) + \alpha^2 \Msup\nu-2k\Msup\nu-2\prime \M\psi\RR-(\alpha^2\la) \right] \D\la
            \right. \\ \left.
            + \int_{\Gamma} \re^{-\ri\la^3t} \left[ \alpha \Msup\nu-1k\Msup\nu-1\prime \M\phi\RR-(\alpha\la) + \alpha^2 \Msup\nu-2k\Msup\nu-2\prime \M\phi\RR-(\alpha^2\la) \right] \D\la
        \right],
    \end{multline*}
    where
    \[
        \widehat U_-(\nu_-(\la)) = \left[\M\phi\RR-(\la) + \M\psi\RR-(\la)\right]\nu'_-(\la),
    \]
    and $\M\phi\RR-(\la)$ is the leading order term $\phi(\nu_-(\la))$ from lemma~\ref{lem:FTLeadingOrder}, with $N=4$.
    Note that the pair $(\phi,\psi)$ from lemma~\ref{lem:vertexRewritingLemma.Real-} is not exactly $(\M\phi\RR-,\M\psi\RR-)$ but rather $(\M\phi\RR-\Msup\nu-0\prime,\M\psi\RR-\Msup\nu-0\prime)$; this is inconsequential for their asymptotic and analytic properties, because lemma~\ref{lem:generalNu} gurantees that $\Msup\nu-0\prime$ is analytic outside $B(0,R)$ and is $1+\lindecayla$ as $\la\to\infty$ in any direction.
    The mismatch between notations affords some efficiency of presentation later in this proof.

    We rewrite the first integral of~\eqref{eqn:mismatchSolution.+}, as
    \[
        \int_{\Gamma_+} \re^{\ri\nu_+(\la)x-\ri\la^3t}\widehat U_+(\nu_+(\la)) \nu'_+(\la) \D\la,
    \]
    where $\Gamma_+$ is the contour that consists of the positively oriented real line except that the part from $-R$ to $R$ is deformed along a semicircular arc in $\CC^-$, also shown in figure~\ref{fig:Gammapm}.
    Now by lemma~\ref{lem:vertexRewritingLemma.Real+},
    \begin{multline*} %\label{eqn:vertexRewritingLemma.Real+.Step3}
        -(-\ri)^k\lim_{x\to0^+} \partial_x^k \left[\begin{matrix}\text{first term on right}\\\text{of equation~\eqref{eqn:mismatchSolution.+}}\end{matrix}\right]
        \\
        =
        \PV\,\left[
            \int_{\partial D} \re^{-\ri\la^3t} \Msup\nu+0k\Msup\nu+0\prime \M\psi\RR+(\la) \D\la
            + \int_{\Gamma} \re^{-\ri\la^3t} \Msup\nu+0k\Msup\nu+0\prime \M\phi\RR+(\la) \D\la
        \right],
    \end{multline*}
    where
    \[
        \widehat U_+(\nu_+(\la)) = \left[\M\phi\RR+(\la) + \M\psi\RR+(\la)\right]\nu'_+(\la),
    \]
    and $\M\phi\RR+(\la)$ is the leading order term $\phi(\nu_+(\la))$ from lemma~\ref{lem:FTLeadingOrder}, with $N=4$.

    Using lemma~\ref{lem:vertexRewritingLemma.D3},
    \begin{multline*} %\label{eqn:vertexRewritingLemma.Real+.Step3}
        (-\ri)^kB_k \lim_{x\to0^-} \partial_x^k \left[\begin{matrix}\text{second term on right}\\\text{of equation~\eqref{eqn:mismatchSolution.-}}\end{matrix}\right]
        - (-\ri)^k\lim_{x\to0^+} \partial_x^k \left[\begin{matrix}\text{second term on right}\\\text{of equation~\eqref{eqn:mismatchSolution.+}}\end{matrix}\right]
        \\
        =
        \PV\,\left[
            \int_{\partial D} \re^{-\ri\la^3t} \left[ \alpha \Msup\nu+1k\Msup\nu+1\prime \M\psi1+(\la) + \alpha^2 \Msup\nu+2k\Msup\nu+2\prime \M\psi2+(\la) - B_k \Msup\nu-0k\Msup\nu-0\prime \M\psi3-(\la) \right] \D\la
            \right. \\ \left.
            + \int_{\Gamma} \re^{-\ri\la^3t} \left[ \alpha \Msup\nu+1k\Msup\nu+1\prime \M\phi1+(\la) + \alpha^2 \Msup\nu+2k\Msup\nu+2\prime \M\phi2+(\la) - B_k \Msup\nu-0k\Msup\nu-0\prime \M\phi3-(\la) \right] \D\la
        \right],
    \end{multline*}
    where $\Gamma$ is defined by equation~\eqref{eqn:defn.Gamma}.
    \begin{subequations} \label{eqn:mismatchVertexFull.defnphipsijpm}
    \begin{align}
        \label{eqn:mismatchVertexFull.defnphipsijpm.firstbit}
        \M\omega3-(\la) &= \Omega_1 + \M\nu-0\Omega_2 + (a_--\Msup\nu-02)\Omega_3,
        \\
        \label{eqn:mismatchVertexFull.defnphipsij+}
        \M\omega j+(\la) &= B_2\Omega_1 + B_1\M\nu+j\Omega_2 + B_0(a_+-\Msup\nu+j2)\Omega_3
        & j &\in \{1,2\},
    \end{align}
    for
    \begin{equation}
        \label{eqn:mismatchVertexFull.matrixPhi}
        \mathcal A
        \begin{pmatrix}
            \Omega_1 \\ \Omega_2 \\ \Omega_3
        \end{pmatrix}
        =
        \begin{pmatrix}
            -\M\omega\RR-(\alpha^2\la) \\ -\M\omega\RR-(\alpha\la) \\ \M\omega\RR+(\la)
        \end{pmatrix},
    \end{equation}
    \end{subequations}
    in which the symbols $\omega,\Omega$ represent $\phi,\Phi$ or $\psi,\Psi$, respectively.
    Indeed, by lemma~\ref{lem:FTLeadingOrder}, $\M\phi\RR\pm=\lindecayla$ as $\la\to\infty$ within
    \[
        \{\la-\ri\epsilon:\la\in\clos(\CC^+)\} \mbox{ and } \{\la+\ri\epsilon:\la\in\clos(\CC^-)\}
    \]
    respectively, which justifies the use of lemmata~\ref{lem:vertexRewritingLemma.Real-} and~\ref{lem:vertexRewritingLemma.Real+} above.
    But these decay results combined with construction~\eqref{eqn:mismatchVertexFull.defnphipsijpm} also imply that $\M\phi j\pm$ and $\M\psi j\pm$ obey the bounds and analyticity criteria for the required applications of lemma~\ref{lem:vertexRewritingLemma.D3}.
    We have shown that the left side of equation~\eqref{eqn:mismatchVertexFull.toProve} evaluates to
    \begin{multline} \label{eqn:mismatchVertexFull.allD3Gamma3}
        \PV\,\left\{ \int_{\partial D} + \int_{\Gamma} \right\} \re^{-\ri\la^3t}
        \bigg[
            B_k \alpha \Msup\nu-1k\Msup\nu-1\prime \M\omega\RR-(\alpha\la)
            + B_k \alpha^2 \Msup\nu-2k\Msup\nu-2\prime \M\omega\RR-(\alpha^2\la)
            + \Msup\nu+0k\Msup\nu+0\prime \M\omega\RR+(\la)
            \\
            + \alpha \Msup\nu+1k\Msup\nu+1\prime \M\omega1+(\la)
            + \alpha^2 \Msup\nu+2k\Msup\nu+2\prime \M\omega2+(\la)
            - B_k \Msup\nu-0k\Msup\nu-0\prime \M\omega3-(\la)
        \bigg]
        \D\la,
    \end{multline}
    in which the symbol $\omega$ refers to $\psi$ for the integral on $\partial D$ and $\phi$ for the integral on $\Gamma$.

    Using construction~\eqref{eqn:mismatchVertexFull.defnphipsijpm}, after some simplification, we find that the bracket in expression~\eqref{eqn:mismatchVertexFull.allD3Gamma3} evaluates to
    \begin{equation*}
        \begin{pmatrix}
            \displaystyle B_2 \sum_{j=0}^2 \alpha^j\Msup\nu+jk\Msup\nu+j\prime
            - B_k \sum_{j=0}^2 \alpha^j\Msup\nu-jk\Msup\nu-j\prime \\
            \displaystyle B_1 \sum_{j=0}^2 \alpha^j\Msup\nu+j{k+1}\Msup\nu+j\prime
            - B_k \sum_{j=0}^2 \alpha^j\Msup\nu-j{k+1}\Msup\nu-j\prime \\
            \displaystyle B_0 \sum_{j=0}^2 \alpha^j\Msup\nu+jk\left(a_+-\Msup\nu+j2\right)\Msup\nu+j\prime
            - B_k \sum_{j=0}^2 \alpha^j\Msup\nu-jk\left(a_--\Msup\nu-j2\right)\Msup\nu-j\prime
        \end{pmatrix}
        \cdot
        \begin{pmatrix}
            \Omega_1 \\ \Omega_2 \\ \Omega_3
        \end{pmatrix}.
    \end{equation*}
    The identities for $\nu_\pm$ given in lemma~\ref{lem:generalNu} imply that every entry in the first vector of the above dot product evalates to $0$.
    Therefore, both integrands in expression~\eqref{eqn:mismatchVertexFull.allD3Gamma3} are zero, and equation~\eqref{eqn:mismatchVertexFull.toProve} holds.
\end{proof}

\section{Loop defect} \label{sec:loop}
Here we study a metric graph in which there is a single bond attached as a loop of length $\eta>0$ to the vertex joining the leads, as displayed in figure~\ref{fig:graph-loop}.
\begin{figure}
    \centering
    \includegraphics{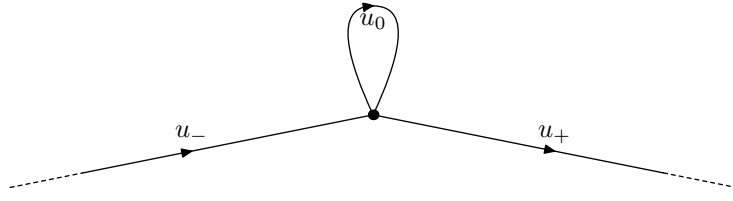}
    \caption{The metric graph domain with a loop defect.}
    \label{fig:graph-loop}
\end{figure}
Expecting that the new bond should require three extra vertex conditions compared with the simpler problem studied in~\S\ref{sec:mismatch}, we impose a total of six vertex conditions.
We select for their algebraic simplicity the vertex conditions which (mis-)match both amplitudes and fluxes:
\begin{equation} \label{eqn:loopVC}
    \partial_x^k u_-(0,t) = b_k^{-1} \partial_x^k u_0(0,t) = \beta_k^{-1} \partial_x^k u_0(\eta,t) = B_k^{-1} \partial_x^k u_+(0,t), \qquad t \in [0,T], \qquad k\in\{0,1\},
\end{equation}
with coefficients $b_k,\beta_k,B_k\in\CC\setminus\{0\}$.

\begin{prop} \label{prop:loopSolution}
    Consider problem~\eqref{eqn:generalproblem} on the metric graph described above, with vertex conditions~\eqref{eqn:loopVC} for coefficients satisfying, for both $j\in\{1,2\}$, $\abs{\beta_1-\beta_0}>\abs{\alpha^j b_1-b_0}$.
    Define, for $e\in\{+,-,0\}$ and $j\in\{0,1,2\}$, $\M\nu ej = \nu_e(\alpha^j\la)$, $E_j = \exp\left(-\ri\eta\M\nu0j\right)$, and
    \begin{equation}  \label{eqn:loopA}
        \mathcal A(\la) =
        \begin{pmatrix}
            1 & 0 & 0 & 0 &  \M\nu-2 &  (a_--\Msup\nu-22) \\
            1 & 0 & 0 & 0 &  \M\nu-1 &  (a_--\Msup\nu-12) \\
            0 & 1 & 0 & 0 & -\M\nu+0B_1 & -(a_+-\Msup\nu+02)B_0 \\
            0 & 0 & 1 & E_2 & \M\nu02\left(E_2\beta_1-b_1\right) & (a_0-\Msup\nu022)\left(E_2\beta_0-b_0\right) \\
            0 & 0 & 1 & E_1 & \M\nu01\left(E_1\beta_1-b_1\right) & (a_0-\Msup\nu012)\left(E_1\beta_0-b_0\right) \\
            0 & 0 & 1 & E_0 & \M\nu00\left(E_0\beta_1-b_1\right) & (a_0-\Msup\nu002)\left(E_0\beta_0-b_0\right)
        \end{pmatrix}.
    \end{equation}
    Let $\Delta=\det\mathcal A$, and $\delta_k$ be the determinant of $\mathcal A$ but with the $k$\textsuperscript{th} column replaced by
    \begin{equation} \label{eqn:loopSolnY}
        -\left(
            \widehat U_-(\M\nu-2), \;
            \widehat U_-(\M\nu-1), \;
            \widehat U_+(\M\nu-0), \;
            \widehat U_0(\M\nu02), \;
            \widehat U_0(\M\nu01), \;
            \widehat U_0(\M\nu00)
        \right).
    \end{equation}
    Suppose $u$ is sufficiently smooth and satisfies this problem.
    Then, for $R$ sufficiently large,
    \begin{subequations} \label{eqn:loopSoln}
    \begin{align}
    \notag
        2\pi u_-(x,t)
        &= \int_{-\infty}^\infty \re^{\ri\la x - \ri(\la^3-a_-\la)t} \widehat U_-(\la) \D \la \\
    \label{eqn:loopSoln.-} \tag{\theparentequation.\(-\)}
        &\hspace{1.2em} - \int_{\partial D} \re^{\ri\nu_-(\la)x-\ri\la^3t} \left[ \frac{\delta_1(\la) + \nu_-(\la)\delta_5(\la) + (a_--\nu_-(\la)^2)\delta_6(\la)}{\Delta(\la)} \right] \nu'_-(\la) \D\la, \displaybreak[0] \\
    \notag
        2\pi u_+(x,t)
        &= \int_{-\infty}^\infty \re^{\ri\la x - \ri(\la^3-a_+\la)t} \widehat U_+(\la) \D \la \\
    \notag
        &\hspace{3em} - \int_{\partial D} \frac{\re^{-\ri\la^3t}}{\Delta(\la)} \Bigg[
            -\delta_2(\la) \left( \alpha \re^{\ri\nu_+(\alpha\la)x} \nu'_+(\alpha\la) + \alpha^2 \re^{\ri\nu_+(\alpha^2\la)x} \nu'_+(\alpha^2\la) \right) \\
    \notag
            &\hspace{6em} + B_1\delta_5(\la) \left( \alpha \nu_+(\alpha\la)\re^{\ri\nu_+(\alpha\la)x} \nu'_+(\alpha\la) + \alpha^2 \nu_+(\alpha^2\la) \re^{\ri\nu_+(\alpha^2\la)x} \nu'_+(\alpha^2\la) \right) \\
    \label{eqn:loopSoln.+} \tag{\theparentequation.\(+\)}
            &\hspace{-4em} + B_0\delta_6(\la) \left( \alpha \left(a_+-\nu_+(\alpha\la)^2\right) \re^{\ri\nu_+(\alpha\la)x} \nu'_+(\alpha\la) + \alpha^2 \left(a_+-\nu_+(\alpha^2\la)^2\right) \re^{\ri\nu_+(\alpha^2\la)x} \nu'_+(\alpha^2\la) \right)
        \Bigg] \D\la, \displaybreak[0] \\
    \notag
        2\pi u_0(x,t)
        &= \int_{-\infty}^\infty \re^{\ri\la x - \ri(\la^3-a_0\la)t} \widehat U_0(\la) \D \la \\
    \notag
        &\hspace{3em} - \int_{\partial D} \frac{\re^{-\ri\la^3t}}{\Delta(\la)} \Bigg[
            -\delta_3(\la) \left( \alpha \re^{\ri\nu_0(\alpha\la)x} \nu'_0(\alpha\la) + \alpha^2 \re^{\ri\nu_0(\alpha^2\la)x} \nu'_0(\alpha^2\la) \right) \\
    \notag
            &\hspace{6em} + b_1\delta_5(\la) \left( \alpha \nu_0(\alpha\la)\re^{\ri\nu_0(\alpha\la)x} \nu'_0(\alpha\la) + \alpha^2 \nu_0(\alpha^2\la) \re^{\ri\nu_0(\alpha^2\la)x} \nu'_0(\alpha^2\la) \right) \\
    \notag
            &\hspace{-2em} + b_0\delta_6(\la) \left( \alpha \left(a_0-\nu_0(\alpha\la)^2\right) \re^{\ri\nu_0(\alpha\la)x} \nu'_0(\alpha\la) + \alpha^2 \left(a_0-\nu_0(\alpha^2\la)^2\right) \re^{\ri\nu_0(\alpha^2\la)x} \nu'_0(\alpha^2\la) \right)
        \Bigg] \D\la \\
    \label{eqn:loopSoln.0} \tag{\theparentequation.\(0\)}
        &\hspace{1.2em} - \int_{\partial D} \re^{\ri\nu_0(\la)(x-\eta)-\ri\la^3t} \left[ \frac{\delta_4(\la) + \nu_0(\la)\beta_1\delta_5(\la) + (a_0-\nu_0(\la)^2)\beta_0\delta_6(\la)}{\Delta(\la)} \right] \nu'_0(\la) \D\la.
    \end{align}
    \end{subequations}
\end{prop}

\begin{prop} \label{prop:loopExistence}
    Assume $U_e\in\AC^4_{\mathrm p}(\Omega_e)$, per definition~\ref{defn:piecewiseACfunctions}.
    For all $x\in\Omega_e$ and $t\geq0$, let $u_e(x,t)$ be defined by equations~\eqref{eqn:loopSoln}.
    Then the functions $u_-,u_+,u_0$ satisfy problem~\eqref{eqn:generalproblem},~\eqref{eqn:loopVC}.
\end{prop}

\begin{figure}
    \centering
    \includegraphics[width=0.65\textwidth]{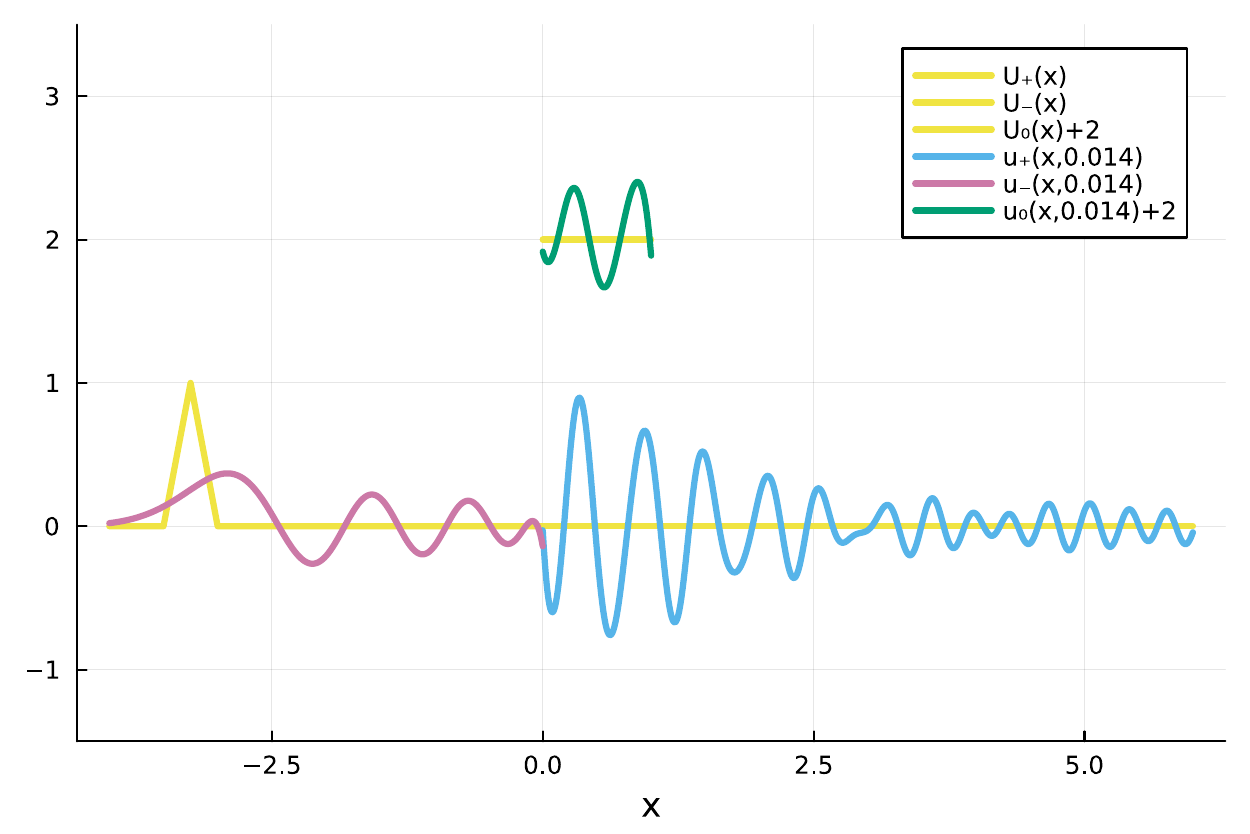}
    \caption{
        The solution of the mismatch problem at time $t=0.014$ with vertex coefficients $(b_0,b_1,\beta_0,\beta_1,B_0,B_1)=(0.6,0.7,0.8,2.1,0.2,3.0)$, with initial data shown in yellow.
        The loop edge of length $\eta=1$ is displayed separately from the leads so that the solution and initial datum are clear on the loop, at the cost of inaccurately representing the graph; c.f.\ figure~\ref{fig:graph-loop}.
    }
    \label{fig:plot-loop-eg}
\end{figure}

For certain parameters and data, the solution of this problem is shown in figure~\ref{fig:plot-loop-eg}.
This plot is not perfectly comparable with that in figure~\ref{fig:plot-mismatch-eg} because the vertex conditions have a different structure, but the systems have in common that a shock wave spreads from initial support on only the incoming lead into the outgoing lead.
% A qualitative difference is that the outgoing wave has its amplitude modulated in figure~\ref{fig:plot-loop-eg}, but not in figure~\ref{fig:plot-mismatch-eg}.
% It appears that this is due to a superposition of something qualitatively like the solution of the simpler problem (reducing amplitude but also, crucially, reducing wavelength, both as $x$ increases) with another wave that has already traversed the loop, hence already has a slightly reduced wavelength for the same $x$.
% Additional superpositions can be expected from successive traversals of the loop defect before entering the outgoing lead.
% While this is a reasonable intuitive interpretation of the behaviour of the system, further quantitative study is required to confirm that this is truly the structure.
A qualitative difference is that the outgoing wave has a more complicated amplitude modulation in figure~\ref{fig:plot-loop-eg}, than that in figure~\ref{fig:plot-mismatch-eg}.
% This may be explained informally as follows.
An informal explanation for this phenomenon follows.
In the simpler mismatch graph, the wave may only proceed from the incoming lead to the outgoing lead directly, whereas in the loop problem the wave can be considered to propogate directly across the interface between the leads, but also via one or more traversals of the loop.
Therefore, around a given point $x=x_0>0$ on the outgoing lead, there is a superposition of multiple copies of the wave, having travelled a distance of $x_0+N$ for $N\in\{0,1,2,\ldots\}$ since first encountering the defect, and these waves must interfere in some way.
Returning to figure~\ref{fig:plot-mismatch-eg}, we see that, as $x>0$ increases, the amplitude of the wave decreases but, crucially, the wavelength also gradually decreases.
Therefore, in the superposition of multiple dispersive waves found close to $x_0$, it is not surprising that interference occurs, and this interference may be constructive or destructive at different $x_0$.

\subsection{Solution representation: proof of proposition~\protect\ref{prop:loopSolution}} \label{ssec:loopUnicity}

Applying vertex conditions~\eqref{eqn:loopVC} to the global relations~\eqref{eqn:generalGRe}, we obtain
\begin{subequations} \label{eqn:loopGR}
\begin{align}
\label{eqn:loopGR.-} \tag{\theparentequation.\(-\)}
    \la &\in \clos(\CC^+), &
    0 &= \re^{\ri\la^3t} \hat u_-(\M\nu-0;t) - \widehat U_-(\M\nu-0) - \left[ \M f-2 + \ri \M\nu-0 \M f-1 + \left( a_- - \M\nu-0^2 \right) \M f-0 \right] \\
\label{eqn:loopGR.+} \tag{\theparentequation.\(+\)}
    \la &\in \clos(\CC^-), &
    0 &= \re^{\ri\la^3t} \hat u_+(\M\nu+0;t) - \widehat U_+(\M\nu+0) + \left[ \M f+2 + \ri \M\nu+0 B_1 \M f-1 + \left( a_+ - \M\nu+0^2 \right) B_0 \M f-0 \right], \\
\notag
    \la &\in \CC, &
    0 &= \re^{\ri\la^3t} \hat u_0(\M\nu00;t) - \widehat U_0(\M\nu00) + \left[ g_2 + \ri \M\nu00 b_1 \M f-1 + \left( a_0 - \M\nu00^2 \right) b_0 \M f-0 \right] \\
\label{eqn:loopGR.0} \tag{\theparentequation.\(0\)}
    &&&\hspace{10em} - E_0\left[ G_2 + \ri \M\nu00 \beta_1 \M f-1 + \left( a_0 - \M\nu00^2 \right) \beta_0 \M f-0 \right],
\end{align}
\end{subequations}
in which
\[
    \M f\pm j = \M f\pm j(\la;0,t), \qquad g_2 = \M f02(\la;0,t), \qquad G_2 = \M f02(\la;\eta,t).
\]
For $\la\in\clos(D)$, we use equations
\begin{align*}
    &\mbox{\eqref{eqn:loopGR.-}}\Big\rvert_{\la\mapsto\alpha^{2}\la},
    &
    &\mbox{\eqref{eqn:loopGR.-}}\Big\rvert_{\la\mapsto\alpha\la},
    &
    &\mbox{\eqref{eqn:loopGR.+}}\Big\rvert_{\la\mapsto\la},
    \\
    &\mbox{\eqref{eqn:loopGR.0}}\Big\rvert_{\la\mapsto\alpha^{2}\la},
    &
    &\mbox{\eqref{eqn:loopGR.0}}\Big\rvert_{\la\mapsto\alpha\la},
    &
    &\mbox{\eqref{eqn:loopGR.0}}\Big\rvert_{\la\mapsto\la}.
\end{align*}
This yields the linear system, for $\la\in\clos(D)$,
\begin{equation} \label{eqn:loopLinearSystem}
    \mathcal A(\la)
    \left( \M f-2 \quad -\M f+2 \quad -g_2 \quad G_2 \quad \ri \M f-1 \quad \M f-0 \right)^\top
    =
    Y(\la),
\end{equation}
where $\mathcal A(\la)$ is defined in equation~\eqref{eqn:loopA}
and
\begin{equation} \label{eqn:loopY}
    Y(\la) =
    \re^{\ri\la^3t}
    \begin{pmatrix}
        \hat u_-(\M\nu-2;t) \\ \hat u_-(\M\nu-1;t) \\ \hat u_+(\M\nu+0;t) \\ \hat u_0(\M\nu02;t) \\ \hat u_0(\M\nu01;t) \\ \hat u_0(\M\nu00;t)
    \end{pmatrix}
    -
    \begin{pmatrix}
        \widehat U_-(\M\nu-2) \\ \widehat U_-(\M\nu-1) \\ \widehat U_+(\M\nu+0) \\ \widehat U_0(\M\nu02) \\ \widehat U_0(\M\nu01) \\ \widehat U_0(\M\nu00)
    \end{pmatrix}.
\end{equation}
After employing Cramer's rule, it is a matter of straightforward row and column operations to establish the following lemma on the solution of linear system~\eqref{eqn:loopLinearSystem}.

\begin{lem} \label{lem:loopLinearSystemSolution}
    Linear system~\eqref{eqn:loopLinearSystem} has solution
    \begin{align*}
        \M f-2 &= \delta_1(\la)/\Delta(\la), &
        -\M f+2 &= \delta_2(\la)/\Delta(\la), &
        -g_2 &= \delta_3(\la)/\Delta(\la), \\
        G_2 &= \delta_4(\la)/\Delta(\la), &
        \ri \M f-1 &= \delta_5(\la)/\Delta(\la), &
        \M f-0 &= \delta_6(\la)/\Delta(\la),
    \end{align*}
    where $\Delta=\det\mathcal{A}$ and $\delta_k$ is the determinant of $\mathcal A$ but with the $k$\textsuperscript{th} column replaced by $Y$ (not necessarily given by equation~\eqref{eqn:loopY}).
    Moreover,
    \begin{equation} \label{eqn:loopDelta}
        \Delta(\la) = \det
        \begin{pmatrix}
            1 & E_2 & \left(\Msup\nu-22-\Msup\nu-12\right)\M\nu02\left(E_2\beta_1-b_1\right) - \left(\M\nu-1-\M\nu-2\right)\left(a_0-\Msup\nu022\right)\left(E_2\beta_0-b_0\right) \\
            1 & E_1 & \left(\Msup\nu-22-\Msup\nu-12\right)\M\nu01\left(E_1\beta_1-b_1\right) - \left(\M\nu-1-\M\nu-2\right)\left(a_0-\Msup\nu012\right)\left(E_1\beta_0-b_0\right) \\
            1 & E_0 & \left(\Msup\nu-22-\Msup\nu-12\right)\M\nu00\left(E_0\beta_1-b_1\right) - \left(\M\nu-1-\M\nu-2\right)\left(a_0-\Msup\nu002\right)\left(E_0\beta_0-b_0\right) \\
        \end{pmatrix}
    \end{equation}
    and
    \begin{align*}
        \delta_1(\la) &=
        \det\begin{pmatrix}
            1 & E_2 & \begin{MLME}{30em}\left[ Y_1(a_--\Msup\nu-12) - Y_2(a_--\Msup\nu-22) \right] \M\nu01(E_2\beta_1-b_1) \\[-12pt] - \left[ Y_1\M\nu-1 - Y_2\M\nu-2 \right](a_0-\Msup\nu022)(E_2\beta_0-b_0)\end{MLME} \\
            1 & E_1 & \begin{MLME}{30em}\left[ Y_1(a_--\Msup\nu-12) - Y_2(a_--\Msup\nu-22) \right] \M\nu01(E_1\beta_1-b_1) \\[-12pt] - \left[ Y_1\M\nu-1 - Y_2\M\nu-2 \right](a_0-\Msup\nu012)(E_1\beta_0-b_0)\end{MLME} \\
            1 & E_0 & \begin{MLME}{30em}\left[ Y_1(a_--\Msup\nu-12) - Y_2(a_--\Msup\nu-22) \right] \M\nu00(E_0\beta_1-b_1) \\[-12pt] - \left[ Y_1\M\nu-1 - Y_2\M\nu-2 \right](a_0-\Msup\nu002)(E_0\beta_0-b_0)\end{MLME}
        \end{pmatrix} \\{}
        &\hspace{5em} - \det\begin{pmatrix} Y_4 & 1 & E_2 \\ Y_5 & 1 & E_1 \\ Y_6 & 1 & E_0 \end{pmatrix} \det\begin{pmatrix} \M\nu-2 & (a_--\Msup\nu-22) \\ \M\nu-1 & (a_--\Msup\nu-12) \end{pmatrix}, \displaybreak[0] \\
        \delta_2(\la) &=
        (Y_2-Y_1)\det\begin{pmatrix}
            1 & E_2 & (a_+-\Msup\nu+02)B_0\M\nu02(E_2\beta_1-b_1) - \M\nu+0B_1(a_0-\Msup\nu022)(E_2\beta_0-b_0) \\
            1 & E_1 & (a_+-\Msup\nu+02)B_0\M\nu01(E_1\beta_1-b_1) - \M\nu+0B_1(a_0-\Msup\nu012)(E_1\beta_0-b_0) \\
            1 & E_0 & (a_+-\Msup\nu+02)B_0\M\nu00(E_0\beta_1-b_1) - \M\nu+0B_1(a_0-\Msup\nu002)(E_0\beta_0-b_0)
        \end{pmatrix} \\
        &\hspace{5em} + Y_3\Delta(\la) + \det\begin{pmatrix} Y_4 & 1 & E_2 \\ Y_5 & 1 & E_1 \\ Y_6 & 1 & E_0 \end{pmatrix} \det\begin{pmatrix} 1 & \M\nu-2 & (a_--\Msup\nu-22) \\ 1 & \M\nu-1 & (a_--\Msup\nu-12) \\ 0 & B_1\M\nu+0 & -B_0(a_+-\Msup\nu+02) \end{pmatrix}, \displaybreak[0] \\
        \delta_3(\la) &=
        (Y_1-Y_2)\det\begin{pmatrix}
            E_2 & \M\nu02(E_2\beta_1-b_1) & (a_0-\Msup\nu022)(E_2\beta_0-b_0) \\
            E_1 & \M\nu01(E_1\beta_1-b_1) & (a_0-\Msup\nu012)(E_1\beta_0-b_0) \\
            E_0 & \M\nu00(E_0\beta_1-b_1) & (a_0-\Msup\nu002)(E_0\beta_0-b_0)
        \end{pmatrix} \\
        &\hspace{5em} + \det\begin{pmatrix}
            Y_4 & E_2 & (\Msup\nu-22-\Msup\nu-12)\M\nu02(E_2\beta_1-b_1) - (\M\nu-1-\M\nu-2)(E_2\beta_0-b_0) \\
            Y_5 & E_1 & (\Msup\nu-22-\Msup\nu-12)\M\nu01(E_1\beta_1-b_1) - (\M\nu-1-\M\nu-2)(E_1\beta_0-b_0) \\
            Y_6 & E_0 & (\Msup\nu-22-\Msup\nu-12)\M\nu00(E_0\beta_1-b_1) - (\M\nu-1-\M\nu-2)(E_0\beta_0-b_0)
        \end{pmatrix}, \displaybreak[0] \\
        \delta_4(\la) &=
        (Y_2-Y_1)\det\begin{pmatrix}
            1 & \M\nu02(E_2\beta_1-b_1) & (a_0-\Msup\nu022)(E_2\beta_0-b_0) \\
            1 & \M\nu01(E_1\beta_1-b_1) & (a_0-\Msup\nu012)(E_1\beta_0-b_0) \\
            1 & \M\nu00(E_0\beta_1-b_1) & (a_0-\Msup\nu002)(E_0\beta_0-b_0)
        \end{pmatrix} \\
        &\hspace{5em} - \det\begin{pmatrix}
            Y_4 & 1 & (\Msup\nu-22-\Msup\nu-12)\M\nu02(E_2\beta_1-b_1) - (\M\nu-1-\M\nu-2)(E_2\beta_0-b_0) \\
            Y_5 & 1 & (\Msup\nu-22-\Msup\nu-12)\M\nu01(E_1\beta_1-b_1) - (\M\nu-1-\M\nu-2)(E_1\beta_0-b_0) \\
            Y_6 & 1 & (\Msup\nu-22-\Msup\nu-12)\M\nu00(E_0\beta_1-b_1) - (\M\nu-1-\M\nu-2)(E_0\beta_0-b_0)
        \end{pmatrix}, \displaybreak[0] \\
        \delta_5(\la) &=
        \begin{pmatrix}
            1 & E_2 & (\Msup\nu-22-\Msup\nu-12)Y_4 - (Y_2-Y_1)(a_0-\Msup\nu022)(E_2\beta_0-b_0) \\
            1 & E_1 & (\Msup\nu-22-\Msup\nu-12)Y_5 - (Y_2-Y_1)(a_0-\Msup\nu012)(E_1\beta_0-b_0) \\
            1 & E_0 & (\Msup\nu-22-\Msup\nu-12)Y_6 - (Y_2-Y_1)(a_0-\Msup\nu002)(E_0\beta_0-b_0)
        \end{pmatrix}, \displaybreak[0] \\
        \delta_6(\la) &=
        \begin{pmatrix}
            1 & E_2 & (Y_1-Y_2)\M\nu02(E_2\beta_0-b_0) + (\M\nu-2-\M\nu-1)Y_4 \\
            1 & E_1 & (Y_1-Y_2)\M\nu01(E_1\beta_0-b_0) + (\M\nu-2-\M\nu-1)Y_5 \\
            1 & E_0 & (Y_1-Y_2)\M\nu00(E_0\beta_0-b_0) + (\M\nu-2-\M\nu-1)Y_6
        \end{pmatrix}.
    \end{align*}
\end{lem}

%Solving this system and changing variables to map each sectorial boundary $\partial D^pm$ to $\partial D$, we obtain solution representation~\eqref{eqn:loopSoln}, but with $\delta_k$ defined using the vector $Y$ given by equation~\eqref{eqn:loopY} in place of the vector specified in equation~\eqref{eqn:loopSolnY}.
%It remains to argue that the additional terms resulting from the first vector in equation~\eqref{eqn:loopY} yield zero contribution.

We aim to substitute the formulae afforded by lemma~\ref{lem:loopLinearSystemSolution} into the Ehrenpreis form~\eqref{eqn:generalEFe} via equation~\eqref{eqn:generalDefnF}.
However, doing so does not immediately provide a useful solution representation, because our formulae for $F_e$ would all depend on $\hat u_-(\argdot;t),\hat u_+(\argdot;t),\hat u_0(\argdot;t)$ via formula~\eqref{eqn:loopY} for $Y$.
We shall argue, via Jordan's lemma and Cauchy's theorem, that this dependence may be removed.
This requires an asymptotic analysis of $\Delta,\delta_k$, on which we now embark.

By lemma~\ref{lem:generalNu} and equation~\eqref{eqn:loopDelta}, as $\la\to\infty$ along any ray of constant argument,
\begin{equation} \label{eqn:loopDeltaAsymptotic}
    \Delta(\la) = -3 \la^3 \left( \sum_{j=0}^2 \re^{\ri\alpha^j\la\eta} \left[\alpha^j\beta_1-\alpha^{2j}\beta_0\right] + \sum_{j=0}^2 \re^{-\ri\alpha^j\la\eta} \left[\alpha^jb_1-\alpha^{2j}b_0\right] \right)\left(1+\bigoh{\la^{-1}}\right).
\end{equation}
In particular, as $\la\to\infty$ from within $\clos(D)$, the dominant term is the $j=0$ term in the first sum plus, if $\la$ is approaching $\infty$ asymptotically parallel to the boundary, a term from the second sum.
We improve on this observation in the following lemma, whose proof appears at the end of this section.
The result of this lemma is essentially that the leading order terms in expression~\eqref{eqn:loopDeltaAsymptotic} do not cancel on $\clos(D)$ and that the zeros of $\Delta$ are exterior to $D$.

\begin{lem} \label{lem:loopDeltaDominant}
    Let
    \begin{equation} \label{eqn:loopDeltaDominant}
        \Delta_{\mathrm{dominant}}(\la)
        =
        (\beta_1-\beta_0)\re^{\ri\la\eta}
        + \re^{-\ri\alpha\la\eta}(\alpha b_1-\alpha^2b_0)
        + \re^{-\ri\alpha^2\la\eta}(\alpha^2b_1-\alpha b_0).
    \end{equation}
    For $\sigma>0$, define $S(\sigma)$ to be the set $S$ given by $\epsilon=\sigma$, and
    \[
        D_\sigma
        =
        D\cup S(\sigma),
    \]
    so that $\clos(D)\subsetneq D_\sigma$.
    As $\la\to\infty$ within $D_\sigma$,
    \[
        -\Delta(\la)/3\la^3 = \Delta_{\mathrm{dominant}}(\la) + \bigoh{\la^{-1}\left[\abs{\re^{\ri\la\eta}}+\abs{\re^{-\ri\alpha\la\eta}}+\abs{\re^{-\ri\alpha^2\la\eta}}\right]}.
    \]
    Moreover, provided the criteria on $\beta_j,b_j$ specified in proposition~\ref{prop:loopSolution} hold and $\sigma'>0$ is sufficiently small, there exists a function $\mu$ with $\sup_{\la\in\clos(D)}\abs{\mu(\la)}<1$ such that as $\la\to\infty$ within $D_{\sigma'}$,
    \[
        \Delta(\la) = -\ri\sqrt3\la^3(\beta_1-\beta_0)\left(1+\mu(\la)\right)\re^{\ri\la\eta} + \bigoh{\la^{2}\abs{\re^{\ri\la\eta}}}.
    \]
    Finally, the asymptotic locus of the zeros of $\Delta$ in $D_\sigma$ is
    \begin{alignat}{6}
        \label{eqn:loopDeltaZeros+}
        \la &\sim-\alpha  (x+\ri y)/\eta & &\mbox{ where } & y &= s+q,\; & x &= r+p+2n\pi, & &\mbox{ for } & n &\in\NN, \\
        \label{eqn:loopDeltaZeros-}
        \la &\sim \alpha^2(x+\ri y)/\eta & &\mbox{ where } & -y &= s'+q,\; & x &= -(r'+p)+2n\pi, & &\mbox{ for } & n &\in\NN,
    \end{alignat}
    in which
    \[
        -\re^{-\ri p + q} = \beta_1-\beta_0,
        \qquad
        \re^{\ri r-s} = \alpha b_1-\alpha^2b_0,
        \qquad
        \re^{\ri r'-s'} = \alpha^2 b_1-\alpha b_0.
    \]
\end{lem}

\begin{lem} \label{lem:loopDeformContour}
    Suppose $V_-\in\AC_{\mathrm p}^0(-\infty,0]$, $V_+\in\AC_{\mathrm p}^0[0,\infty)$ and $V_0\in\AC_{\mathrm p}^0[0,\eta]$.
    Denote by $\BVec{V}$ the list $(V_-,V_+,V_0)$.
    For $\la\in\clos(D)$, let $\widehat{\BVec{V}}(\la)$ be the vector of Fourier transforms given by
    \begin{multline*}
        \widehat{\BVec{V}}(\la)
        =
        \left(
            \int_{-\infty}^0 \re^{-\ri\nu_-(\alpha^2\la) x} V_-(x) \D x, \quad
            \int_{-\infty}^0 \re^{-\ri\nu_-(\alpha\la) x} V_-(x) \D x, \quad
            \int_0^\infty \re^{-\ri\nu_+(\la) x} V_+(x) \D x, \right. \\ \left.
            \int_0^\eta \re^{-\ri\nu_0(\alpha^2\la) x} V_0(x) \D x, \quad
            \int_0^\eta \re^{-\ri\nu_0(\alpha\la) x} V_0(x) \D x, \quad
            \int_0^\eta \re^{-\ri\nu_0(\la) x} V_0(x) \D x
        \right).
    \end{multline*}
    Define
    \begin{align*}
        \M\zeta-3(\la;\BVec V) &= \left[ \delta_1(\la) + \nu_-(\la)\delta_5(\la) + (a_--\nu_-(\la)^2)\delta_6(\la) \right] \nu'_-(\la), \\
        \M\zeta+1(\la;\BVec V) &= \left[ -\delta_2(\la) + B_1\nu_+(\alpha\la)\delta_5(\la) + B_0(a_+-\nu_+(\alpha\la)^2)\delta_6(\la) \right] \nu'_+(\alpha\la), \\
        \M\zeta+2(\la;\BVec V) &= \left[ -\delta_2(\la) + B_1\nu_+(\alpha^2\la)\delta_5(\la) + B_0(a_+-\nu_+(\alpha^2\la)^2)\delta_6(\la) \right] \nu'_+(\alpha^2\la), \\
        \M\zeta03(\la;\BVec V) &= \left[ \delta_4(\la) + \beta_1\nu_0(\la)\delta_5(\la) + B_0(a_0-\nu_0(\la)^2)\delta_6(\la) \right] \nu'_0(\la), \\
        \M\zeta01(\la;\BVec V) &= \left[ -\delta_3(\la) + b_1\nu_0(\alpha\la)\delta_5(\la) + b_0(a_0-\nu_0(\alpha\la)^2)\delta_6(\la) \right] \nu'_0(\alpha\la), \\
        \M\zeta02(\la;\BVec V) &= \left[ -\delta_3(\la) + b_1\nu_0(\alpha^2\la)\delta_5(\la) + b_0(a_0-\nu_0(\alpha^2\la)^2)\delta_6(\la) \right] \nu'_0(\alpha^2\la),
    \end{align*}
    for $\delta_k$ the determinant of matrix $\mathcal A$ given by equation~\eqref{eqn:loopA}, except with the $k$\textsuperscript{th} column replaced by $\widehat{\BVec{V}}(\la)$.
    Then, as $\la\to\infty$ within $\clos(D)$, $\M\zeta ej(\la;\BVec V)/\Delta(\la)=\lindecayla$ and they are also all continuous functions on $\clos(D)$ and analytic in $D$.
    Moreover,
    \begin{subequations} \label{eqn:loopDeformContour.Integral0}
    \begin{align}
        \label{eqn:loopDeformContour.Integral0.--}
        &\forall\;x<0, & 0 &= \int_{\partial D} \re^{\ri\nu_-(\la)x} \frac{\M\zeta-3(\la;\BVec V)}{\Delta(\la)} \D \la, \\
        \label{eqn:loopDeformContour.Integral0.++}
        &\forall\;x>0, & 0 &= \int_{\partial D} \frac{\alpha\re^{\ri\nu_+(\alpha\la)x} \M\zeta+1(\la;\BVec V) + \alpha^2\re^{\ri\nu_+(\alpha^2\la)x} \M\zeta+2(\la;\BVec V)}{\Delta(\la)} \D \la, \\
        \label{eqn:loopDeformContour.Integral0.0-}
        &\forall\;x<\eta, & 0 &= \int_{\partial D} \re^{\ri\nu_0(\la)(x-\eta)} \frac{\M\zeta03(\la;\BVec V)}{\Delta(\la)} \D \la, \\
        \label{eqn:loopDeformContour.Integral0.0+}
        &\forall\;x>0, & 0 &= \int_{\partial D} \frac{\alpha\re^{\ri\nu_0(\alpha\la)x} \M\zeta01(\la;\BVec V) + \alpha^2\re^{\ri\nu_0(\alpha^2\la)x} \M\zeta02(\la;\BVec V)}{\Delta(\la)} \D \la.
    \end{align}
    \end{subequations}
\end{lem}

By construction, setting $\BVec v=(u_-(\argdot,t),u_+(\argdot,t),u_0(\argdot,t))$ and $\BVec w=(-U_-,-U_+,-U_0)$, we obtain from equations~(\ref{eqn:loopLinearSystem}--\ref{eqn:loopY}) that, for all $\la\in\clos(D)$,
\begin{align*}
    F_-(        \la;0,t) &= \re^{\ri\la^3t}\frac{\M\zeta-3(\la;\BVec v)}{\Delta(\la)} + \frac{\M\zeta-3(\la;\BVec w)}{\Delta(\la)}, \\
    F_+(\alpha^2\la;0,t) &= \re^{\ri\la^3t}\frac{\M\zeta+2(\la;\BVec v)}{\Delta(\la)} + \frac{\M\zeta+2(\la;\BVec w)}{\Delta(\la)}, \\
    F_+(\alpha  \la;0,t) &= \re^{\ri\la^3t}\frac{\M\zeta+1(\la;\BVec v)}{\Delta(\la)} + \frac{\M\zeta+1(\la;\BVec w)}{\Delta(\la)}, \\
    F_0(\alpha^j\la;0,t) &= \re^{\ri\la^3t}\frac{\M\zeta0j(\la;\BVec v)}{\Delta(\la)} + \frac{\M\zeta0j(\la;\BVec w)}{\Delta(\la)}, \qquad j\in\{1,2,3\}.
\end{align*}
We substitute these formulae into the Ehrenpreis form~\eqref{eqn:generalEFe} to obtain
\begin{multline} \label{eqn:loopAlmostSolution}
    2\pi u_-(x,t)
    = \int_{-\infty}^\infty \re^{\ri\la x - \ri(\la^3-a_-\la)t} \widehat U_-(\la) \D \la \\
    - \int_{\partial D} \re^{\ri\nu_-(\la)x} \left[ \frac{\M\zeta-3(\la;\BVec v)}{\Delta(\la)} \right] \D\la
    - \int_{\partial D} \re^{\ri\nu_-(\la)x-\ri\la^3t} \left[ \frac{\M\zeta-3(\la;\BVec w)}{\Delta(\la)} \right] \D\la,
\end{multline}
and similarly for the other $u_e$.
Note that the $\re^{-\ri\la^3t}$ factor in each integrand of the Ehrenpreis form has canceled with the $\re^{\ri\la^3t}$ multiplying the $\BVec v$ dependent parts of the above formulae.
By lemma~\ref{lem:loopDeformContour}, the $\BVec v$ dependent integral in equation~\eqref{eqn:loopAlmostSolution} evaluates to $0$, and such terms also vanish from the other $u_e$.
In this way, equation~\eqref{eqn:loopAlmostSolution} reduces to equation~\eqref{eqn:loopSoln.-}, and similarly for the other $u_e$.
This completes the proof of proposition~\ref{prop:loopSolution}.

\begin{rmk}
    It may be possible to weaken the strict inequality $\abs{\beta_1-\beta_0}>\abs{\alpha^j b_1-b_0}$ to permit equality, resulting in zeros of $\Delta$ lying asymptotically on $\partial D$, and still solve this problem, but this would necessitate contour deformation along a series of finite semicircular indentations away from such zeros.
    Similar arguments have been made for finite interval problems in, for example,~\cite{Fok2008a}.
    However, such an approach requires an analysis of the locus of the zeros of $\Delta$ more precise than that presented here.
\end{rmk}

\begin{proof}[Proof of lemma~\ref{lem:loopDeltaDominant}]
    The first claim follows from equation~\eqref{eqn:loopDelta} and the asymptotic form of $\nu_e(\la)=\la+\lindecayla$ established in lemma~\ref{lem:generalNu}.

    Suppose $\la\to\infty$ within $D_\sigma$.
    By the criteria of the proposition, $\beta_1-\beta_0\neq0$, so the $\re^{\ri\la\eta}$ term appears in $\Delta_{\mathrm{dominant}}$ with nonzero coefficient and, by the first claim of the lemma, this term dominates all terms in $\Delta$ that are not explicitly given in $-3\la^3\Delta_{\mathrm{dominant}}$.
    If, for any $\epsilon'>0$, $\arg(\la)$ is bounded within $[\frac{-\pi}3+\epsilon',\frac{-\pi}6-\epsilon']$, then the $\re^{\ri\la\eta}$ term dominates all other terms, and the result is plainly true with $\mu(\la)=0$.
    It remains to consider $\la\to\infty$ within sectors of arbitarily narrow aperture $\epsilon'>0$ given by $\frac{-\pi}3\leq\arg(\la)\leq\frac{-\pi}3+\epsilon'$ and $\frac{-\pi}6-\epsilon'\leq\arg(\la)\leq\frac{-\pi}6$, together with the two semistrips that make up $D_\sigma\setminus D$.
    To do so, we define
    \begin{align*}
        \Msup D\sigma{\epsilon'}+ &:= D_{\frac\sigma\eta} \cap \{\la\in\CC:\tfrac{-\pi}6-\epsilon'\leq\arg(\la)\}, \\
        \Msup D\sigma{\epsilon'}- &:= D_{\frac\sigma\eta} \cap \{\la\in\CC:\arg(\la)\leq\tfrac{-\pi}3+\epsilon'\},
    \end{align*}
    and follow the approach of Langer~\cite{Lan1931a}, which will also yield the asymptotic locus of the zeros.

    Consider first $\la\to\infty$ within $\Msup D\sigma{\epsilon'}+$.
    We parametrize such $\la$ as $\la=-\alpha(x+\ri y)/\eta$, for $x\gg0$ and $-x\arctan(\epsilon') \leq y \leq \sigma$.
    Then the $\re^{-\ri\alpha^2\la\eta}$ term in $\Delta_{\mathrm{dominant}}$ is dominated by the other two.
    Evaluating just those terms of $\Delta_{\mathrm{dominant}}$, we get
    \[
        (\beta_1-\beta_0)\re^{\ri\la\eta} + \re^{-\ri\alpha\la\eta}(\alpha b_1-\alpha^2b_0)
        =
        \re^{\frac{\sqrt3}2(x+\ri y)} \left[ -\re^{-\ri p + q} \re^{\ri\frac x2 - \frac y2} + \re^{\ri r-s} \re^{-(\ri\frac x2 - \frac y2)}\right],
    \]
    for $p,q,r,s$ as defined in the statement of the lemma.
    Hence the leading order terms of $\Delta_{\mathrm{dominant}}$ cancel if and only if $\re^{\ri x - y} = \re^{\ri(r+p)-(s+q)}$.
    Therefore, the asymptotic locus of the zeros of $\Delta$ within $\Msup D\sigma{\epsilon'}+$ is expression~\eqref{eqn:loopDeltaZeros+}.
    By definition, $s+q>0$ is equivalent to $\abs{\beta_1-\beta_0}>\abs{\alpha^2b_1-b_0}$, one of the criteria of proposition~\ref{prop:loopSolution}, and this guarantees that there are at most finitely many zeros of $\Delta$ in $\Msup D\sigma{\epsilon'}+$ within a fixed small distance from $\clos(D)$.
    More precisely, by restricting $\sigma<s+q$, we guarantee the second claim of the lemma within $\Msup D\sigma{\epsilon'}+$, with $\sup_{\la\in \Msup D\sigma{\epsilon'}+}\abs{\phi(\la)}\leq\re^{-(s+q-\sigma)/\eta}<1$.

    If instead $\la\to\infty$ within $\Msup D\sigma{\epsilon'}-$, we can express $\la=\alpha^2(x+\ri y)/\eta$, for $x\gg0$ and $-\sigma\leq y \leq x\arctan(\epsilon)$.
    In this case, the dominant terms in $\Delta_{\mathrm{dominant}}$ are
    \[
        (\beta_1-\beta_0)\re^{\ri\la\eta} + \re^{-\ri\alpha^2\la\eta}(\alpha^2b_1-\alpha b_0)
        =
        \re^{\frac{\sqrt3}2(x+\ri y)} \left[ -\re^{-\ri p + q} \re^{-(\ri\frac x2 - \frac y2)} + \re^{\ri r'-s'} \re^{\ri\frac x2 - \frac y2)}\right],
    \]
    for $p,q,r',s'$ as defined in the statement of the lemma.
    The locus~\eqref{eqn:loopDeltaZeros-} and bound on $\abs{\phi(\la)}$ follow as before.
\end{proof}

\begin{proof}[Proof of lemma~\ref{lem:loopDeformContour}]
    Lemma~\ref{lem:generalNu} guarantees that, if $\la\in\clos(D)$, then $\nu_e(\alpha\la),\nu_e(\alpha^2\la)\in\clos(\CC^+)$ and $\nu_e(\la)\in\clos(\CC^-)$.
    Hence, by lemma~\ref{lem:FTLeadingOrder}, as $\la\to\infty$ within $\clos(D)$,
    \begin{gather*}
        \abs{\widehat{\BVec V}_1(\la)},\abs{\widehat{\BVec V}_2(\la)},\abs{\widehat{\BVec V}_3(\la)},\abs{\widehat{\BVec V}_6(\la)} = \lindecayla,
        \\
        \abs{\widehat{\BVec V}_4(\la)} = \bigoh{\la^{-1}\re^{-\ri\alpha^2\la\eta}},
        \qquad
        \abs{\widehat{\BVec V}_5(\la)} = \bigoh{\la^{-1}\re^{-\ri\alpha\la\eta}},
    \end{gather*}
    all uniformly in $\arg(\la)$.
    From the above estimates and lemmata~\ref{lem:generalNu} and~\ref{lem:loopLinearSystemSolution}, we see that
    \[
        \delta_6(\la)
        =
        \det
        \begin{pmatrix}
            1 & E_2 & (\widehat{\BVec V}_2-\widehat{\BVec V}_1)\M\nu02(E_2\beta_1-b_1) - (\M\nu-1-\M\nu-2)\widehat{\BVec V}_4 \\
            1 & E_1 & (\widehat{\BVec V}_2-\widehat{\BVec V}_1)\M\nu01(E_1\beta_1-b_1) - (\M\nu-1-\M\nu-2)\widehat{\BVec V}_5 \\
            1 & E_0 & (\widehat{\BVec V}_2-\widehat{\BVec V}_1)\M\nu00(E_0\beta_1-b_1) - (\M\nu-1-\M\nu-2)\widehat{\BVec V}_6
        \end{pmatrix}
        =
        \bigoh{\re^{\ri\la\eta}}.
    \]
    Hence, noting that $\nu'_e(\alpha^j\la)=\bigoh{1}$, the contributions of $\delta_6(\la)$ to $\M\zeta ej(\la;\BVec V)$, multiplied as they are by $(a_e-\nu_e(\alpha^j\la)^2)$, are all $\bigoh{\la^{2}\re^{\ri\la\eta}}$.
    In exactly the same way, it may be established that $\M\zeta ej(\la;\BVec V)=\bigoh{\la^{2}\re^{\ri\la\eta}}$ for all $e,j$ for which they are defined, and all these asymptotic bounds are uniform in $\arg(\la)$ as $\la\to\infty$ within $\clos(D)$.
    Hence, by lemma~\ref{lem:loopDeltaDominant}, the first claim of the present lemma is justified.

    By lemma~\ref{lem:generalNu}, $\nu_0$ is analytic on an open set containing $\clos(D)$ and
    \[
        \re^{\ri\nu_0(\la)(x-\eta)} = \re^{\ri\la(x-\eta)}\left[1+\lindecayla\right],
    \]
    so the integrand in equation~\eqref{eqn:loopDeformContour.Integral0.0-} can be expressed as
    \(
        \re^{\ri(-\la)(\eta-x)}y(\la)
    \),
    with $y$ analytic on $D$ and continuous on $\clos(D)$ having the property that $\max_{\theta\in[\frac{-\pi}3,\frac{-\pi}6]}y(r\re^{\ri\theta})\to0$ as $r\to\infty$.
    By Jordan's lemma, using $(\eta-x)>0$, the integral along the circular arc from $r\re^{-\ri\frac\pi3}$ to $r\re^{-\ri\frac\pi6}$ of $\re^{\ri(-\la)(\eta-x)}y(\la)$ has limit zero for large $r$.
    By Cauchy's theorem, the integral about $\partial(D\cap B(0,r))$ of the same integrand is zero.
    Equation~\eqref{eqn:loopDeformContour.Integral0.0-} follows.
    The other equations~\eqref{eqn:loopDeformContour.Integral0} may be justified similarly.
\end{proof}

\subsection{Solution satisfies the problem: proof of proposition~\protect\ref{prop:loopExistence}} \label{ssec:loop.existence}

\begin{lem} \label{lem:loopConvergenceZetaLeadingOrder}
    Suppose $U_-,U_+,U_0$, are as in the hypothesis of proposition~\ref{prop:loopExistence} and denote by $\BVec{v}$ the list $(U_-,U_+,U_0)$.
    For $\la\in\partial D$, let $\hat{\BVec{v}}(\la)$ and $\M\zeta ej$ be defined as in the statement of lemma~\ref{lem:loopDeformContour}, but in terms of this $\BVec{v}$.
    Then
    \[
        \M\zeta ej(\la;\BVec v)/\Delta(\la) = \M\phi ej(\la) + \M\psi ej(\la),
    \]
    with $\M\psi ej(\la) = \bigoh{\la^{-5}}$ as $\la\to\infty$ along $\partial D$ and $\M\phi ej$ is holomorphically extensible to the closed semistrips $S$ and obeys
    \[
        \abs{\M\phi ej(\la)}=\bigoh{\la^{-1}}
    \]
    as $\la\to\infty$ within $S$.
\end{lem}

\begin{proof}[Proof of lemma~\ref{lem:loopConvergenceZetaLeadingOrder}]
    This proof is identical to the first paragraph of the proof of lemma~\ref{lem:loopDeformContour}, except that the improved regularity of $U_-,U_+,U_0$ allows application of lemma~\ref{lem:FTLeadingOrder} with $N=4$ instead of $N=1$.
\end{proof}

\begin{proof}[Proof that each $u_e$ satisfies its PDE]
    The arguments for the integrals in equations~(\ref{eqn:loopSoln}.$\pm$) match exactly the corresponding arguments in the proof for the mismatch problem, except that now lemma~\ref{lem:loopConvergenceZetaLeadingOrder} plays the role of lemma~\ref{lem:mismatchConvergenceZetaLeadingOrder}.
    Similarly to the argument for the final integral of equation~\eqref{eqn:loopSoln.-}, via its absolute convergence uniformly in $(x,t)\in[0,x_0]\times[0,\infty)$ for any $x_0<\eta$, the final integral on the right of equation~\eqref{eqn:loopSoln.0} satisfies~\eqref{eqn:generalproblem.PDE}.
    The proof for the first integral on the right of equation~\eqref{eqn:loopSoln.0} mirrors that for the corresponding integrals in equation~(\ref{eqn:loopSoln}.$\pm$).
    The analysis of the second integral on the right of equation~\eqref{eqn:loopSoln.0} is identical to the analysis of the second integral on the right of~\eqref{eqn:loopSoln.+}.
\end{proof}

\begin{proof}[Proof that each $u_e$ satisfies its initial condition]
    As in the mismatch problem, evaluation at $t=0$ reduces each solution formula to~\eqref{eqn:generalproblem.IC}, via lemma~\ref{lem:loopDeformContour} and the Fourier inversion theorem; we need only check continuity of the integrals at $t=0$.

    The last integral in equation~\eqref{eqn:loopSoln.0} converges uniformly in $t\in[0,\infty)$, just like the last integral in equation~\eqref{eqn:loopSoln.-}, from which continuity follows.
    The argument for the first integral of solution formula~\eqref{eqn:loopSoln.0} is exactly the same as for the first integrals of the other formulae.
    The second integral in equation~\eqref{eqn:loopSoln.0} is analysed similarly to the second integral in equation~\eqref{eqn:loopSoln.+}.
\end{proof}

\begin{proof}[Proof that the $u_e$ satisfy the vertex conditions]
    We aim to show that, for $k\in\{0,1\}$,
    \begin{subequations} \label{eqn:loopVertex.toProve}
    \begin{align}
        \label{eqn:loopVertex.toProve.b}
        (-\ri)^k b_k \lim_{x\to0^-} 2\pi \partial_x^k u_-(x,t) - (-\ri)^k \lim_{x\to0^+} 2\pi \partial_x^k u_0(x,t) &= 0, \\
        \label{eqn:loopVertex.toProve.beta}
        (-\ri)^k \beta_k \lim_{x\to0^-} 2\pi \partial_x^k u_-(x,t) - (-\ri)^k \lim_{x\to\eta^-} 2\pi \partial_x^k u_0(x,t) &= 0, \\
        \label{eqn:loopVertex.toProve.B}
        (-\ri)^k B_k \lim_{x\to0^-} 2\pi \partial_x^k u_-(x,t) - (-\ri)^k \lim_{x\to0^+} 2\pi \partial_x^k u_+(x,t) &= 0.
    \end{align}
    \end{subequations}

    Of these, equation~\eqref{eqn:loopVertex.toProve.B} is most similar to the argument presented in~\S\ref{sec:mismatch}, because we need not analyse solution representation~\eqref{eqn:loopSoln.0}.
    Indeed, one may treat the terms in each of equations~(\ref{eqn:loopSoln}$.\pm$) in exactly the same way as we analysed the corresponding terms in equations~(\ref{eqn:mismatchSolution}.$\pm$), determining that the left side of equation~\eqref{eqn:loopVertex.toProve.B} reduces to exactly expression~\eqref{eqn:mismatchVertexFull.allD3Gamma3} but now with%
    \begin{subequations} \label{eqn:loopVertexFull.defnphipsijpm}
    \begin{align}
        \label{eqn:loopVertexFull.defnphipsijpm.firstbit}
        \M\omega3-(\la) &= \Omega_1 + \M\nu-0\Omega_5 + (a_--\Msup\nu-02)\Omega_6,
        \\
        \label{eqn:loopVertexFull.defnphipsij+}
        \M\omega j+(\la) &= -\Omega_2 + B_1\M\nu+j\Omega_5 + B_0(a_+-\Msup\nu+j2)\Omega_6
        & j &\in \{1,2\}, \\
        \label{eqn:loopVertexFull.defnphipsij0}
        \M\omega j0(\la) &= -\Omega_3 + b_1\M\nu0j\Omega_5 + b_0(a_0-\Msup\nu0j2)\Omega_6
        & j &\in \{1,2\}, \\
        \M\omega30(\la) &= \Omega_4 + \beta_1\M\nu00\Omega_5 + \beta_0(a_0-\Msup\nu002)\Omega_6,
    \end{align}
    for
    \begin{equation}
        \label{eqn:loopVertexFull.matrixPhi}
        \mathcal A
        \begin{pmatrix}
            \Omega_1 \\ \Omega_2 \\ \Omega_3 \\ \Omega_4 \\ \Omega_5 \\ \Omega_6
        \end{pmatrix}
        =
        \begin{pmatrix}
            -\M\omega\RR-(\alpha^2\la) \\ -\M\omega\RR-(\alpha\la) \\ -\M\omega\RR+(\la) \\ -\M\omega\RR0(\alpha^2\la) \\ -\M\omega\RR0(\alpha\la) \\ -\M\omega\RR0(\la)
        \end{pmatrix},
    \end{equation}
    \end{subequations}
    in which the symbols $\omega,\Omega$ represent $\phi,\Phi$ or $\psi,\Psi$, respectively and $\mathcal A$ is now given by equation~\eqref{eqn:loopA}.
    Here $\M\omega\RR\pm$ are defined following the construction in the earlier proof, and $\M\omega\RR0$ are the functions with
    \[
        \widehat U_0(\nu_0(\la)) = \left[\M\phi\RR0(\la)+\M\psi\RR0(\la)\right]\nu'_0(\la),
    \]
    with the appropriate decay properties in $\clos(\CC^-)$, guaranteed to exist by lemma~\ref{lem:FTLeadingOrder} with $N=4$.
    The first result of lemma~\ref{lem:loopDeformContour} gives the corresponding decay properties of $\M\omega j\pm$, enabling applications of lemmata~\ref{lem:vertexRewritingLemma.Real-} and~\ref{lem:vertexRewritingLemma.Real+}.
    (The definitions of $\M\omega j0$, $\Omega_3$, $\Omega_4$, and $\M\omega\RR0(\alpha^j\la)$ are not required in this calculation, but shall be needed for the later proofs of statements~\eqref{eqn:loopVertex.toProve.b} and~\eqref{eqn:loopVertex.toProve.beta}, so we give the definitions here for convenience of reference.)

    By equations~\eqref{eqn:loopVertexFull.defnphipsijpm}, the bracket in expression~\eqref{eqn:mismatchVertexFull.allD3Gamma3} reduces to
    \begin{equation*}
        \begin{pmatrix}
            \displaystyle - B_k \sum_{j=0}^2 \alpha^j\Msup\nu-jk\Msup\nu-j\prime \\
            \displaystyle - \sum_{j=0}^2 \alpha^j\Msup\nu+jk\Msup\nu+j\prime \\
            \displaystyle B_1 \sum_{j=0}^2 \alpha^j\Msup\nu+j{k+1}\Msup\nu+j\prime
            - B_k \sum_{j=0}^2 \alpha^j\Msup\nu-j{k+1}\Msup\nu-j\prime \\
            \displaystyle B_0 \sum_{j=0}^2 \alpha^j\Msup\nu+jk\left(a_+-\Msup\nu+j2\right)\Msup\nu+j\prime
            - B_k \sum_{j=0}^2 \alpha^j\Msup\nu-jk\left(a_--\Msup\nu-j2\right)\Msup\nu-j\prime
        \end{pmatrix}
        \cdot
        \begin{pmatrix}
            \Omega_1 \\ \Omega_2 \\ \Omega_5 \\ \Omega_6
        \end{pmatrix}.
    \end{equation*}
    The symmetries of $\nu_e$ presented in lemma~\ref{lem:generalNu} imply that every entry in the first vector is $0$.

    Proofs of equations~\eqref{eqn:loopVertex.toProve.b} and~\eqref{eqn:loopVertex.toProve.beta} are very similar in structure.
    In particular, the analysis of the limit of $\partial_x^ku_-(x,t)$ is identical; only the coefficient $B_k$ changes to $b_k$ or $\beta_k$.

    When analysing the small $x$ limit of $\partial_x^ku_0(x,t)$, one uses lemma~\ref{lem:FTLeadingOrder} to treat the first integral of solution representation~\eqref{eqn:loopSoln.0} via lemma~\ref{lem:vertexRewritingLemma.Real+} exactly as we did above for the first integral of~\eqref{eqn:loopSoln.+}.
    Still for the small $x$ limit, the second and third terms of~\eqref{eqn:loopSoln.0} can be treated the same way as the second term of~\eqref{eqn:loopSoln.+} and the second term of~\eqref{eqn:loopSoln.-}, respectively.
    This means that the left side of equation~\eqref{eqn:loopVertex.toProve.b} reduces to
    \begin{multline} \label{eqn:loopVertexFull.allD3Gamma3.b}
        \PV\,\left\{ \int_{\partial D} + \int_{\gamma} \right\} \re^{-\ri\la^3t}
        \bigg[
            b_k\sum_{j=1}^2 \alpha^j \Msup\nu-jk\Msup\nu-j\prime \M\omega\RR-(\alpha^j\la)
            - b_k \Msup\nu-0k\Msup\nu-0\prime \M\omega3-(\la)
            \\
            + \Msup\nu00k\Msup\nu00\prime \M\omega\RR0(\la)
            + \sum_{j=1}^2 \alpha^j \Msup\nu0jk\Msup\nu0j\prime \M\omega j0(\la)
            + \Msup\nu00k\Msup\nu00\prime E_0 \M\omega30(\la)
        \bigg]
        \D\la.
    \end{multline}
    Using again construction~\eqref{eqn:loopVertexFull.defnphipsijpm} and the symmetries of lemma~\ref{lem:generalNu}, the coefficients of $\Omega_\ell$ all vanish inside the bracket of expression~\eqref{eqn:loopVertexFull.allD3Gamma3.b}.
    The only algebraic difference from the simplifications hereunto studied is that the coefficient of $\Omega_4$ in this case evaluates to $(1-1)E_0\Msup\nu00k\Msup\nu00\prime$, a pair of terms that mutually annihilate but which include an exponential factor, and similar pairs of terms appear in the coefficients of $\Omega_5$ and $\Omega_6$.

    To analyse the $x\to\eta^-$ limit of $\partial_x^ku_0(x,t)$, we must apply lemma~\ref{lem:vertexRewritingLemma.Real-} instead of lemma~\ref{lem:vertexRewritingLemma.Real+} to treat the first integral of solution representation~\eqref{eqn:loopSoln.0}; the result is similar to that for the first integral of~\eqref{eqn:loopSoln.-}.
    The $x\to\eta^-$ limit is equivalent to $(x-\eta)\to0^-$, matching the direction of the limit addressed in lemma~\ref{lem:vertexRewritingLemma.Real-}, and different from the limit analysed in lemma~\ref{lem:vertexRewritingLemma.Real+} but, more importantly, the $\re^{-\ri\eta\nu_0(\la)}$ factor that consequently appears in the integrand can be used to cancel the blowup of $\widehat U_0(\nu_0(\la))$ in $\CC^+$, per asymptotic equations~\eqref{eqn:FTLeadingOrder.etaU}.
    The left side of equation~\eqref{eqn:loopVertex.toProve.beta} then reduces to
    \begin{multline} \label{eqn:loopVertexFull.allD3Gamma3.beta}
        \PV\,\left\{ \int_{\partial D} + \int_{\gamma} \right\} \re^{-\ri\la^3t}
        \bigg[
            \beta_k\sum_{j=1}^2 \alpha^j \Msup\nu-jk\Msup\nu-j\prime \M\omega\RR-(\alpha^j\la)
            - \beta_k \Msup\nu-0k\Msup\nu-0\prime \M\omega3-(\la)
            \\
            - \sum_{j=1}^2 \alpha^j E_j^{-1} \Msup\nu0jk\Msup\nu0j\prime \M\omega\RR0(\alpha^j\la)
            + \sum_{j=1}^2 \alpha^j E_j^{-1} \Msup\nu0jk\Msup\nu0j\prime \M\omega j0(\la)
            + \Msup\nu00k\Msup\nu00\prime \M\omega30(\la)
        \bigg]
        \D\la.
    \end{multline}
    As above, the bracket evaluates to $0$ using construction~\eqref{eqn:loopVertexFull.defnphipsijpm} and lemma~\ref{lem:generalNu}.
\end{proof}

\section{Source defect} \label{sec:source}
Here we study a metric graph with a single bond which terminates at the same vertex to which both leads are attached, but whose source is another vertex at which certain boundary conditions are presecribed; see figure~\ref{fig:graph-source}.
\begin{figure}
    \centering
    \includegraphics{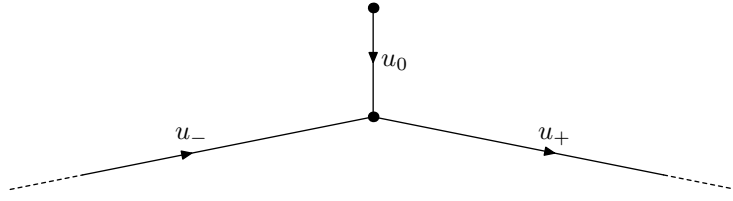}
    \caption{The metric graph domain with a source defect.}
    \label{fig:graph-source}
\end{figure}
Guided by the results of~\cite{FP2001a} for finite interval problems, we expect that the problem would be solvable if, in contrast with the mismatch problem of~\S\ref{sec:mismatch}, one further vertex condition were supplied at the target of bond $0$, and two extra boundary conditions were prescribed at its source.
We select the vertex and boundary conditions
\begin{subequations} \label{eqn:sourceVCandBC}
\begin{align}
    \label{eqn:sourceVC}
    \partial_x^k u_-(0,t) &= \beta_k^{-1} \partial_x^k u_0(\eta,t) = B_k^{-1} \partial_x^k u_+(0,t), & t &\in [0,T], & k &\in \{0,1\}, \\
    \label{eqn:sourceBC}
    \partial_x^k u_0(0,t) &= H_k(t), & t &\in [0,T], & k &\in \{0,1\},
\end{align}
\end{subequations}
with coefficients $b_k,\beta_k,B_k\in\CC\setminus\{0\}$ and known functions $H_k$ of regularity to be specified later.

We see this problem representing a ``source'' type defect because the orientation of the bond suggests that information, specified by the inhomogeneities $H_k$ and $U_0$, is flowing into the system from the defect.

\begin{prop} \label{prop:sourceSolution}
    Consider problem~\eqref{eqn:generalproblem} on the metric graph described above, with vertex and boundary conditions~\eqref{eqn:sourceVCandBC} for coefficients satisfying $\beta_0\neq\sqrt3\ri\beta_1$.
    Define, for $e\in\{+,-,0\}$ and $j\in\{0,1,2\}$, $\M\nu ej = \nu_e(\alpha^j\la)$, $E_j = \exp\left(-\ri\eta\M\nu0j\right)$, and
    \begin{equation}  \label{eqn:sourceA}
        \mathcal A(\la) =
        \begin{pmatrix}
            1 & 0 & 0 & 0 &  \M\nu-2    &  (a_--\Msup\nu-22) \\
            1 & 0 & 0 & 0 &  \M\nu-1    &  (a_--\Msup\nu-12) \\
            0 & 1 & 0 & 0 & -\M\nu+0B_1 & -(a_+-\Msup\nu+02)B_0 \\
            0 & 0 & 1 & E_2 & \M\nu02 E_2 \beta_1 & (a_0-\Msup\nu022)E_2\beta_0 \\
            0 & 0 & 1 & E_1 & \M\nu01 E_1 \beta_1 & (a_0-\Msup\nu012)E_1\beta_0 \\
            0 & 0 & 1 & E_0 & \M\nu00 E_0 \beta_1 & (a_0-\Msup\nu002)E_0\beta_0
        \end{pmatrix}.
    \end{equation}
    Define also $h_k(\la;t) = \int_0^t \re^{\ri\la^3s}H_k(s)\D s$.
    Let $\Delta=\det\mathcal A$, and $\delta_k$ be the determinant of $\mathcal A$ but with the $k$\textsuperscript{th} column replaced by
    \begin{multline} \label{eqn:sourceSolnY}
        \left(
            0,\;0,\;0,\;
            \ri \M\nu02 h_1 + (a_0-\Msup\nu022) h_0, \;
            \ri \M\nu01 h_1 + (a_0-\Msup\nu012) h_0, \;
            \ri \M\nu00 h_1 + (a_0-\Msup\nu002) h_0
        \right)
        \\
        -
        \left(
            \widehat U_-(\M\nu-2), \;
            \widehat U_-(\M\nu-1), \;
            \widehat U_+(\M\nu-0), \;
            \widehat U_0(\M\nu02), \;
            \widehat U_0(\M\nu01), \;
            \widehat U_0(\M\nu00)
        \right).
    \end{multline}
    Suppose $u$ is sufficiently smooth and satisfies this problem.
    Then, for $R$ sufficiently large,
    \begin{subequations} \label{eqn:sourceSoln}
    \begin{align}
    \notag
        2\pi u_-(x,t)
        &= \int_{-\infty}^\infty \re^{\ri\la x - \ri(\la^3-a_-\la)t} \widehat U_-(\la) \D \la \\
    \label{eqn:sourceSoln.-} \tag{\theparentequation.\(-\)}
        &\hspace{1.2em} - \int_{\partial D} \re^{\ri\nu_-(\la)x-\ri\la^3t} \left[ \frac{\delta_1(\la) + \nu_-(\la)\delta_5(\la) + (a_--\nu_-(\la)^2)\delta_6(\la)}{\Delta(\la)} \right] \nu'_-(\la) \D\la, \displaybreak[0] \\
    \notag
        2\pi u_+(x,t)
        &= \int_{-\infty}^\infty \re^{\ri\la x - \ri(\la^3-a_+\la)t} \widehat U_+(\la) \D \la \\
    \notag
        &\hspace{3em} - \int_{\partial D} \frac{\re^{-\ri\la^3t}}{\Delta(\la)} \Bigg[
            -\delta_2(\la) \left( \alpha \re^{\ri\nu_+(\alpha\la)x} \nu'_+(\alpha\la) + \alpha^2 \re^{\ri\nu_+(\alpha^2\la)x} \nu'_+(\alpha^2\la) \right) \\
    \notag
            &\hspace{6em} + B_1\delta_5(\la) \left( \alpha \nu_+(\alpha\la)\re^{\ri\nu_+(\alpha\la)x} \nu'_+(\alpha\la) + \alpha^2 \nu_+(\alpha^2\la) \re^{\ri\nu_+(\alpha^2\la)x} \nu'_+(\alpha^2\la) \right) \\
    \label{eqn:sourceSoln.+} \tag{\theparentequation.\(+\)}
            &\hspace{-4em} + B_0\delta_6(\la) \left( \alpha \left(a_+-\nu_+(\alpha\la)^2\right) \re^{\ri\nu_+(\alpha\la)x} \nu'_+(\alpha\la) + \alpha^2 \left(a_+-\nu_+(\alpha^2\la)^2\right) \re^{\ri\nu_+(\alpha^2\la)x} \nu'_+(\alpha^2\la) \right)
        \Bigg] \D\la, \displaybreak[0] \\
    \notag
        2\pi u_0(x,t)
        &= \int_{-\infty}^\infty \re^{\ri\la x - \ri(\la^3-a_0\la)t} \widehat U_0(\la) \D \la \\
    \notag
        &\hspace{3em} - \int_{\partial D} \re^{-\ri\la^3t} \Bigg[
            \frac{-\delta_3(\la)}{\Delta(\la)} \left( \alpha \re^{\ri\nu_0(\alpha\la)x} \nu'_0(\alpha\la) + \alpha^2 \re^{\ri\nu_0(\alpha^2\la)x} \nu'_0(\alpha^2\la) \right) \\
    \notag
            &\hspace{6em} + \ri h_1(\la;t) \left( \alpha \nu_0(\alpha\la)\re^{\ri\nu_0(\alpha\la)x} \nu'_0(\alpha\la) + \alpha^2 \nu_0(\alpha^2\la) \re^{\ri\nu_0(\alpha^2\la)x} \nu'_0(\alpha^2\la) \right) \\
    \notag
            &\hspace{-1.5em} + h_0(\la;t) \left( \alpha \left(a_0-\nu_0(\alpha\la)^2\right) \re^{\ri\nu_0(\alpha\la)x} \nu'_0(\alpha\la) + \alpha^2 \left(a_0-\nu_0(\alpha^2\la)^2\right) \re^{\ri\nu_0(\alpha^2\la)x} \nu'_0(\alpha^2\la) \right)
        \Bigg] \D\la \\
    \label{eqn:sourceSoln.0} \tag{\theparentequation.\(0\)}
        &\hspace{-1em} - \int_{\partial D} \re^{\ri\nu_0(\la)(x-\eta)-\ri\la^3t} \left[ \frac{\delta_4(\la) + \nu_0(\la)\beta_1\delta_5(\la) + (a_0-\nu_0(\la)^2)\beta_0\delta_6(\la)}{\Delta(\la)} \right] \nu'_0(\la) \D\la.
    \end{align}
    \end{subequations}
    Moreover, in expressions~\eqref{eqn:sourceSolnY} and~\eqref{eqn:sourceSoln.0}, $h_k(\argdot;t)$ may be replaced with $h_k(\argdot,T)$.
\end{prop}

\begin{prop} \label{prop:sourceExistence}
    Assume $U_e\in\AC^4_{\mathrm p}(\Omega_e)$ and $H_0,H_1\in\AC^2_{\mathrm p}[0,\infty)$, per definition~\ref{defn:piecewiseACfunctions}.
    For all $x\in\Omega_e$ and $t\geq0$, let $u_e(x,t)$ be defined by equations~\eqref{eqn:sourceSoln}.
    Then the functions $u_-,u_+,u_0$ satisfy problem~\eqref{eqn:generalproblem},~\eqref{eqn:sourceVCandBC}.
\end{prop}

\subsection{Solution representation: proof of proposition~\protect\ref{prop:sourceSolution}} \label{ssec:sourceUnicity}

With vertex and boundary conditions~\eqref{eqn:sourceVCandBC}, the global relations~\eqref{eqn:generalGRe} simplify to
\begin{subequations} \label{eqn:sourceGR}
\begin{align}
\label{eqn:sourceGR.-} \tag{\theparentequation.\(-\)}
    \la &\in \clos(\CC^+), &
    0 &= \re^{\ri\la^3t} \hat u_-(\M\nu-0;t) - \widehat U_-(\M\nu-0) - \left[ \M f-2 + \ri \M\nu-0 \M f-1 + \left( a_- - \M\nu-0^2 \right) \M f-0 \right] \\
\label{eqn:sourceGR.+} \tag{\theparentequation.\(+\)}
    \la &\in \clos(\CC^-), &
    0 &= \re^{\ri\la^3t} \hat u_+(\M\nu+0;t) - \widehat U_+(\M\nu+0) + \left[ \M f+2 + \ri \M\nu+0 B_1 \M f-1 + \left( a_+ - \M\nu+0^2 \right) B_0 \M f-0 \right], \\
\notag
    \la &\in \CC, &
    0 &= \re^{\ri\la^3t} \hat u_0(\M\nu00;t) - \widehat U_0(\M\nu00) + \left[ g_2 + \ri \M\nu00 h_1 + \left( a_0 - \M\nu00^2 \right) h_0 \right] \\
\label{eqn:sourceGR.0} \tag{\theparentequation.\(0\)}
    &&&\hspace{10em} - E_0\left[ G_2 + \ri \M\nu00 \beta_1 \M f-1 + \left( a_0 - \M\nu00^2 \right) \beta_0 \M f-0 \right],
\end{align}
\end{subequations}
in which, for $k\in\{0,1\}$,
\[
    \M f \pm j = \M f \pm j(\la;0,t), \qquad g_2 = \M f02(\la;0,t), \qquad G_2 = \M f02(\la;\eta,t), \qquad h_k = h_k(\la;t).
\]
For $\la\in\clos(D)$, we use equations
\begin{align*}
    &\mbox{\eqref{eqn:sourceGR.-}}\Big\rvert_{\la\mapsto\alpha^{2}\la},
    &
    &\mbox{\eqref{eqn:sourceGR.-}}\Big\rvert_{\la\mapsto\alpha\la},
    &
    &\mbox{\eqref{eqn:sourceGR.+}}\Big\rvert_{\la\mapsto\la},
    \\
    &\mbox{\eqref{eqn:sourceGR.0}}\Big\rvert_{\la\mapsto\alpha^{2}\la},
    &
    &\mbox{\eqref{eqn:sourceGR.0}}\Big\rvert_{\la\mapsto\alpha\la},
    &
    &\mbox{\eqref{eqn:sourceGR.0}}\Big\rvert_{\la\mapsto\la}.
\end{align*}
This yields the linear system, for $\la\in\clos(D)$,
\begin{equation} \label{eqn:sourceLinearSystem}
    \mathcal A(\la)
    \left( \M f-2 \quad -\M f+2 \quad -g_2 \quad G_2 \quad \ri \M f-1 \quad \M f-0 \right)^\top
    =
    Y(\la),
\end{equation}
where $\mathcal A(\la)$ is defined in equation~\eqref{eqn:sourceA}
and
\begin{equation} \label{eqn:sourceY}
    Y(\la) =
    \re^{\ri\la^3t}
    \begin{pmatrix}
        \hat u_-(\M\nu-2;t) \\ \hat u_-(\M\nu-1;t) \\ \hat u_+(\M\nu+0;t) \\ \hat u_0(\M\nu02;t) \\ \hat u_0(\M\nu01;t) \\ \hat u_0(\M\nu00;t)
    \end{pmatrix}
    -
    \begin{pmatrix}
        \widehat U_-(\M\nu-2) \\ \widehat U_-(\M\nu-1) \\ \widehat U_+(\M\nu+0) \\ \widehat U_0(\M\nu02) \\ \widehat U_0(\M\nu01) \\ \widehat U_0(\M\nu00)
    \end{pmatrix}
    +
    \begin{pmatrix}
        0 \\ 0 \\ 0 \\
        \ri\M\nu02h_1 + (a_0-\Msup\nu022)h_0 \\
        \ri\M\nu01h_1 + (a_0-\Msup\nu012)h_0 \\
        \ri\M\nu00h_1 + (a_0-\Msup\nu002)h_0
    \end{pmatrix}.
\end{equation}
Cramer's rule and row and column operations justify the following lemma.

\begin{lem} \label{lem:sourceLinearSystemSolution}
    Linear system~\eqref{eqn:sourceLinearSystem} has solution
    \begin{align*}
        \M f-2 &= \delta_1(\la)/\Delta(\la), &
        -\M f+2 &= \delta_2(\la)/\Delta(\la), &
        -g_2 &= \delta_3(\la)/\Delta(\la), \\
        G_2 &= \delta_4(\la)/\Delta(\la), &
        \ri \M f-1 &= \delta_5(\la)/\Delta(\la), &
        \M f-0 &= \delta_6(\la)/\Delta(\la),
    \end{align*}
    where $\Delta=\det\mathcal{A}$ and $\delta_k$ is the determinant of $\mathcal A$ but with the $k$\textsuperscript{th} column replaced by $Y$ (not necessarily given by equation~\eqref{eqn:sourceY}).
    Moreover,
    \begin{equation} \label{eqn:sourceDelta}
        \Delta(\la) = (\M\nu-1-\M\nu-2) \det
        \begin{pmatrix}
            1 & E_2^{-1} & \beta_0(a_0-\Msup\nu022) + \beta_1(\M\nu-2+\M\nu-1)\M\nu02 \\
            1 & E_1^{-1} & \beta_0(a_0-\Msup\nu012) + \beta_1(\M\nu-2+\M\nu-1)\M\nu01 \\
            1 & E_0^{-1} & \beta_0(a_0-\Msup\nu002) + \beta_1(\M\nu-2+\M\nu-1)\M\nu00 \\
        \end{pmatrix}
    \end{equation}
    and
    \begin{align*}
        \delta_1(\la) &=
        -\det\begin{pmatrix}
            1 & E_2 & Y_4 \\
            1 & E_1 & Y_5 \\
            1 & E_0 & Y_6
        \end{pmatrix} \det\begin{pmatrix} \M\nu-2 & (a_--\Msup\nu-22) \\ \M\nu-1 & (a_--\Msup\nu-12) \end{pmatrix} \\
        &\hspace{5em} + \beta_1\det\begin{pmatrix} E_2^{-1} & 1 & \M\nu02 \\ E_1^{-1} & 1 & \M\nu01 \\ E_0^{-1} & 1 & \M\nu00 \end{pmatrix} \det\begin{pmatrix} Y_1 & (a_--\Msup\nu-22) \\ Y_2 & (a_--\Msup\nu-12) \end{pmatrix}, \\
        &\hspace{5em} + \beta_0\det\begin{pmatrix} E_2^{-1} & 1 & (a_0-\Msup\nu022) \\ E_1^{-1} & 1 & (a_0-\Msup\nu012) \\ E_0^{-1} & 1 & (a_0-\Msup\nu002) \end{pmatrix} \det\begin{pmatrix} \M\nu-2 & Y_1 \\ \M\nu-1 & Y_2 \end{pmatrix}, \displaybreak[0] \\
        \delta_2(\la) &=
        (Y_1-Y_2)\det\begin{pmatrix}
            E_2^{-1} & 1 & \M\nu+0B_1\beta_0(a_0-\Msup\nu022) - (a_+-\Msup\nu+02)B_0\beta_1\M\nu02 \\
            E_1^{-1} & 1 & \M\nu+0B_1\beta_0(a_0-\Msup\nu012) - (a_+-\Msup\nu+02)B_0\beta_1\M\nu01 \\
            E_0^{-1} & 1 & \M\nu+0B_1\beta_0(a_0-\Msup\nu002) - (a_+-\Msup\nu+02)B_0\beta_1\M\nu00
        \end{pmatrix} \\
        &\hspace{5em} + Y_3\Delta(\la) - \det\begin{pmatrix} Y_4 & 1 & E_2 \\ Y_5 & 1 & E_1 \\ Y_6 & 1 & E_0 \end{pmatrix} \det\begin{pmatrix} 1 & \M\nu-2 & (a_--\Msup\nu-22) \\ 1 & \M\nu-1 & (a_--\Msup\nu-12) \\ 0 & B_1\M\nu+0 & B_0(a_+-\Msup\nu+02) \end{pmatrix}, \displaybreak[0] \\
        \delta_3(\la) &=
        (Y_1-Y_2)\beta_0\beta_1\det\begin{pmatrix}
            1 & \M\nu02 & (a_0-\Msup\nu022) \\
            1 & \M\nu01 & (a_0-\Msup\nu012) \\
            1 & \M\nu00 & (a_0-\Msup\nu002)
        \end{pmatrix} \\
        &\hspace{5em} + (\M\nu-2-\M\nu-1) \det\begin{pmatrix}
            Y_4E_2^{-1} & 1 & \beta_0(a_0-\Msup\nu022) + \beta_1(\M\nu-2+\M\nu-1)\M\nu02 \\
            Y_5E_1^{-1} & 1 & \beta_0(a_0-\Msup\nu012) + \beta_1(\M\nu-2+\M\nu-1)\M\nu01 \\
            Y_6E_0^{-1} & 1 & \beta_0(a_0-\Msup\nu002) + \beta_1(\M\nu-2+\M\nu-1)\M\nu00
        \end{pmatrix}, \displaybreak[0] \\
        \delta_4(\la) &=
        (Y_2-Y_1)\beta_0\beta_1\det\begin{pmatrix}
            E_2^{-1} & \M\nu02 & (a_0-\Msup\nu022) \\
            E_1^{-1} & \M\nu01 & (a_0-\Msup\nu012) \\
            E_0^{-1} & \M\nu00 & (a_0-\Msup\nu002)
        \end{pmatrix} \\
        &\hspace{5em} + (\M\nu-2-\M\nu-1) \det\begin{pmatrix}
            1 & Y_4 & \beta_1(\M\nu-2+\M\nu-1)\M\nu02E_2 + \beta_0(a_0-\Msup\nu022)E_2 \\
            1 & Y_5 & \beta_1(\M\nu-2+\M\nu-1)\M\nu01E_1 + \beta_0(a_0-\Msup\nu012)E_1 \\
            1 & Y_6 & \beta_1(\M\nu-2+\M\nu-1)\M\nu00E_0 + \beta_0(a_0-\Msup\nu002)E_0
        \end{pmatrix}, \displaybreak[0] \\
        \delta_5(\la) &=
        \begin{pmatrix}
            1 & E_2 & (\Msup\nu-22-\Msup\nu-12)Y_4 + (Y_2-Y_1)(a_0-\Msup\nu022)\beta_0E_2 \\
            1 & E_1 & (\Msup\nu-22-\Msup\nu-12)Y_5 + (Y_2-Y_1)(a_0-\Msup\nu012)\beta_0E_1 \\
            1 & E_0 & (\Msup\nu-22-\Msup\nu-12)Y_6 + (Y_2-Y_1)(a_0-\Msup\nu002)\beta_0E_0
        \end{pmatrix}, \displaybreak[0] \\
        \delta_6(\la) &=
        \begin{pmatrix}
            1 & E_2 & (Y_2-Y_1)\M\nu02E_2\beta_1 + (\M\nu-2-\M\nu-1)Y_4 \\
            1 & E_1 & (Y_2-Y_1)\M\nu01E_1\beta_1 + (\M\nu-2-\M\nu-1)Y_5 \\
            1 & E_0 & (Y_2-Y_1)\M\nu00E_0\beta_1 + (\M\nu-2-\M\nu-1)Y_6
        \end{pmatrix}.
    \end{align*}
\end{lem}

We aim to substitute the formulae afforded by lemma~\ref{lem:sourceLinearSystemSolution} into the Ehrenpreis form~\eqref{eqn:generalEFe} via equation~\eqref{eqn:generalDefnF}.
However, doing so does not immediately provide a useful solution representation, because our formulae for $F_e$ would all depend on $\hat u_-(\argdot;t),\hat u_+(\argdot;t),\hat u_0(\argdot;t)$ via formula~\eqref{eqn:sourceY} for $Y$.
We shall argue, via Jordan's lemma and Cauchy's theorem, that this dependence may be removed, but that requires an asymptotic analysis of $\Delta,\delta_k$.

By lemma~\ref{lem:generalNu} and equation~\eqref{eqn:sourceDelta}, as $\la\to\infty$,
\begin{align}
    \notag
    \Delta(\la)
    &= (\M\nu-2-\M\nu-1) \sum_{j=0}^2 E_j^{-1}(\M\nu0{j+1}-\M\nu0{j+2}) \left[\beta_1(\M\nu-2-\M\nu-1) - \beta_0(\M\nu0{j+1}+\M\nu0{j+2})\right]
    \\ \label{eqn:sourceDeltaAsymptotic}
    &= 3\la^3 \sum_{j=0}^2 \re^{\ri\alpha^j\la\eta}\alpha^j \left[ -\sqrt3\ri\beta_1 + \alpha^j\beta_0 \right]\left(1+\lindecayla\right).
\end{align}
In particular, as $\la\to\infty$ from within $\clos(D)$, the dominant term does not vanish, provided $\beta_0\neq\sqrt3\ri\beta_1$.
The zeros of $\Delta$ occur asymptotically on rays of argument $\frac\pi2+\frac{2j\pi}3$ for $j\in\{0,1,2\}$, which are all exterior to $D$.

\begin{lem} \label{lem:sourceDeformContour}
    Suppose $V_-\in\AC_{\mathrm p}^0(-\infty,0]$, $V_+\in\AC_{\mathrm p}^0[0,\infty)$ and $V_0\in\AC_{\mathrm p}^0[0,\eta]$.
    Denote by $\BVec{V}$ the list $(V_-,V_+,V_0)$.
    Suppose $H_0,H_1\in\AC_{\mathrm p}^0[0,T]$.
    Denote by $\BVec{H}$ the list $(H_0,H_1)$.
    For $\la\in\clos(D)$, let $\widehat{\BVec{V}}(\la)$ and $\widehat{\BVec{H}}(\la)$ be the vectors of transforms given by
    \begin{multline*}
        \widehat{\BVec{V}}(\la)
        =
        \left(
            \int_{-\infty}^0 \re^{-\ri\nu_-(\alpha^2\la) x} V_-(x) \D x, \quad
            \int_{-\infty}^0 \re^{-\ri\nu_-(\alpha\la) x} V_-(x) \D x, \quad
            \int_0^\infty \re^{-\ri\nu_+(\la) x} V_+(x) \D x, \right. \\ \left.
            \int_0^\eta \re^{-\ri\nu_0(\alpha^2\la) x} V_0(x) \D x, \quad
            \int_0^\eta \re^{-\ri\nu_0(\alpha\la) x} V_0(x) \D x, \quad
            \int_0^\eta \re^{-\ri\nu_0(\la) x} V_0(x) \D x
        \right).
    \end{multline*}
    and
    \begin{multline*}
        \widehat{\BVec{H}}(\la;t) = \left(0,0,0,
        \int_0^t\re^{\ri\la^3s}\left[\ri\M\nu02H_1(s) + (a_0-\Msup\nu022)H_0(s)\right]\D s,
        \right. \\ \left.
        \int_0^t\re^{\ri\la^3s}\left[\ri\M\nu01H_1(s) + (a_0-\Msup\nu012)H_0(s)\right]\D s,
        \int_0^t\re^{\ri\la^3s}\left[\ri\M\nu00H_1(s) + (a_0-\Msup\nu002)H_0(s)\right]\D s
        \right).
    \end{multline*}
    Define
    \begin{align*}
        \M\zeta-3(\la;\BVec V, \BVec H, t) &= \left[ \delta_1(\la) + \nu_-(\la)\delta_5(\la) + (a_--\nu_-(\la)^2)\delta_6(\la) \right] \nu'_-(\la), \\
        \M\zeta+1(\la;\BVec V, \BVec H, t) &= \left[ -\delta_2(\la) + B_1\nu_+(\alpha\la)\delta_5(\la) + B_0(a_+-\nu_+(\alpha\la)^2)\delta_6(\la) \right] \nu'_+(\alpha\la), \\
        \M\zeta+2(\la;\BVec V, \BVec H, t) &= \left[ -\delta_2(\la) + B_1\nu_+(\alpha^2\la)\delta_5(\la) + B_0(a_+-\nu_+(\alpha^2\la)^2)\delta_6(\la) \right] \nu'_+(\alpha^2\la), \\
        \M\zeta03(\la;\BVec V, \BVec H, t) &= \left[ \delta_4(\la) + \beta_1\nu_0(\la)\delta_5(\la) + B_0(a_0-\nu_0(\la)^2)\delta_6(\la) \right] \nu'_0(\la), \\
        \M\zeta01(\la;\BVec V, \BVec H, t) &= \left[-\delta_3(\la)  \left( - \ri\M\nu01 h_1(\la;t) - (a_0-\Msup\nu012) h_0(\la;t) \right) \Delta(\la) \right] \nu'_0(\alpha\la), \\
        \M\zeta02(\la;\BVec V, \BVec H, t) &= \left[-\delta_3(\la)  \left( - \ri\M\nu02 h_1(\la;t) - (a_0-\Msup\nu022) h_0(\la;t) \right) \Delta(\la) \right] \nu'_0(\alpha^2\la),
    \end{align*}
    for $\delta_k$ the determinant of matrix $\mathcal A$ given by equation~\eqref{eqn:sourceA}, except with the $k$\textsuperscript{th} column replaced by $\widehat{\BVec{H}}(\la;t)-\widehat{\BVec{V}}(\la)$.
    Then, as $\la\to\infty$ within $\clos(D)$, $\M\zeta ej(\la;\BVec V, 0, 0)/\Delta(\la)=\lindecayla$ and they are also all continuous functions on $\clos(D)$ and analytic in $D$.
    The functions $\M\zeta ej(\la;0, \BVec H, t)/\Delta(\la)$ are also all continuous functions on $\clos(D)$ and analytic in $D$.
    Moreover,
    \begin{subequations} \label{eqn:sourceDeformContour.IntegralV0}
    \begin{align}
        \label{eqn:sourceDeformContour.IntegralV0.--}
        &\forall\;x<0, & 0 &= \int_{\partial D} \re^{\ri\nu_-(\la)x} \frac{\M\zeta-3(\la;\BVec V,0,0)}{\Delta(\la)} \D \la, \\
        \label{eqn:sourceDeformContour.IntegralV0.++}
        &\forall\;x>0, & 0 &= \int_{\partial D} \frac{\alpha\re^{\ri\nu_+(\alpha\la)x} \M\zeta+1(\la;\BVec V,0,0) + \alpha^2\re^{\ri\nu_+(\alpha^2\la)x} \M\zeta+2(\la;\BVec V,0,0)}{\Delta(\la)} \D \la, \\
        \label{eqn:sourceDeformContour.IntegralV0.0-}
        &\forall\;x<\eta, & 0 &= \int_{\partial D} \re^{\ri\nu_0(\la)(x-\eta)} \frac{\M\zeta03(\la;\BVec V,0,0)}{\Delta(\la)} \D \la, \\
        \label{eqn:sourceDeformContour.IntegralV0.0+}
        &\forall\;x>0, & 0 &= \int_{\partial D} \frac{\alpha\re^{\ri\nu_0(\alpha\la)x} \M\zeta01(\la;\BVec V,0,0) + \alpha^2\re^{\ri\nu_0(\alpha^2\la)x} \M\zeta02(\la;\BVec V,0,0)}{\Delta(\la)} \D \la,
    \end{align}
    \end{subequations}
    and, provided $T\geq t$, and $j\in\{1,2\}$,
    \begin{subequations} \label{eqn:sourceDeformContour.IntegralH0}
    \begin{align}
        \label{eqn:sourceDeformContour.IntegralH0.--}
        &\forall\;x<0, & 0 &= \int_{\partial D} \re^{\ri\nu_-(\la)x-\ri\la^3t} \frac{\M\zeta-3(\la;0,\BVec H, T)-\M\zeta-3(\la;0,\BVec H, t)}{\Delta(\la)} \D \la, \\
        \label{eqn:sourceDeformContour.IntegralH0.++}
        &\forall\;x>0, & 0 &= \int_{\partial D} \alpha^j\re^{\ri\nu_+(\alpha^j\la)x-\ri\la^3t} \frac{\M\zeta+j(\la;0,\BVec H, T)-\M\zeta+j(\la;0,\BVec H, t)}{\Delta(\la)} \D \la, \\
        \label{eqn:sourceDeformContour.IntegralH0.0-}
        &\forall\;x<\eta, & 0 &= \int_{\partial D} \re^{\ri\nu_0(\la)(x-\eta)-\ri\la^3t} \frac{\M\zeta03(\la;0,\BVec H, T)-\M\zeta03(\la;0,\BVec H, t)}{\Delta(\la)} \D \la, \\
        \label{eqn:sourceDeformContour.IntegralH0.0+}
        &\forall\;x>0, & 0 &= \int_{\partial D} \alpha^j\re^{\ri\nu_+(\alpha^j\la)x-\ri\la^3t}\frac{\M\zeta0j(\la;0,\BVec H, T)-\M\zeta0j(\la;0,\BVec H, t)}{\Delta(\la)} \D \la.
    \end{align}
    \end{subequations}
\end{lem}

Equations~\eqref{eqn:sourceLinearSystem} and~\eqref{eqn:sourceY} imply that, if
\begin{equation*}
    \BVec v = (u_-(\argdot,t),u_+(\argdot,t),u_0(\argdot,t)),
    \qquad
    \BVec w = (U_-,U_+,U_0),
    \qquad
    \BVec h = (h_0,h_1),
\end{equation*}
then
\begin{align*}
    F_-(        \la;0,t) &= \re^{\ri\la^3t}\frac{\M\zeta-3(\la;\BVec v,0,0)}{\Delta(\la)} - \frac{\M\zeta-3(\la;\BVec w,0,0)}{\Delta(\la)} + \frac{\M\zeta-3(\la;0, \BVec h,t)}{\Delta(\la)}, \\
    F_+(\alpha^j\la;0,t) &= \re^{\ri\la^3t}\frac{\M\zeta+j(\la;\BVec v,0,0)}{\Delta(\la)} - \frac{\M\zeta+j(\la;\BVec w,0,0)}{\Delta(\la)} + \frac{\M\zeta+j(\la;0, \BVec h,t)}{\Delta(\la)}, & j &\in\{1,2\}, \\
    % F_+(\alpha^2\la;0,t) &= \re^{\ri\la^3t}\frac{\M\zeta+2(\la;\BVec v,0,0)}{\Delta(\la)} - \frac{\M\zeta+2(\la;\BVec w,0,0)}{\Delta(\la)} + \frac{\M\zeta+2(\la;0, \BVec h,t)}{\Delta(\la)}, \\
    % F_+(\alpha  \la;0,t) &= \re^{\ri\la^3t}\frac{\M\zeta+1(\la;\BVec v,0,0)}{\Delta(\la)} - \frac{\M\zeta+1(\la;\BVec w,0,0)}{\Delta(\la)} + \frac{\M\zeta+1(\la;0, \BVec h,t)}{\Delta(\la)}, \\
    F_0(        \la;0,t) &= \re^{\ri\la^3t}\frac{\M\zeta03(\la;\BVec v,0,0)}{\Delta(\la)} - \frac{\M\zeta03(\la;\BVec w,0,0)}{\Delta(\la)} + \frac{\M\zeta03(\la;0, \BVec h,t)}{\Delta(\la)}, \\
    F_0(\alpha^j\la;0,t) &= \re^{\ri\la^3t}\frac{\M\zeta0j(\la;\BVec v,0,0)}{\Delta(\la)} - \frac{\M\zeta0j(\la;\BVec w,0,0)}{\Delta(\la)} + \frac{\M\zeta0j(\la;0, \BVec h,t)}{\Delta(\la)}, & j &\in\{1,2\}.
\end{align*}
Substituting these formulae into the Ehrenpreis form~\eqref{eqn:generalEFe}, we obtain solution representations~\eqref{eqn:sourceSoln}, except with the additional terms corresponding to $\BVec v$.
However, by equations~\eqref{eqn:sourceDeformContour.IntegralV0} of lemma~\ref{lem:sourceDeformContour}, these terms evaluate to $0$.
By equations~\eqref{eqn:sourceDeformContour.IntegralH0} of lemma~\ref{lem:sourceDeformContour}, each of the $\M\zeta ej(\la;0, \BVec h,t)$ may be replaced by $\M\zeta ej(\la;0, \BVec h,T)$ without affecting the result.
This is sufficient to justify the ``$h_k(\argdot;t)$ may be replaced with $h_k(\argdot,T)$'' claim in all except the terms of the second integral of equation~\eqref{eqn:sourceSoln.0} whose dependence on $h_k(\la;t)$ is explicitly stated.
The $t\mapsto T$ invariance for those terms follows from the fact that,
\[
    \re^{-\ri\la^3t}\left[- \ri\M\nu0j \big\{h_1(\la;T)-h_1(\la;t)\big\} - (a_0-\Msup\nu0j2) \big\{h_0(\la;T)-h_0(\la;t)\big\}\right] = \lindecayla
\]
as $\la\to\infty$ in $\clos(D)$ and another application of Jordan's lemma.
This completes the proof of proposition~\ref{prop:sourceSolution}.

\begin{proof}[Proof of lemma~\ref{lem:sourceDeformContour}]
    The proof of the asymptotic formula $\M\zeta ej(\la;\BVec V,0,0)/\Delta(\la)=\lindecayla$ follows exactly the proof of lemma~\ref{lem:loopDeformContour}, except that the formulae for $\delta_k$ in lemma~\ref{lem:sourceLinearSystemSolution} replace those from lemma~\ref{lem:loopLinearSystemSolution}, and the asymptotic analysis of $\Delta$ and the locus of its zeros is as presented in this section.
    The argument for equations~\eqref{eqn:sourceDeformContour.IntegralV0} is then also like that in the proof of lemma~\ref{lem:loopDeformContour}.

    We can consider $\re^{-\ri\la^3t}[\M\zeta ej(\la;0,\BVec H, T)-\M\zeta ej(\la;0,\BVec H, t)]$ as being given by the formula for $\M\zeta ej$ where the $\delta_k$ are the determinants of $\mathcal A$ with the $k$\textsuperscript{th} column replaced by
    \begin{multline*}
        \left(0,0,0,
        \int_0^{T-t}\re^{\ri\la^3s}\left[\ri\M\nu02H_1(s-t) + (a_0-\Msup\nu022)H_0(s-t)\right]\D s,
        \right. \\ \left.
        \int_0^{T-t}\re^{\ri\la^3s}\left[\ri\M\nu01H_1(s-t) + (a_0-\Msup\nu012)H_0(s-t)\right]\D s,
        \right. \\ \left.
        \int_0^{T-t}\re^{\ri\la^3s}\left[\ri\M\nu00H_1(s-t) + (a_0-\Msup\nu002)H_0(s-t)\right]\D s
        \right).
    \end{multline*}
    By lemma~\ref{lem:FTLeadingOrder} under $\la\mapsto\la^3$, each entry in the above vector is $\lindecayla$ as $\la\to\infty$ in $\clos D$.
    In this case,
    \[
        \delta_3(\la) = (\M\nu-2-\M\nu-1) \det
        \begin{pmatrix}
            \lindecayla & E_2 & \beta_0(a_0-\Msup\nu022) + \beta_1(\M\nu-2+\M\nu-1)\M\nu02 E_2 \\
            \lindecayla & E_1 & \beta_0(a_0-\Msup\nu012) + \beta_1(\M\nu-2+\M\nu-1)\M\nu01 E_1 \\
            \lindecayla & E_0 & \beta_0(a_0-\Msup\nu002) + \beta_1(\M\nu-2+\M\nu-1)\M\nu00 E_0
        \end{pmatrix}
        =
        \bigoh{\la^2\re^{\ri\eta\la}},
    \]
    and $\delta_1(\la),\delta_2(\la),\delta_4(\la)=\bigoh{\la^2\re^{\ri\eta\la}}$ also.
    Similarly, $\delta_5=\bigoh{\la\re^{\ri\eta\la}}$ and $\delta_6=\bigoh{\re^{\ri\eta\la}}$, so
    \[
        \re^{-\ri\la^3t}\frac{\M\zeta ej(\la;0,\BVec H, T)-\M\zeta ej(\la;0,\BVec H, t)}{\Delta(\la)} = \lindecayla.
    \]
    Equations~\eqref{eqn:sourceDeformContour.IntegralH0} follow by Jordan's lemma.
\end{proof}

\subsection{Solution satisfies the problem: proof of proposition~\protect\ref{prop:sourceExistence}} \label{ssec:source.existence}

The following lemma allows us to write the data dependent parts of the integrands of the various integrals over $\partial D$ as a sum of three terms.
One part, labelled $\phi$, will have its integral deformed from $\partial D$ outwards to $\gamma$.
The integrals of another part, labelled $\varphi$, will be deformed inwards to $\gamma$.
The final part, which has the symbol $\psi$, has faster decay on $\partial D$, so we leave its integrals on $\partial D$ itself.
The actual contour deformations are justified in the general lemmata~\ref{lem:vertexRewritingLemma.D3Stronger} and~\ref{lem:ConvergenceD3IntegralsStronger}; the purpose of lemma~\ref{lem:sourceConvergenceZetaLeadingOrder} is to justify their applicability to this particular problem.
Broadly, the argument is similar in structure to that in~\S\ref{ssec:loop.existence}, the difference being that in this case we have the extra term $\varphi$ which arises from the inhomogeneous vertex conditions.

\begin{lem} \label{lem:sourceConvergenceZetaLeadingOrder}
    Suppose $U_-,U_+,U_0,H_0,H_1$, are as in the hypothesis of proposition~\ref{prop:sourceExistence} and denote by $\BVec{v}$ and $\BVec{h}$ the lists $(U_-,U_+,U_0)$ and $(H_0,H_1)$.
    For $\la\in\partial D$, let
    $\M\zeta ej$ be defined as in the statement of lemma~\ref{lem:sourceDeformContour}, but in terms of these $\BVec{v},\BVec{h}$.
    Let $\tau'<\tau$ be consecutive points in the union of the partition sets of $H_0,H_1$ such that $\tau'<t\leq\tau$; see definition~\ref{defn:piecewiseACfunctions}.
    Then,
    \[
        \M\zeta ej(\la;\BVec v,\BVec h, T)/\Delta(\la) = \re^{\ri\la^3\tau'}\M\phi ej(\la) + \re^{\ri\la^3\tau}\M\varphi ej(\la) + \M\psi ej(\la),
    \]
    with $\M\psi ej(\la) = \bigoh{\la^{-5}}$ as $\la\to\infty$ along $\partial D$ and $\M\phi ej$, respectively $\M\varphi ej$, is continuously extensible to $S$, respectively $s$, and analytically to its interior, obeying
    \[
        \abs{\M\phi ej(\la)}=\bigoh{\la^{-1}},
        \qquad\mbox{respectively}\qquad
        \abs{\M\varphi ej(\la)}=\bigoh{\la^{-1}},
    \]
    as $\la\to\infty$ within $S$, respectively $s$.
\end{lem}

\begin{proof}[Proof of lemma~\ref{lem:sourceConvergenceZetaLeadingOrder}]
    By definition, $\M\zeta ej(\la;\BVec v, \BVec h, T) = \M\zeta ej(\la;\BVec v, 0, 0) + \M\zeta ej(\la;0, \BVec h, T)$, so we may give the proof separately for each part.
    For the $\BVec v$ dependent part, the proof essentially follows the part of the proof of lemma~\ref{lem:sourceDeformContour} not concerned with equations~\eqref{eqn:sourceDeformContour.IntegralV0}--\eqref{eqn:sourceDeformContour.IntegralH0}, except that the improved regularity of the initial data allows an enhancement to the application of lemma~\ref{lem:FTLeadingOrder}: now $N=4$.
    This argument can be made for $\tau'=0$, but it also holds for the $\tau'$ specified.
    Henceforth, we may assume $\BVec v=0$.

    Let $T'$ be the union of the partition sets of $H_0$ and $H_1$.
    For notational convenience, we define $H_\ell(0^-),H_\ell(T^+):=0$.
    Integration by parts thrice establishes that
    \begin{multline} \label{eqn:sourceConvergenceZetaLeadingOrder.proof1}
        \int_0^T \re^{\ri\la^3s} H_\ell(s) \D s
        =
        \sum_{\sigma\in T'} \sum_{k=1}^3 (\ri\la)^{-3k}
        \re^{\ri\la^3\sigma}
        \left[
            H^{(k-1)}_\ell(\sigma^+) - H^{(k-1)}_\ell(\sigma^-)
        \right]
        \\
        - \ri \la^{-9}
        \sum_{\substack{\tau_1,\tau_2\in T'\\\text{consecutive}}} \int_{\tau_1}^{\tau_2} \re^{\ri\la^3s} H'''_\ell (s) \D s.
    \end{multline}
    For $\la$ in either ray of $\partial D$, the integrals on the right all converge absolutely and their sum is dominated by
    \[
        \int_0^T \abs{H'''_\ell(s)}\D s,
    \]
    in which $H'''_\ell$ represents the piecewise (not distributional) derivative.
    Therefore, the whole last term is $\bigoh{\la^{-9}}$ as $\la\to\infty$ along $\partial D$.
    In the construction of $\M\zeta ej(\la;0, \BVec h, T)/\Delta(\la)$ from the $\ell\in\{0,1\}$ cases of the integral on the left of equation~\eqref{eqn:sourceConvergenceZetaLeadingOrder.proof1}, it is multiplied by a factor which is $\bigoh{\la^2}$, hence overall the final term on the right of equation~\eqref{eqn:sourceConvergenceZetaLeadingOrder.proof1} contributes a $\bigoh{\la^{-7}}$ term, which can be absorbed into $\M\psi ej$.

    We partition $T'$ into $T_<,T_>$ so that
    \[
        T_\lessgtr = \{\sigma\in T': \sigma \lessgtr(\tau'+\tau)/2\},
        \qquad\mbox{so}\qquad
        \max T_<=\tau'<t,
        \qquad\mbox{and}\qquad
        \min T_>=\tau\geq t.
    \]
    The first double sum on the right of equation~\eqref{eqn:sourceConvergenceZetaLeadingOrder.proof1} can be expressed as
    \begin{multline*}
        \re^{\ri\la^3\tau'}\sum_{\sigma\in T_<} \sum_{k=1}^3 (\ri\la)^{-3k}
        \re^{-\ri\la^3(\tau'-\sigma)}
        \left[
            H^{(k-1)}_\ell(\sigma^+) - H^{(k-1)}_\ell(\sigma^-)
        \right]
        \\
        +
        \re^{\ri\la^3\tau}\sum_{\sigma\in T_>} \sum_{k=1}^3 (\ri\la)^{-3k}
        \re^{\ri\la^3(\sigma-\tau)}
        \left[
            H^{(k-1)}_\ell(\sigma^+) - H^{(k-1)}_\ell(\sigma^-)
        \right].
    \end{multline*}
    Ignoring the $\re^{\ri\la^3\tau}$ factor, the latter double sum is $\bigoh{\la^{-3}}$ in $s$, so it contributes to the ratio $\M\zeta ej(\la;0, \BVec h, T)/\Delta(\la)$ a term which is $\lindecayla$, and may be absorbed into $\M\varphi ej$.
    Similarly, the former double sum contributes to $\M\zeta ej(\la;0, \BVec h, T)/\Delta(\la)$ a term which can be accommodated in $\M\phi ej$.
\end{proof}

\begin{proof}[Proof that each $u_e$ satisfies its PDE]
    Suppose $(x,t)$ is in the interior of the domain of $u_e$.
    Then there exist $\xi>0$ and $\tau',\tau\geq0$ such that $x\in(-\infty,-\xi)$ or $(\xi,\infty)$ and $t\in(\tau',\tau]$ with $\tau',\tau$ being consecutive points in the union of the partition sets of $H_0,H_1$.
    Therefore, by lemmata~\ref{lem:ConvergenceD3IntegralsStronger} and~\ref{lem:sourceConvergenceZetaLeadingOrder}, we can rewrite each integral on the right of equations~\eqref{eqn:sourceSoln} in such a way that we may differentiate under the integral thrice in $x$ or once in $t$.
    The result~\eqref{eqn:generalproblem.PDE} follows from the definition of $\nu$.
\end{proof}

\begin{proof}[Proof that each $u_e$ satisfies its initial condition]
    The argument from~\S\ref{ssec:loop.existence} requires no modification for this problem.
\end{proof}

\begin{proof}[Proof that the $u_e$ satisfy the vertex conditions almost everywhere]
    In this argument, so that we may successfully apply lemma~\ref{lem:vertexRewritingLemma.D3Stronger}, we shall have to assume not only that $t>0$, but also that $t$ is not a point in the union of the partition sets of $H_0$ and $H_1$.
    We aim to show, for $k\in\{0,1\}$,
    \begin{subequations} \label{eqn:sourceVertex.toProve}
    \begin{align}
        \label{eqn:sourceVertex.toProve.beta}
        (-\ri)^k \beta_k \lim_{x\to0^-} 2\pi \partial_x^k u_-(x,t) - (-\ri)^k \lim_{x\to\eta^-} 2\pi \partial_x^k u_0(x,t) &= 0, \\
        \label{eqn:sourceVertex.toProve.B}
        (-\ri)^k B_k \lim_{x\to0^-} 2\pi \partial_x^k u_-(x,t) - (-\ri)^k \lim_{x\to0^+} 2\pi \partial_x^k u_+(x,t) &= 0, \\
        \label{eqn:sourceVertex.toProve.H}
        \lim_{x\to0^+} \partial_x^k u_0(x,t) &= H_k(t).
    \end{align}
    \end{subequations}
    Most of the argument is very similar to the corresponding proof in~\S\ref{ssec:loop.existence}.
    Throughout, we use notation
    \[
        (\M\omega je(\la),\M\Omega je(\la)) =
        \begin{cases}
            (\M\psi je(\la),\M\Psi je(\la)) & \la \in \partial D, \\
            (\re^{\ri\la^3\tau'} \M\phi    je(\la),\re^{\ri\la^3\tau'} \M\Phi    je(\la)) & \la \in \Gamma, \\
            (\re^{\ri\la^3\tau}  \M\varphi je(\la),\re^{\ri\la^3\tau}  \M{\widehat\Phi} je(\la)) & \la \in \gamma.
        \end{cases}
    \]
    % for $\in$ representing ``is a variable of integration in an integral along the contour''.
    These functions are defined by
    \begin{align*}
        \M\omega 3-(\la) &= \Omega_1 + \M\nu-0\Omega_5 + (a_--\Msup\nu-02)\Omega_6,
        \\
        \M\omega j+(\la) &= -\Omega_2 + B_1\M\nu+j\Omega_5 + B_0(a_+-\Msup\nu+j2)\Omega_6
        & j &\in \{1,2\}, \\
        \M\omega j0(\la) &= -\Omega_3 + \M\nu0j\ri h_1 + (a_0-\Msup\nu0j2)h_0
        & j &\in \{1,2\}, \\
        \M\omega 30(\la) &= \Omega_4 + \beta_1\M\nu00\Omega_5 + \beta_0(a_0-\Msup\nu002)\Omega_6,
    \end{align*}
    with
    \begin{equation*}
        \mathcal A
        \begin{pmatrix}
            \Omega_1 \\ \Omega_2 \\ \Omega_3 \\ \Omega_4 \\ \Omega_5 \\ \Omega_6
        \end{pmatrix}
        =
        \begin{pmatrix}
            -\M\omega\RR-(\alpha^2\la) \\
            -\M\omega\RR-(\alpha\la) \\
            -\M\omega\RR+(\la) \\
            -\M\omega\RR0(\alpha^2\la)
            + \omega_{\mathrm d}(\alpha^2\la) \\
            -\M\omega\RR0(\alpha\la)
            + \omega_{\mathrm d}(\alpha  \la) \\
            -\M\omega\RR0(\la)
            + \omega_{\mathrm d}(        \la)
        \end{pmatrix},
    \end{equation*}
    and, for $\la\in\partial D$,
    \begin{align*}
        \widehat U_e(\nu_e(\la)) &= \left[ \M\psi\RR e(\la) + \re^{\ri\la^3\tau'}\M\phi\RR e(\la) + \re^{\ri\la^3\tau}\M\varphi\RR e(\la) \right] \nu'_e(\la), \\
        h_k(\la) &= \left[ \M\psi{\mathrm d}k(\la) + \re^{\ri\la^3\tau'}\M\phi{\mathrm d}k(\la) + \re^{\ri\la^3\tau}\M\varphi{\mathrm d}k(\la) \right] \nu'_0(\la),
    \end{align*}
    in which the functions $\M\psi\RR e,\M\phi\RR e, \M\varphi\RR e$ have the desired decay properties, as justified by lemma~\ref{lem:sourceConvergenceZetaLeadingOrder}.

    By lemmata~\ref{lem:vertexRewritingLemma.Real-},~\ref{lem:vertexRewritingLemma.Real+}, and~\ref{lem:vertexRewritingLemma.D3Stronger}
    the left side of equation~\eqref{eqn:sourceVertex.toProve.beta} reduces to
    \begin{multline*}
        \PV\,\left\{ \int_{\partial D} + \int_{\Gamma} + \int_{\gamma} \right\} \re^{-\ri\la^3t} %\\
        \bigg[
            \beta_k\bigg(
                % \alpha\M\omega\RR-(\alpha\la)\Msup\nu-1k\Msup\nu-1\prime
                % + \alpha^2\M\omega\RR-(\alpha^2\la)\Msup\nu-2k\Msup\nu-2\prime
                - \M\omega3-(\la)\Msup\nu-0k\Msup\nu-0\prime
                + \sum_{j=1}^2\alpha^j\M\omega\RR-(\alpha^j\la)\Msup\nu-jk\Msup\nu-j\prime
            \bigg) \\
            + \M\omega30(\la)\Msup\nu00k\Msup\nu00\prime
            % + \alpha   E_1^{-1} \left( \M\omega10(\la) - \M\omega\RR0(\alpha  \la) \right) \Msup\nu01k\Msup\nu01\prime
            % + \alpha^2 E_2^{-1} \left( \M\omega20(\la) - \M\omega\RR0(\alpha^2\la) \right) \Msup\nu02k\Msup\nu02\prime
            + \sum_{j=1}^2\alpha^j E_j^{-1} \left(
                \M\omega j0(\la)
                + \M\nu0j\ri\M\omega{\mathrm d}1(\alpha^j\la)
                + (a_0-\Msup\nu0j2)\M\omega{\mathrm d}0(\alpha^j\la)
                - \M\omega\RR0(\alpha^j\la)
            \right) \Msup\nu0jk\Msup\nu0j\prime
        \bigg]
        \D\la
    \end{multline*}
    By the above construction, the bracket evaluates to
    \[
        \begin{pmatrix}
            \Omega_1 \\ \Omega_2 \\ \Omega_3 \\ \Omega_4 \\ \Omega_5 \\ \Omega_6 \\ \M\omega{\mathrm d}0 \\ \ri \M\omega{\mathrm d}1
        \end{pmatrix}
        \cdot
        \begin{pmatrix}
            -\beta_k\sum_{j=0}^2 \alpha^j \Msup\nu-jk\Msup\nu-j\prime \\
            0 \\
            (1-1)\sum_{j=1}^2 \alpha^j E_j^{-1} \Msup\nu0jk\Msup\nu0j\prime \\
            \sum_{j=0}^2 \alpha^j \Msup\nu0jk\Msup\nu0j\prime \\
            -\beta_k\sum_{j=0}^2 \alpha^j \Msup\nu-j{k+1}\Msup\nu-j\prime + \beta_1\sum_{j=0}^2 \alpha^j \Msup\nu0j{k+1}\Msup\nu0j\prime \\
            -\beta_k\sum_{j=0}^2 \alpha^j(a_0-\Msup\nu-j2)\Msup\nu-jk\Msup\nu-j\prime + \beta_0\sum_{j=0}^2 \alpha^j(a_0-\Msup\nu0j2)\Msup\nu0jk\Msup\nu0j\prime \\
            (1-1)\sum_{j=1}^2 \alpha^j E_j^{-1} (a_0-\Msup\nu0j2) \Msup\nu0jk\Msup\nu0j\prime \\
            (1-1)\sum_{j=1}^2 \alpha^j E_j^{-1} \Msup\nu0j{k+1}\Msup\nu0j\prime
        \end{pmatrix}.
    \]
    By lemma~\ref{lem:generalNu}, the coefficients each evaluate to $0$, justifying equation~\eqref{eqn:sourceVertex.toProve.beta}.
    A similar calculation yields equation~\eqref{eqn:sourceVertex.toProve.B}.
    Another similar calculation for the left side of equation~\eqref{eqn:sourceVertex.toProve.H} may be carried out, but in this case the coefficients are not all zero.
    Rather, the left side of equation~\eqref{eqn:sourceVertex.toProve.H} simplifies to
    \begin{equation*}
        \frac1{2\pi} \PV\,\left\{ \int_{\partial D} + \int_{\Gamma} + \int_{\gamma} \right\} \re^{-\ri\la^3t}
        \bigg[
            3\la^2\M\omega{\mathrm d}k
        \bigg]
        \D\la
        =
        \frac1{2\pi} \PV \int_{\partial D} \re^{-\ri\la^3t}
        \bigg[
            3\la^2h_k(\la)
        \bigg]
        \D\la;
    \end{equation*}
    the latter equality follows by Jordan's lemma and Cauchy's theorem.
    Making the change of variables $\mu=-\la^3$ and a small contour deformation close to $0$, this simplifies to
    \[
        \frac1{2\pi} \PV \int_{-\infty}^\infty \re^{\ri\mu t}
            \int_0^T \re^{-\ri\mu s} H_k(s) \D s
        \D\mu.
    \]
    The Fourier inversion theorem completes the proof.
\end{proof}

\section{Sink defect} \label{sec:sink}

Suppose that, attached to the vertex where the incoming and outgoing leads meet, there is a single bond but, in contrast to the source problem of~\S\ref{sec:source}, here the bond is attached at its $x=0$ end, and detached at its $x=\eta$ end.
On such a metric graph, as displayed in figure~\ref{fig:graph-sink},
\begin{figure}
    \centering
    \includegraphics{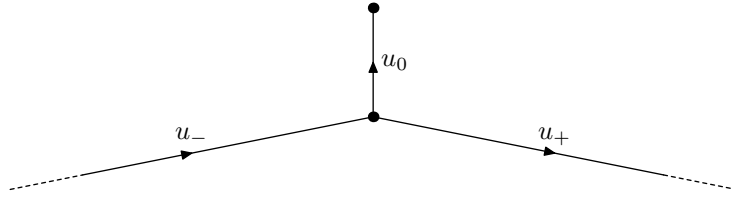}
    \caption{The metric graph domain with a sink defect.}
    \label{fig:graph-sink}
\end{figure}
we study problem~\eqref{eqn:generalproblem} together with vertex conditions
\begin{subequations} \label{eqn:sinkVC}
\begin{align}
    \label{eqn:sinkVC1}
    u_-(0,t) &= b_0^{-1} u_0(0,t) = B_0^{-1} u_+(0,t), & t &\in [0,T], \\
    \label{eqn:sinkVC2}
    \partial_x u_-(0,t) &= b_1^{-1} \partial_x u_0(0,t) = B_1^{-1} \partial_x u_+(0,t), & t &\in [0,T], \\
    \label{eqn:sinkVC3}
    \partial_{xx} u_-(0,t) &= b_2 \partial_{xx} u_0(0,t) + B_2 \partial_{xx} u_+(0,t), & t &\in [0,T], \\
    \label{eqn:sinkBC}
    u_0(\eta,t) &= H(t), & t &\in [0,T],
\end{align}
\end{subequations}
with coefficients $b_k,B_k\in\CC\setminus\{0\}$ and known function $H$ of regularity to be specified later.

We see this problem representing a ``sink'' type defect because information may flow out of the system through the bond.

\begin{prop} \label{prop:sinkSolution}
    Consider problem~\eqref{eqn:generalproblem} on the metric graph described above, with vertex conditions~\eqref{eqn:sinkVC} for coefficients satisfying
    \(
        b_2(b_0+b_1) - B_2(B_0+B_1) \neq 1.
    \)
    Define, for $e\in\{+,-,0\}$ and $j\in\{0,1,2\}$, $\M\nu ej = \nu_e(\alpha^j\la)$, $E_j = \exp\left(-\ri\eta\M\nu0j\right)$, and
    \begin{equation}  \label{eqn:sinkA}
        \mathcal A(\la) =
        \begin{pmatrix}
            B_2 & b_2 & \M\nu-2     &  (a_--\Msup\nu-22)    &  0    &  0 \\
            B_2 & b_2 & \M\nu-1     &  (a_--\Msup\nu-12)    &  0    &  0 \\
            -1  & 0   & -B_1\M\nu+0 & -(a_+-\Msup\nu+02)B_0 & 0 & 0 \\
            0   & -1  & -b_1\M\nu02 &  (a_0-\Msup\nu022)b_0 & E_2 & \M\nu02E_2 \\
            0   & -1  & -b_1\M\nu01 &  (a_0-\Msup\nu012)b_0 & E_1 & \M\nu01E_1 \\
            0   & -1  & -b_1\M\nu00 &  (a_0-\Msup\nu002)b_0 & E_0 & \M\nu00E_0
        \end{pmatrix}.
    \end{equation}
    Define also $h(\la;t) = \int_0^t \re^{\ri\la^3s}H(s)\D s$.
    Let $\Delta=\det\mathcal A$, and $\delta_k$ be the determinant of $\mathcal A$ but with the $k$\textsuperscript{th} column replaced by
    \begin{multline} \label{eqn:sinkSolnY}
        -
        \left(
            0,\;0,\;0,\;
            (a_0-\Msup\nu022) E_2 h(\la), \;
            (a_0-\Msup\nu012) E_1 h(\la), \;
            (a_0-\Msup\nu002) E_0 h(\la)
        \right)
        \\
        -
        \left(
            \widehat U_-(\M\nu-2), \;
            \widehat U_-(\M\nu-1), \;
            \widehat U_+(\M\nu-0), \;
            \widehat U_0(\M\nu02), \;
            \widehat U_0(\M\nu01), \;
            \widehat U_0(\M\nu00)
        \right).
    \end{multline}
    Suppose $u$ is sufficiently smooth and satisfies this problem.
    Then, for $R$ sufficiently large,
    \begin{subequations} \label{eqn:sinkSoln}
    \begin{align}
    \notag
        2\pi u_-(x,t)
        &= \int_{-\infty}^\infty \re^{\ri\la x - \ri(\la^3-a_-\la)t} \widehat U_-(\la) \D \la \\
    \label{eqn:sinkSoln.-} \tag{\theparentequation.\(-\)}
        &\hspace{1.2em} - \int_{\partial D} \re^{\ri\nu_-(\la)x-\ri\la^3t} \left[ \frac{B_2\delta_1(\la) + b_2\delta_2(\la) + \nu_-(\la)\delta_3(\la) + (a_--\nu_-(\la)^2)\delta_4(\la)}{\Delta(\la)} \right] \nu'_-(\la) \D\la, \displaybreak[0] \\
    \notag
        2\pi u_+(x,t)
        &= \int_{-\infty}^\infty \re^{\ri\la x - \ri(\la^3-a_+\la)t} \widehat U_+(\la) \D \la \\
    \notag
        &\hspace{3em} - \int_{\partial D} \frac{\re^{-\ri\la^3t}}{\Delta(\la)} \Bigg[
            \delta_1(\la) \left( \alpha \re^{\ri\nu_+(\alpha\la)x} \nu'_+(\alpha\la) + \alpha^2 \re^{\ri\nu_+(\alpha^2\la)x} \nu'_+(\alpha^2\la) \right) \\
    \notag
            &\hspace{6em} + B_1\delta_3(\la) \left( \alpha \nu_+(\alpha\la)\re^{\ri\nu_+(\alpha\la)x} \nu'_+(\alpha\la) + \alpha^2 \nu_+(\alpha^2\la) \re^{\ri\nu_+(\alpha^2\la)x} \nu'_+(\alpha^2\la) \right) \\
    \label{eqn:sinkSoln.+} \tag{\theparentequation.\(+\)}
            &\hspace{-4em} + B_0\delta_4(\la) \left( \alpha \left(a_+-\nu_+(\alpha\la)^2\right) \re^{\ri\nu_+(\alpha\la)x} \nu'_+(\alpha\la) + \alpha^2 \left(a_+-\nu_+(\alpha^2\la)^2\right) \re^{\ri\nu_+(\alpha^2\la)x} \nu'_+(\alpha^2\la) \right)
        \Bigg] \D\la, \displaybreak[0] \\
    \notag
        2\pi u_0(x,t)
        &= \int_{-\infty}^\infty \re^{\ri\la x - \ri(\la^3-a_0\la)t} \widehat U_0(\la) \D \la \\
    \notag
        &\hspace{3em} - \int_{\partial D} \frac{\re^{-\ri\la^3t}}{\Delta(\la)} \Bigg[
            \delta_2(\la) \left( \alpha \re^{\ri\nu_0(\alpha\la)x} \nu'_0(\alpha\la) + \alpha^2 \re^{\ri\nu_0(\alpha^2\la)x} \nu'_0(\alpha^2\la) \right) \\
    \notag
            &\hspace{6em} + b_1\delta_3(\la) \left( \alpha \nu_0(\alpha\la)\re^{\ri\nu_0(\alpha\la)x} \nu'_0(\alpha\la) + \alpha^2 \nu_0(\alpha^2\la) \re^{\ri\nu_0(\alpha^2\la)x} \nu'_0(\alpha^2\la) \right) \\
    \notag
            &\hspace{-1.5em} + b_0\delta_4(\la) \left( \alpha \left(a_0-\nu_0(\alpha\la)^2\right) \re^{\ri\nu_0(\alpha\la)x} \nu'_0(\alpha\la) + \alpha^2 \left(a_0-\nu_0(\alpha^2\la)^2\right) \re^{\ri\nu_0(\alpha^2\la)x} \nu'_0(\alpha^2\la) \right)
        \Bigg] \D\la \\
    \label{eqn:sinkSoln.0} \tag{\theparentequation.\(0\)}
        &\hspace{-1em} - \int_{\partial D} \re^{\ri\nu_0(\la)(x-\eta)-\ri\la^3t} \left[ \frac{\delta_5(\la) + \nu_0(\la)\delta_6(\la) + (a_0-\nu_0(\la)^2)h(\la)}{\Delta(\la)} \right] \nu'_0(\la) \D\la.
    \end{align}
    \end{subequations}
    Moreover, in expressions~\eqref{eqn:sinkSolnY} and~\eqref{eqn:sinkSoln.0}, $h(\argdot;t)$ may be replaced with $h(\argdot,T)$.
\end{prop}

\begin{prop} \label{prop:sinkExistence}
    Assume $U_e\in\AC^4_{\mathrm p}(\Omega_e)$ and $H\in\AC^2_{\mathrm p}[0,\infty)$, per definition~\ref{defn:piecewiseACfunctions}.
    For all $x\in\Omega_e$ and $t\geq0$, let $u_e(x,t)$ be defined by equations~\eqref{eqn:sinkSoln}.
    Then the functions $u_-,u_+,u_0$ satisfy problem~\eqref{eqn:generalproblem},~\eqref{eqn:sinkVC}.
\end{prop}

The proofs of propositions~\ref{prop:sinkSolution} and~\ref{prop:sinkExistence} are very similar to the proofs of propositions~\ref{prop:sourceSolution} and~\ref{prop:sourceExistence}, so they are omitted.

\section{Conclusion} \label{sec:Conclusion}

We have shown how to implement the unified transform method for problem~\eqref{eqn:generalproblem} on metric graph domains composed of incoming leads, outgoing leads, and bonds.
In particular, we solved the problem and proved existence and unicity for four particular classes of problems, representing all metric graph topologies in which $0$ or $1$ bonds are attached at a single defect point.
Although, in contrast with~\cite{DSS2016a,MNS2018a}, the vertex conditions we considered are not absolutely general linear vertex conditions for any given metric graph topology, they were kept sufficiently broad to illustrate how wellposedness criteria, such as $\sum B_k\neq0$ in proposition~\ref{prop:mismatchSolution}, on the vertex conditions may arise.
Specifically, we showed how these wellposedness criteria manifest in the unified transform method as restrictions on the validity of the unicity and existence arguments via asymptotic analysis of the determinant $\Delta$.
This $\Delta$ is analogous to the object referred to as the characteristic determinant in finite interval problems~\cite{Bir1908b,Loc2000a,Loc2008a}, which is known to be closely related to wellposedness~\cite{Smi2012a,PS2013a,ABS2022a}.
For example, it is known that finite interval problems without the transport term are wellposed if and only if a certain minor of $\Delta$ is nonzero~\cite{Smi2012a} and this is one of the two minors Locker uses to characterize the corresponding spatial differential operators as ``regular'', ``simply irregular'', or ``degenerate irregular''~\cite{Loc2008a}.
It is an open question exactly how this correspondence extends to problems on the half line and generalizations, such as those studied in this paper, to metric graph domains including leads.

One limitation of the present formulation of the unified transform method on metric graphs came to light in the preparation of this paper: the solution of the linear system that represents the D to N map (that is, the formulation of results such as lemmata~\ref{lem:loopLinearSystemSolution} and~\ref{lem:sourceLinearSystemSolution}) becomes somewhat arduous as the number of edges in the metric graph increases.
Indeed, the dimension of the matrix $\mathcal A$ is expected to behave like $\abs{\mathcal L^+} + 2\abs{\mathcal L^-} + 3\abs{\mathcal B}$ for a more general metric graph.
Of course, if the partial differential equation were of still higher spatial order, then this difficulty would be exacerbated.
Therefore, one may hope that it were possible to decompose the problem on a complicated metric graph into a number of smaller problems, each posed on relatively simple metric graphs, solve the subsiduary problems separately, then recombine into a solution of the original problem.
This is precisely the approach that is followed for the linear Schr\"odinger equation, where the scattering at each vertex is derived separately, by considering the graph locally at each vertex as a star graph, and then the scattering on the full graph is reconstructed from those simpler scattering results~\cite{BK2013a}.
It would be interesting to attempt a similar approach for the linearized Korteweg de Vries equation.
However, because the associated differential operators are so far from self adjoint, the usual scattering data alone are insufficient to find a solution in this setting.
Therefore, it is expected that such a decomposition and reconstruction approach will have to be formulated in coordinate space, rather than purely in the spectral domain.

\section*{Acknowledgement}
The author is grateful to Andreas Chatziafratis, Robert Marangell, Michael Meylan, Dmitry Pelinovsky, and Ravi Pethiyagoda for valuable conversations.

\appendix

\section{Piecewise absolutely continuous functions} \label{sec:piecewiseACfunctions}

Many of the arguments in this paper rely on (repeated) integration by parts in the formula for the Fourier transform of a given function.
For a finite real interval $\Omega$, it is convenient to work with spaces of absolutely continuous functions $\AC^N(\Omega)$, because the Fourier transforms of such functions, and their first $N+1$ derivatives, are all defined, allowing integration by parts up to $N+1$ times.
Because functions in $\AC^N(\Omega)$ may have nonzero boundary values, when integrating by parts, it is necessary to account for the existence of boundary terms.
Therefore, it is no great complication (save some occasional notational awkwardness) to allow that the functions be only ``piecewise absolutely continuous'', rather than truly absolutely continuous on all of $\Omega$.
Definition~\ref{defn:piecewiseACfunctions} makes this concept precise, by parititioning $\Omega$ using the finite set of discontinuous points $X$.
However, because we also want to work with semiinfinite real intervals $\Omega$, and the classical definition of the Fourier transform requires absolute integrability of the function, we have selected a rather strong sense of piecewise absolute continuity.
Specifically, the function and the first $N+1$ of its derivatives are absolutely integrable not only locally on each finite connected subinterval of $\Omega\setminus X$, but also on the semiinfinite subinterval.
The following lemma~\ref{lem:FTLeadingOrder} summarises the asymptotic behaviour of Fourier transforms of such functions.

\begin{defn} \label{defn:piecewiseACfunctions}
    Suppose $\Omega$ is a finite or semiinfinite real interval and $X=\{\xi_0,\xi_1,\ldots,\xi_n\}$ is a finite set with
    \[
        \inf(\Omega) = \xi_0 < \xi_1 < \ldots < \xi_n = \sup(\Omega).
    \]
    If $U\in\Lebesgue^1(\Omega)$ has the property that, for all $j\in\{1,2,\ldots,n\}$ and all $k\in\{0,1,\ldots,N\}$, the restriction of the derivative $U^{(k)}\rvert_{(\xi_{j-1},\xi_j)}$ admits absolutely continuous extension to $\clos(\xi_{j-1},\xi_j)$ whose derivative is in $\Lebesgue^1(\xi_{j-1},\xi_j)$, then we say that $U$ is \emph{piecewise absolutely continuous} of order $N$; $U\in\AC^N_{\mathrm p}(\Omega)$.
    Where necessary, we specify the \emph{partition set} $X$ with $U\in\AC^N_{\mathrm p}(\Omega;X)$.
\end{defn}

\begin{lem} \label{lem:FTLeadingOrder}
    For $\Omega$ a real interval with partition set $X=\{\xi_0,\xi_1,\ldots,\xi_n\}$, suppose $U\in\AC^N_{\mathrm p}(\Omega;X)$.
    Then, for all $\la\in\CC\setminus\{0\}$ for which $\widehat U(\la)$ converges,
    \[
        \widehat U(\la)
        = \phi(\la)
        + \sum_{j=1}^n \frac1{(\ri\la)^{N+1}}\int_{\xi_{j-1}}^{\xi_j} \re^{-\ri\la x} U^{(N+1)}(x) \D x,
    \]
    with
    \[
        \phi(\la) = \sum_{j=1}^n \sum_{k=1}^{N+1} \frac1{(\ri\la)^k} \left[ \re^{-\ri\la\xi_{j-1}} U^{(k-1)}(\xi_{j-1}^+) - \re^{-\ri\la\xi_j} U^{(k-1)}(\xi_j^-) \right].
    \]
    Moreover, if $\Omega=(-\infty,0]$, then
    \begin{equation} \label{eqn:FTLeadingOrder.-}
        \widehat U(\la) = \phi(\la) + \bigoh{\la^{-(N+1)}} \mbox{ and } \phi(\la)=\lindecayla,
    \end{equation}
    as $\la\to\infty$ in $\clos(\CC^+)$, both uniformly in $\arg(\la)$, and, for any $\epsilon>0$, $\phi$ may be analytically extended to the closed horizontal strips $\mathfrak s = \{\pm z-\ri y: z>1,\;y\in[0,\epsilon]\}$, as shown in figuer~\ref{fig:Gammapm}, and $\phi(\la)=\lindecayla$, uniformly in $\Im(\la)$, as $\la\to\infty$ in this set.
    If instead $\Omega=[0,\infty)$, then
    \begin{equation} \label{eqn:FTLeadingOrder.+}
        \widehat U(\la) = \phi(\la) + \bigoh{\la^{-(N+1)}} \mbox{ and } \phi(\la)=\lindecayla,
    \end{equation}
    as $\la\to\infty$ in $\clos(\CC^-)$, both uniformly in $\arg(\la)$.
    However, if $\Omega=[0,\eta]$, then $\widehat U$ and $\phi$ obey the same asymptotic bounds on $\clos(\CC^-)$ as if $\Omega$ were $[0,\infty)$, and also
    \begin{equation} \label{eqn:FTLeadingOrder.etaU}
        \re^{-\ri\eta\la}\widehat U(\la) = \re^{-\ri\eta\la}\phi(\la) + \bigoh{\abs{\la}^{-(N+1)}} \mbox{ and }
        \re^{-\ri\eta\la}\phi(\la) = \lindecayla,
    \end{equation}
    as $\la\to\infty$ in $\clos(\CC^+)$.
\end{lem}

\begin{proof}[Proof of lemma~\ref{lem:FTLeadingOrder}]
    The representation follows from integration by parts $N+1$ times, in each of the $n$ connected components of $\Omega\setminus X$.

    If $\Omega=(-\infty,0]$ and $\la\in\clos(\CC^+)$, then, by the triangle inequality,
    \[
        \abs{\sum_{j=1}^n \frac1{(\ri\la)^{N+1}}\int_{\xi_{j-1}}^{\xi_j} \re^{-\ri\la x} U^{(N+1)}(x) \D x}
        \leq
        \abs\la^{-(N+1)} \normp{U^{(N+1)}}{\Lebesgue^1(\Omega)}
        \sup_{\substack{\la\in\clos(\CC^+),\\x\in(-\infty,0]}} \abs{\re^{-\ri\la x}},
    \]
    the norm is finite by hypothesis, and the supremum is $1$.
    The proofs of the other asymptotic bounds on half planes follow this one.

    If $\Omega=(-\infty,0]$, then
    \begin{align*}
        \abs{\phi(\la)}
        &\leq
        \abs\la^{-1}
        \sum_{j,k} \left(\abs{U^{(k-1)}(\xi_{j-1}^+)}+\abs{U^{(k-1)}(\xi_{j}^-)}\right)
        \sup_{\substack{\la\in S,\\\xi\in X}} \abs{\re^{-\ri\la x}}
        \\
        &=
        \abs\la^{-1}
        \sum_{j,k} \left(\abs{U^{(k-1)}(\xi_{j-1}^+)}+\abs{U^{(k-1)}(\xi_{j}^-)}\right)
        \re^{-\epsilon \xi_1},
    \end{align*}
    but $\xi_1$ is a finite negative number and the boundary values are all finite by hypothesis.
    The other bounds of $\phi$ on the strip may be obtained similarly.
\end{proof}

\section{Technical lemmata} \label{sec:technicalLemmata}

In this section, we give a series of lemmata that are useful in the analysis of integrals such as those appearing in solution formulae~\eqref{eqn:mismatchSolution},~\eqref{eqn:loopSoln},~\eqref{eqn:sourceSoln}, and~\eqref{eqn:sinkSoln}.
In showing that these formulae satisify the initial interface problems we study in this paper, it is necessary to take various $x$ and $t$ limits and derivatives of such integrals.
This section is organized as follows.
\begin{description}
    \item[\S\ref{ssec:techncialLemmata.BV}]{
        houses lemmata that can be used to evaluate the boundary values, or ``vertex values'', of each integral.
        \begin{description}
            \item[Lemma~\ref{lem:vertexRewritingLemma.Real-}]{lets one evaluate target vertex values of real integrals like the first of equation~\eqref{eqn:mismatchSolution.-} or, rather, the $\la\mapsto\nu(\la)$ change of variables of such integrals.}
            \item[Lemma~\ref{lem:vertexRewritingLemma.Real+}]{is like lemma~\ref{lem:vertexRewritingLemma.Real-}, except for source vertex values instead, such as the first integral of equation~\eqref{eqn:mismatchSolution.+} after a $\la\mapsto\nu(\la)$ change of variables.}
            \item[Lemma~\ref{lem:vertexRewritingLemma.D3}]{is the tool for evaluating either a source (with equation~\eqref{eqn:vertexRewritingLemma.D3.Solutionj}) or target (using equation~\eqref{eqn:vertexRewritingLemma.D3.Solution0}) vertex value of a $\partial D$ integral such as the second integral of equation~(\ref{eqn:mismatchSolution}$\pm$).}
            \item[Lemma~\ref{lem:vertexRewritingLemma.D3Stronger}]{is a generalization of lemma~\ref{lem:vertexRewritingLemma.D3} appropriate for integrals arising from problems with inhomogeneous boundary conditions.}
        \end{description}
        Because the vertex values can be the limiting values of the spatial derivatives of $u_e$, these all rely on lemmata in the following section.
    }
    \item[\S\ref{ssec:techncialLemmata.ContourDeformation}]{
        features lemmata that enable taking $x$ or $t$ derivatives of such integrals.
        They find direct application in arguments that~\eqref{eqn:generalproblem.PDE} holds.
        \begin{description}
            \item[Lemma~\ref{lem:ConvergenceRealIntegrals}]{
                is essentially a Jordan's lemma type argument for deforming the leading order part of the integrand of a real integral (such as the first integral of equation~(\ref{eqn:mismatchSolution}$\pm$)) away from the real line to a parallel contour.
                The result is a sum of two contour integrals, each of which converges uniformly.
                Moreover, they are shown to converge uniformly even after the integrands are modified by $x$ or $t$ differentiation of the integrands.
            }
            \item[Lemma~\ref{lem:ConvergenceRealIntegralsNu}]{is the same as lemma~\ref{lem:ConvergenceRealIntegrals}, but afer a $\la\mapsto\nu(\la)$ change of variables in the integrals of interest.}
            \item[Lemma~\ref{lem:ConvergenceD3Integrals}]{plays the role of lemma~\ref{lem:ConvergenceRealIntegrals}, except for starting integrals over $\partial D$ instead of real starting integrals.}
            \item[Lemma~\ref{lem:ConvergenceD3IntegralsStronger}]{is a generalization of lemma~\ref{lem:ConvergenceD3Integrals} to allow for the kinds of $\partial D$ integrals arising in problems with inhomogeneous boundary conditions.}
        \end{description}
    }
    \item[\S\ref{ssec:techncialLemmata.tLimit}]{
        contains lemma~\ref{lem:tLimitMainLemma}, which is used to take small $t$ limits of these integrals.
    }
\end{description}

All of these results are essentially contained in~\cite{CKS2023a,COT2024a} by Chatziafratis and collaborators.
However, it is not easy to adapt those results to the present setting, so we give new proofs, following the original proofs wherever convenient.
The inclusion of a transport term in our linearization of the KdV equation requires some additional analysis compared with~\cite{CKS2023a}, but the main change is, because we study our equation on three different domains $(-\infty,0]$, $[0,\infty)$, and $[0,\eta]$, certain kinds of integrals that were absent from the single interval case do appear in our solution formulae.
The proof for the short time limit lemma is a fairly straightforward reworking of a similar analysis in~\cite{COT2024a}, adapting from second order to third.

\subsection{To evaluate the vertex values of integrals} \label{ssec:techncialLemmata.BV}

The following two lemmata are used to express integrals along $\Gamma_\pm$ as integrals about $\partial D$ and $\Gamma$, while simultaneously taking $x$ derivatives and limits.
The contours $\Gamma_\pm$, shown on the left of figures~\ref{fig:Gammapm}, are the real line but perturbed away from zero along semicircular arcs into $\CC^\mp$ to ensure analyticity of $\nu$, while $\partial D$ and $\Gamma$ are the contours displayed in figure~\ref{fig:bdryD3}.
\begin{figure}
    \centering
    \includegraphics{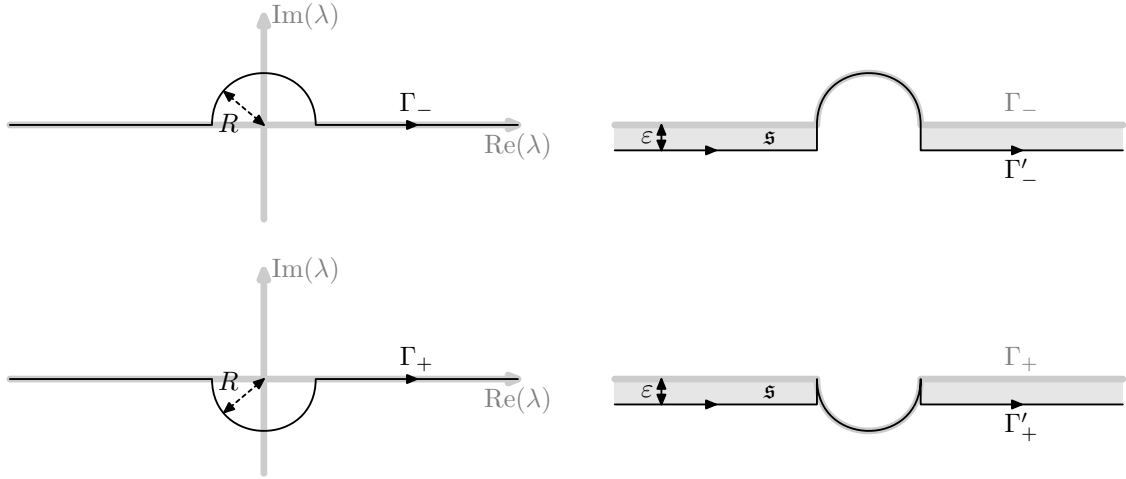}
    \caption{The contours $\Gamma_\pm$, the related contours $\Gamma'_\pm$, and the closed sets $\mathfrak s$ between them, shaded.}
    \label{fig:Gammapm}
\end{figure}%
To state and prove these results, we also require the closed set $\mathfrak s$ and shifted contour $\Gamma'_\pm$, all of which appear on the right in figure~\ref{fig:Gammapm}.
Also used in the proofs are sets $E=\{\la\in\CC:-\la\in D\}$ and $E'$, defined as $E$ but missing two semistrips parallel to its edges; $E'=E\setminus(\alpha S\cup\alpha^2S)$.
The latter two sets and their boundaries are displayed in figure~\ref{fig:vertexRewritingRealMinus}.
\begin{figure}
    \centering
    \includegraphics{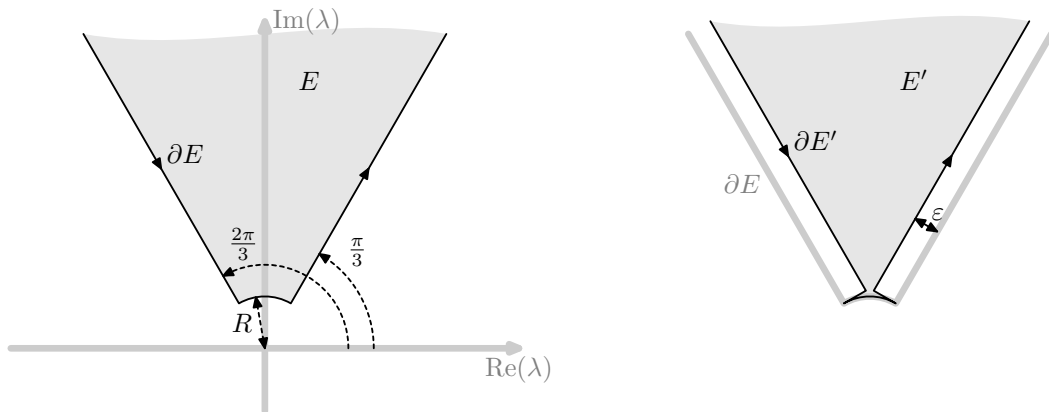}
    \caption{The sets $E,E'$, and their positively oriented boundaries $\partial E,\partial E'$.}
    \label{fig:vertexRewritingRealMinus}
\end{figure}
The $x\to0$ limit in each of the first two lemmata is taken in a different direction but, more importantly, there are different criteria on the asymptotic behaviour of the integrands.
In the present work, it is only necessary to take the limits in the directions specified, but it is really the asymptotic behaviour of the integrands that determines which of them can be used.

\begin{lem} \label{lem:vertexRewritingLemma.Real-}
    Suppose that $\theta\in\Lebesgue^1\cap\Lebesgue^\infty(\Gamma_-)$ may be extended to $\clos\{\la\in\CC^+:\abs\la> R\}$ and represented there as the sum of two holomorphic functions
    \[
        \theta(\la) = \phi(\la) + \psi(\la),
    \]
    for $\psi(\la)=\bigoh{\la^{-4}}$, uniformly in $\arg(\la)$, as $\la\to\pm\infty$ within $\Gamma_-\cup\clos (E)$, and $\phi$ extensible holomorphically still further to a neighbourhood of
    \(
        \mathfrak s
    \)
    with $\phi(\la)\to0$, uniformly in $\arg(\la)$, as $\la\to\infty$ within $\clos (E')\cup\mathfrak s$.
    Then, for $k\in\{0,1,2\}$, and for all $t>0$,
    \begin{multline} \label{eqn:vertexRewritingLemma.Real-.Solution}
        (-\ri)^k \lim_{x\to0^-} \partial_x^k \left[ \int_{\Gamma_-} \re^{\ri\nu(\la)x - \ri\la^3t} \theta(\la) \D\la \right]
        \\
        =
        \PV\,\left[
            \int_{\partial D} \re^{-\ri\la^3t} \left[ \alpha \psi(\alpha\la) \nu(\alpha\la)^k + \alpha^2 \psi(\alpha^2\la) \nu(\alpha^2\la)^k \right] \D\la
            \right. \\ \left.
            + \int_{\Gamma} \re^{-\ri\la^3t} \left[ \alpha \phi(\alpha\la) \nu(\alpha\la)^k + \alpha^2 \phi(\alpha^2\la) \nu(\alpha^2\la)^k \right] \D\la
        \right].
    \end{multline}
\end{lem}

\begin{proof}[Proof of lemma~\ref{lem:vertexRewritingLemma.Real-}]
    We denote the limit of interest
    \[
        I_k(t) = (-\ri)^k \lim_{x\to0^-} \partial_x^k \left[ \int_{\Gamma_-} \re^{\ri\nu(\la)x - \ri\la^3t} \theta(\la) \D\la \right].
    \]
    By lemma~\ref{lem:ConvergenceRealIntegralsNu} with $N=4$, because $t>0$,
    \begin{equation} \label{eqn:vertexRewritingLemma.Real-.Step1}
        I_k(t)
        =
        \ri^k\lim_{C\to\infty}
        \left\{
            \int_{\Gamma_-\rvert_{\abs{\Re(\la)}<C}} \psi(\la)
            + \int_{\Gamma'_-\rvert_{\abs{\Re(\la)}<C}} \phi(\la)
        \right\}
        \nu(\la)^k\re^{-\ri\la^3t} \D\la.
    \end{equation}

    As $\la\to\infty$ in $\clos(E)$, $\psi(\la)\nu(\la)^2=\lindecayla$, uniformly in the argument of $\la$.
    Hence, by Jordan's lemma, $\int_{\partial E} \psi(\la)\nu(\la)^k\re^{-\ri\la^3t}\D\la=0$.
    Similarly,
    \(
        \int_{\partial E'} \phi(\la)\nu(\la)^k \re^{-\ri\la^3t} \D\la = 0.
    \)
    Adding these two zero valued integrals to the right of equation~\eqref{eqn:vertexRewritingLemma.Real-.Step1}, we arrive at
    \begin{equation}
        I_k(t)
        =
        \PV\,\left[
            \int_{\alpha\partial D \cup \alpha^2\partial D} \re^{-\ri\la^3t} \psi(\la) \nu_-(\la)^k \D\la
            + \int_{\alpha\Gamma \cup \alpha\Gamma} \re^{-\ri\la^3t} \phi(\la) \nu(\la)^k \D\la
        \right],
    \end{equation}
    After changing variables to map the contours $\alpha^j\partial D$ and $\alpha^j\Gamma$ to $\partial D$ and $\Gamma$, this reduces to equation~\eqref{eqn:vertexRewritingLemma.Real-.Solution}
\end{proof}

\begin{lem} \label{lem:vertexRewritingLemma.Real+}
    Suppose that $\theta\in\Lebesgue^1\cap\Lebesgue^\infty(\Gamma_+)$ may be holomorphically extended to $\{\la\in\CC^-:\abs\la\geq R\}$ and represented there by
    \[
        \theta(\la) = \phi(\la) + \psi(\la),
    \]
    for $\psi(\la)=\bigoh{\la^{-4}}$ and $\phi(\la)\to0$, uniformly in $\arg(\la)$, as $\la\to\pm\infty$ within $\clos(\alpha E) \cup \clos(\alpha^2E)$.
    Then, for $k\in\{0,1,2\}$, and for all $t>0$,
    \begin{multline} \label{eqn:vertexRewritingLemma.Real+.Solution}
        (-\ri)^k \lim_{x\to0^+} \partial_x^k \left[ \int_{\Gamma_+} \re^{\ri\nu(\la)x - \ri\la^3t} \theta(\la) \D\la \right]
        \\
        =
        \PV\,\left[
            \int_{\partial D} \re^{-\ri\la^3t} \psi(\la) \nu(\la)^k \D\la
            + \int_{\Gamma} \re^{-\ri\la^3t} \phi(\la) \nu(\la)^k \D\la
        \right].
    \end{multline}
\end{lem}

\begin{proof}[Proof of lemma~\ref{lem:vertexRewritingLemma.Real+}]
    We denote the limit of interest
    \[
        I_k(t) = (-\ri)^k \lim_{x\to0^+} \partial_x^k \left[ \int_{\Gamma_+} \re^{\ri\nu(\la)x - \ri\la^3t} \theta(\la) \D\la \right].
    \]
    By lemma~\ref{lem:ConvergenceRealIntegralsNu} with $N=4$, because $t>0$,
    \begin{equation}
        I_k(t)
        =
        \ri^k\lim_{C\to\infty}
        \left\{
            \int_{\Gamma_+\rvert_{\abs{\Re(\la)}<C}} \psi(\la)
            + \int_{\Gamma'_+\rvert_{\abs{\Re(\la)}<C}} \phi(\la)
        \right\}
        \nu(\la)^k\re^{-\ri\la^3t} \D\la.
    \end{equation}

    To complete the proof, we argue that it is possible to deform these contours of integration over the sets shown in figure~\ref{fig:vertexRewritingRealPlus} down to the contours $\partial D$ and $\Gamma$, arriving at representation~\eqref{eqn:vertexRewritingLemma.Real+.Solution}.
    \begin{figure}
        \centering
        \includegraphics{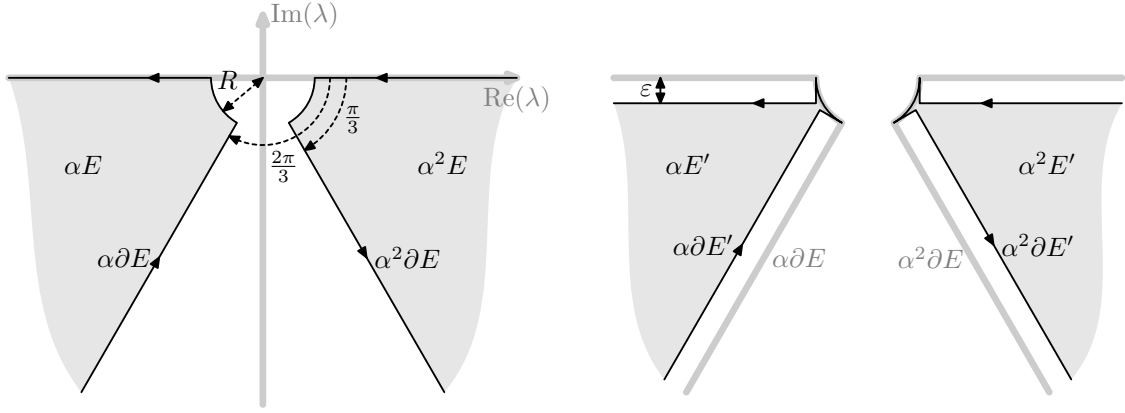}
        \caption{The sets $\alpha E,\alpha^2E,\alpha E',\alpha^2E'$, and their positively oriented boundaries.}
        \label{fig:vertexRewritingRealPlus}
    \end{figure}%
    As $\la\to\infty$ in $\clos(\alpha E\cup\alpha^2E)$, $\psi(\la)\nu(\la)^2=\lindecayla$, uniformly in the argument of $\la$.
    Hence, by Jordan's lemma, the integral about the boundary of this region of $\psi(\la)\nu(\la)^k\re^{-\ri\la^3t}$ is zero.
    Similarly,
    \[
        \int_{\alpha\partial E'\cup\alpha^2\partial E'} \phi(\la)\nu(\la)^k \re^{-\ri\la^3t} \D\la = 0.
        \qedhere
    \]
\end{proof}

\begin{lem} \label{lem:vertexRewritingLemma.D3}
    Suppose that $\theta\in\Lebesgue^1\cap\Lebesgue^\infty(\partial D)$ may be represented as the sum of functions
    \[
        \theta(\la) = \phi(\la) + \psi(\la),
    \]
    for $\psi\in\Lebesgue^1\cap\Lebesgue^\infty(\partial D)$ having $\psi(\la)=\bigoh{\la^{-4}}$ as $\la\to\pm\infty$ along $\partial D$, and $\phi$ extensible continuously to $S$ such that it is holomorphic inside $S$ with
    \begin{equation} \label{eqn:vertexRewritingLemma.D3.phiLimit}
        \lim_{C\to\infty}\max_{\substack{\la\in S:\\\abs\la=C}} \abs{\phi(\la)}=0.
    \end{equation}
    Then, for $k\in\{0,1,2\}$, and for all $t>0$, if $j\in\{1,2\}$,
    \begin{multline} \label{eqn:vertexRewritingLemma.D3.Solutionj}
        (-\ri)^k \lim_{x\to0^+} \partial_x^k \left[ \int_{\partial D} \re^{\ri\nu(\alpha^j\la)x - \ri\la^3t} \theta(\la) \D\la \right]
        \\
        =
        \left[
            \int_{\partial D} \re^{-\ri\la^3t} \psi(\la) \nu(\alpha^j\la)^k \D\la
            + \int_{\Gamma} \re^{-\ri\la^3t} \phi(\la) \nu(\alpha^j\la)^k \D\la
        \right],
    \end{multline}
    and
    \begin{multline} \label{eqn:vertexRewritingLemma.D3.Solution0}
        (-\ri)^k \lim_{x\to0^-} \partial_x^k \left[ \int_{\partial D} \re^{\ri\nu(\la)x - \ri\la^3t} \theta(\la) \D\la \right]
        \\
        =
        \left[
            \int_{\partial D} \re^{-\ri\la^3t} \psi(\la) \nu(\la)^k \D\la
            + \int_{\Gamma} \re^{-\ri\la^3t} \phi(\la) \nu(\la)^k \D\la
        \right].
    \end{multline}
\end{lem}

\begin{proof}[Proof of lemma~\ref{lem:vertexRewritingLemma.D3}]
    This is an application of lemma~\ref{lem:ConvergenceD3Integrals} with $N=4$.
    Indeed, the contour deformation is justified by the first part of that lemma, and the uniform convergence result is sufficient for both the differentiation and the subsequent $x\to0^\pm$ limit.
\end{proof}

\begin{lem} \label{lem:vertexRewritingLemma.D3Stronger}
    Fix $\tau\geq0$.
    Let $\theta,\phi,\psi$ be as in lemma~\ref{lem:vertexRewritingLemma.D3}, except that now
    \[
        \theta(\la) = \re^{\ri\la^3\tau'}\phi(\la) + \re^{\ri\la^3\tau}\varphi(\la) + \psi(\la),
    \]
    with $\varphi$ extensible continuously to $s$ such that it is holomorphic inside $s$ with
    \begin{equation} \label{eqn:vertexRewritingLemma.D3Stronger.phiLimit}
        \lim_{C\to\infty}\max_{\substack{\la\in S:\\\abs\la=C}} \abs{\varphi(\la)}=0.
    \end{equation}
    Then, for $k\in\{0,1,2\}$, and for all $t\in(0,\tau)$, if $j\in\{1,2\}$,
    \begin{multline} \label{eqn:vertexRewritingLemma.D3Stronger.Solutionj}
        (-\ri)^k \lim_{x\to0^+} \partial_x^k \left[ \int_{\partial D} \re^{\ri\nu(\alpha^j\la)x - \ri\la^3t} \theta(\la) \D\la \right]
        \\
        =
        \left[
            \int_{\partial D} \re^{-\ri\la^3t} \psi(\la) \nu(\alpha^j\la)^k \D\la
            + \int_{\Gamma} \re^{-\ri\la^3(t-\tau')} \phi(\la) \nu(\alpha^j\la)^k \D\la
            + \int_{\Gamma} \re^{\ri\la^3(\tau-t)} \varphi(\la) \nu(\alpha^j\la)^k \D\la
        \right],
    \end{multline}
    and
    \begin{multline} \label{eqn:vertexRewritingLemma.D3Stronger.Solution0}
        (-\ri)^k \lim_{x\to0^-} \partial_x^k \left[ \int_{\partial D} \re^{\ri\nu(\la)x - \ri\la^3t} \theta(\la) \D\la \right]
        \\
        =
        \left[
            \int_{\partial D} \re^{-\ri\la^3t} \psi(\la) \nu(\la)^k \D\la
            + \int_{\Gamma} \re^{-\ri\la^3(t-\tau')} \phi(\la) \nu(\la)^k \D\la
            + \int_{\Gamma} \re^{\ri\la^3(\tau-t)} \phi(\la) \nu(\la)^k \D\la
        \right].
    \end{multline}
\end{lem}

\begin{proof}[Proof of lemma~\ref{lem:vertexRewritingLemma.D3Stronger}]
    This follows the proof of lemma~\ref{lem:vertexRewritingLemma.D3}, except it uses lemma~\ref{lem:ConvergenceD3IntegralsStronger} instead of lemma~\ref{lem:ConvergenceD3Integrals}.
\end{proof}

\subsection{To reexpress integrals so that they converge uniformly} \label{ssec:techncialLemmata.ContourDeformation}

\begin{lem} \label{lem:ConvergenceRealIntegrals}
    Fix $B>\sqrt{(\abs a+\epsilon^2)/3}$.
    Suppose that, for some $N\in\NN$, $\theta\in\Lebesgue^1\cap\Lebesgue^\infty(-\infty,\infty)$ may be represented on $\RR\setminus(-B,B)$ by
    \[
        \theta(\la) = \phi(\la) + \psi(\la),
    \]
    for $\psi(\la)=\bigoh{\la^{-N}}$ as $\la\to\pm\infty$, and $\phi$ extensible holomorphically to a neighbourhood of
    \[
        \{z-\ri y: \abs z\geq B,\;0\leq y\leq\epsilon\},
    \]
    with
    \begin{equation} \label{eqn:ConvergenceRealIntegrals.phiLimit}
        \lim_{C\to\pm\infty}\max_{y\in[0,\epsilon]}\abs{\phi(C-\ri y)}=0.
    \end{equation}
    Define $E(\la;x,t) = \exp[\ri\la x - \ri(\la^3-a\la)t]$.
    Then, for all $x,t\in\RR$,
    \begin{multline} \label{eqn:ConvergenceRealIntegrals.integral}
        \PV\int_{-\infty}^\infty E(\la;x,t) \theta(\la) \D\la
        =
        \lim_{C\to\infty}\left[\int_{-B}^B E(\la;x,t) \theta(\la) \D\la %\\
        + \int_{\raisebox{-0.3ex}{$\mathrlap{\scriptstyle[-C,-B]\cup[B,C]}$}\hphantom{1}} E(\la;x,t) \psi(\la) \D\la \right. \\ \left.
        + \left\{
            \int_{\substack{[-C-\epsilon\ri,-B-\epsilon\ri]\\\cup[B-\epsilon\ri,C-\epsilon\ri]}}
            + \int_{\substack{[-B-\epsilon\ri,-B]\\\cup[B,B-\epsilon\ri]}}
        \right\} E(\la;x,t) \phi(\la) \D\la \right].
    \end{multline}
    Further, for each $M\in\{0,1,\ldots,N-2\}$, all $\tau>0$, and all $\xi\in\RR$, uniformly in $(x,t)\in(-\infty,\xi]\times[\tau,\infty)$,
    \begin{multline} \label{eqn:ConvergenceRealIntegrals.integralM}
        \lim_{C\to\infty}\left[ \int_{-B}^B E(\la;x,t) \la^M \theta(\la) \D\la %\\
        + \int_{\raisebox{-0.3ex}{$\mathrlap{\scriptstyle[-C,-B]\cup[B,C]}$}\hphantom{1}} E(\la;x,t) \la^M \psi(\la) \D\la \right. \\ \left.
        + \left\{
            \int_{\substack{[-C-\epsilon\ri,-B-\epsilon\ri]\\\cup[B-\epsilon\ri,C-\epsilon\ri]}}
            + \int_{\substack{[-B-\epsilon\ri,-B]\\\cup[B,B-\epsilon\ri]}}
        \right\} E(\la;x,t) \la^M \phi(\la) \D\la \right]
    \end{multline}
    converges.
\end{lem}

\begin{proof}[Proof of lemma~\ref{lem:ConvergenceRealIntegrals}]
    Although $E$ is pure oscillatory on $\RR$, it decays as $\la\to\infty$ along a contour parallel to but below the real line.
    Indeed, with $x<\xi\in\RR$ and $t>\tau\in\RR$, and for any real $z$ with $\abs{z} \geq B>b:=\sqrt{(\abs a+\epsilon^2)/3}$ and any $y\in[0,\epsilon]$, because $3z^2-a-y^2\geq3(z^2-b^2)$,
    \begin{multline} \label{eqn:ConvergenceRealIntegrals.uniformBoundE}
        \abs{E(z-\ri y;x,t)}
        = \re^{- y(3 z^2 - a - y^2)t + y x}
        \leq \re^{-y(3 z^2 - a - y^2)t + y\xi}
        \leq \re^{-3y(z^2 - b^2)t + y\xi} \\
        \leq \re^{-3y(z^2 - b^2)\tau + y\xi}
        \leq \re^{-3y(B^2 - b^2)\tau + y\xi},
    \end{multline}
    and the penultimate bound above implies that $\abs{E(z-\ri y;x,t)}$ is dominated uniformly in $(x,t)\in(-\infty,\xi]\times[\tau,\infty)$ by a function which converges to $0$ superexponentially as $z\to\pm\infty$.

    Equation~\eqref{eqn:ConvergenceRealIntegrals.integral}, but with the additional term
    \[
        \int_{\substack{[-C,-C-\epsilon\ri]\\\cup[C-\epsilon\ri,C]}} E(\la;x,t) \phi(\la) \D\la
    \]
    inside the limit on the right, is immediate from the definition of the principal value integral and Cauchy's theorem.
    From bound~\eqref{eqn:ConvergenceRealIntegrals.uniformBoundE} with $z=\pm C$,
    \[
        \int_{\substack{[-C,-C-\epsilon\ri]\\\cup[C-\epsilon\ri,C]}} \abs{E(\la;x,t) \phi(\la)} \D\hspace{-0.2em}\abs{\la}
        \leq
        \max_{y\in[0,\epsilon]}\abs{\phi(C-\ri y)} \int_0^\epsilon \re^{-3y(B^2 - b^2)\tau+\xi y} \D y.
    \]
    Hence, by hypothesis~\eqref{eqn:ConvergenceRealIntegrals.phiLimit}, the large $C$ limit of the above integral is zero, and this limit is uniform in $(x,t)\in(-\infty,\xi]\times[\tau,\infty)$.
    This justifies equation~\eqref{eqn:ConvergenceRealIntegrals.integral}.

    The first integral of expression~\eqref{eqn:ConvergenceRealIntegrals.integralM} converges absolutely, is bounded uniformly in $(x,t)\in(-\infty,\xi]\times[\tau,\infty)$ for any fixed$(\xi,\tau)\in\RR^2$, and is independent of $C$.
    Similarly, the second braced integral of~\eqref{eqn:ConvergenceRealIntegrals.integralM} converges absolutely and is bounded uniformly in $(x,t)\in(-\infty,\xi]\times[\tau,\infty)$ for any fixed $(\xi,\tau)\in\RR^2$; it is also independent of $C$.

    Thus uniform convergence of the whole of limit~\eqref{eqn:ConvergenceRealIntegrals.integralM} is equivalent to uniform convergence of
    \[
        \int_{\raisebox{-0.3ex}{$\mathrlap[\scriptstyle]{[-\infty,-B]\cup[B,\infty]}$}\hphantom{1}} E(\la;x,t) \la^M \psi(\la) \D\la
        + \int_{\raisebox{-1.2ex}{$\mathrlap[\scriptstyle]{\substack{[-\infty-\epsilon\ri,-B-\epsilon\ri]\\\cup[B-\epsilon\ri,+\infty-\epsilon\ri]}}$}\hphantom{1}} E(\la;x,t) \la^M \phi(\la) \D\la.
    \]
    The observation immediately following equation~\eqref{eqn:ConvergenceRealIntegrals.uniformBoundE} guarantees absolute convergence of the latter integral, uniformly in $(x,t)\in(-\infty,\xi]\times[\tau,\infty)$.
    Finally, $\abs{E(\la;x,t)\la^M\psi(\la)}=\bigoh{\la^{-2}}$, uniformly in $(x,t)\in(-\infty,\xi]\times[\tau,\infty)$, so the first integral converges absolutely uniformly.
\end{proof}

By essentially the same proof, we also have the following lemma, which may be applied to a similar integral, but after a change of variables.

\begin{lem} \label{lem:ConvergenceRealIntegralsNu}
    Suppose that, for some $N\in\NN$, $\theta\in\Lebesgue^1\cap\Lebesgue^\infty(\Gamma_\pm)$ may be represented on $\Gamma_\pm$ by
    \[
        \theta(\la) = \phi(\la) + \psi(\la),
    \]
    for $\psi(\la)=\bigoh{\la^{-N}}$ as $\la\to\pm\infty$, and $\phi$ extensible holomorphically to a neighbourhood of
    \(
        \mathfrak s,
    \)
    with
    \begin{equation} \label{eqn:ConvergenceRealIntegralsNu.phiLimit}
        \lim_{C\to\pm\infty}\max_{y\in[0,\epsilon]}\abs{\phi(C-\ri y)}=0.
    \end{equation}
    Define $E(\la;x,t) = \exp[\ri\nu(\la) x - \ri\la^3t]$.
    Then, for all $x,t\in\RR$,
    \begin{equation} \label{eqn:ConvergenceRealIntegralsNu.integral}
        \PV\int_{\Gamma_\pm} E(\la;x,t) \theta(\la) \D\la
        =
        \lim_{C\to\infty}\left[
            \int_{\raisebox{-0.3ex}{$\mathrlap[\scriptstyle]{\Gamma_\pm\rvert_{\abs{\Re(\la)}<C}}$}\hphantom{1}} E(\la;x,t) \psi(\la) \D\la
            + \int_{\raisebox{-0.3ex}{$\mathrlap[\scriptsize]{\Gamma'_\pm\rvert_{\abs{\Re(\la)}<C}}$}\hphantom{1}} E(\la;x,t) \phi(\la) \D\la
        \right].
    \end{equation}
    Further, for each $M\in\{0,1,\ldots,N-2\}$, all $\tau>0$, and all $\xi\in\RR$, uniformly in $(x,t)\in(-\infty,\xi]\times[\tau,\infty)$,
    \begin{equation} \label{eqn:ConvergenceRealIntegralsNu.integralM}
        \lim_{C\to\infty}\left[
            \int_{\raisebox{-0.3ex}{$\mathrlap[\scriptsize]{\Gamma_\pm\rvert_{\abs{\Re(\la)}<C}}$}\hphantom{1}} E(\la;x,t) \la^M \psi(\la) \D\la
            + \int_{\raisebox{-0.3ex}{$\mathrlap[\scriptsize]{\Gamma'_\pm\rvert_{\abs{\Re(\la)}<C}}$}\hphantom{1}} E(\la;x,t) \la^M \phi(\la) \D\la
        \right]
    \end{equation}
    converges.
\end{lem}

\begin{lem} \label{lem:ConvergenceD3Integrals}
    Suppose that, for some $N\in\NN$, $\theta\in\Lebesgue^1\cap\Lebesgue^\infty(\partial D)$ may be represented as the sum of functions
    \[
        \theta(\la) = \phi(\la) + \psi(\la),
    \]
    for $\psi\in\Lebesgue^1\cap\Lebesgue^\infty(\partial D)$ having $\psi(\la)=\bigoh{\la^{-N}}$ as $\la\to\pm\infty$ along $\partial D$, and $\phi$ extensible continuously to $S$ such that it is holomorphic inside $S$ with
    \begin{equation} \label{eqn:ConvergenceD3Integrals.phiLimit}
        \lim_{C\to\infty}\max_{\substack{\la\in S:\\\abs\la=C}} \abs{\phi(\la)}=0.
    \end{equation}
    Define $E_j(\la;x,t) = \exp[\ri\nu(\alpha^j\la) x - \ri\la^3t]$.
    If $(j,x)\in\{0\}\times(-\infty,0]\cup\{1,2\}\times[0,\infty)$, then, for all $t\geq0$,
    \begin{equation} \label{eqn:ConvergenceD3Integrals.integral}
        \PV\int_{\partial D} E_j(\la;x,t) \theta(\la) \D\la
        =
        \PV\,\left\{ \int_{\partial D} \psi(\la) +  \int_{\Gamma}\phi(\la) \right\} E_j(\la;x,t) \D\la.
    \end{equation}
    Further, for each $M\in\{0,1,\ldots,N-2\}$, all $\tau'>\tau>0$, the integrals
    \begin{equation} \label{eqn:ConvergenceD3Integrals.integralM}
        \int_{\partial D} E_j(\la;x,t) \la^M \psi(\la) \D\la
        \qquad\mbox{and}\qquad
        \int_{\Gamma} E_j(\la;x,t) \la^M \phi(\la) \D\la
    \end{equation}
    converge absolutely uniformly in $(x,t)\in(-\infty,0]\times[\tau,\tau']$ for $j=0$, and absolutely uniformly in $(x,t)\in[0,\infty)\times[\tau,\tau']$ for $j\in\{1,2\}$.
\end{lem}

\begin{proof}[Proof of lemma~\ref{lem:ConvergenceD3Integrals}]
    Inside the principal value, that is inside a limit $C\to\infty$, we make a contour deformation for the $\phi$ part of the integral from $(\partial D)\rvert_{\abs{\la}<C}$ to the contour $\Gamma_C\cup\gamma_C$, where $\gamma_C$ is the union of line segments
    \[
        \{\alpha^2(C-\ri y): y\in[0,\epsilon]\}
        \cup
        \{\alpha(-C-\ri y): y\in[0,\epsilon]\},
    \]
    oriented as shown on figure~\ref{fig:gammaC}, and
    \[
        \Gamma_C = \Gamma\rvert_{\abs{\la-\epsilon[\ri+\sqrt3\sgn(\Re(\la))]/2}<C}.
    \]
    \begin{figure}
        \centering
        \includegraphics{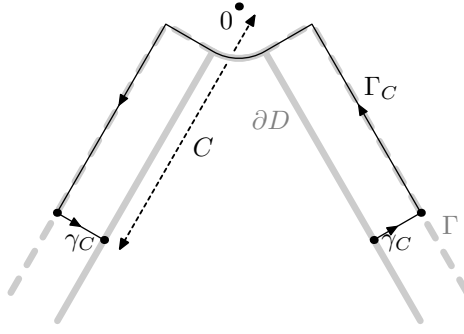}
        \caption{The contours $\Gamma_C$ and $\gamma_C$, for arbitrarily large $C>R$.}
        \label{fig:gammaC}
    \end{figure}%
    Cauchy's theorem justifies this contour deformation, and, provided we can show that
    \begin{equation} \label{eqn:ConvergenceD3Integrals.proof1}
        \lim_{C\to+\infty}\int_{\gamma_C} E_j(\la;x,t)\phi(\la)\D\la = 0,
    \end{equation}
    the first claim is justified.

    By definition,
    \[
        S = \{\alpha^2(z-\ri y):z\geq R, \; y\in[0,\epsilon]\} \cup \{\alpha(-z-\ri y):z\geq R, \; y\in[0,\epsilon]\} \cup \{R\re^{\ri\mu}:\mu\in[\tfrac{-2\pi}3,\tfrac{-\pi}3]\}.
    \]
    For $\la$ given in terms of $(z,y)$ by either of the formulae above, $\Re(-\ri\la^3)=y^3-3z^2y$, which is negative for $z$ sufficiently large, so $\abs{E_j(\la;0,t)}$ is bounded for $\la\in S$, provided $t\geq0$.
    By the asymptotic form for $\nu(\la)$, provided $\la$ is sufficently large, $\Re(\ri\nu(\la))\geq0$, while, for some $\delta>0$, $\Re(\ri\nu(\alpha^j\la))\leq\delta$.
    Hence
    \begin{multline*}
        \lim_{C\to+\infty} \abs{ \int_{\gamma_C} E_j(\la;x,t)\phi(\la)\D\la }
        \\
        \leq
        2\epsilon
        \lim_{C\to+\infty} \max_{\la\in\gamma_C} \abs{ E_j(\la;x,0) }
        \lim_{C\to+\infty} \max_{\la\in\gamma_C} \abs{ E_j(\la;0,t) }
        \lim_{C\to+\infty} \max_{\la\in\gamma_C} \abs{ \phi(\la) }
        \\
        \leq
        \begin{cases}
            2\epsilon \times \re^{\delta x} \times 1 \times 0 & \mbox{if } x\geq0,\; j\in\{1,2\}, \\
            2\epsilon \times 1 \times 1 \times 0 & \mbox{if } x\leq0,\; j=0,
        \end{cases}
    \end{multline*}
    which justifies limit~\eqref{eqn:ConvergenceD3Integrals.proof1}.

    As described above, provided $R$ is sufficently large and $(x,j)$ are chosen as in the statement of the theorem, $E_j(\la;x,0)$ are bounded everywhere on both $\partial D$ and $\Gamma$.
    On the circular arc component of both $\partial D$ and $\Gamma$, $\Re(-\ri\la^3)\leq R^3$, so, provided For $t\leq\tau'$, $\abs{E_j(\la;0,t)}\leq\re^{R^3\tau'}$ there.
    Hence the parts of the integrals along the common circular arc are absolutely bounded, uniformly in $(x,t)$.
    On either ray of $\partial D$, $\abs{E_j(\la;0,t)}=1$.
    Hence
    \[
        \abs{ \int_{\text{ray of }\partial D} E_j(\la;x,t)\psi(\la) \la^M \D\la }
    \]
    is bounded by
    \[
        \int_{\text{ray of }\partial D} \abs{\psi(\la)} \abs{\la}^M \D\abs{\la},
    \]
    which converges by hypothesis on $\psi$ for $M<N-1$.
    If $\la\in\{\alpha^2(z-\ri\epsilon),\alpha(-z-\ri\epsilon)\}\in\Gamma$, then $\Re(-\ri\la^3t)=\epsilon^3-3z^2\epsilon$, so, provided $z$ is sufficiently large and $t\geq\tau$,
    \[
        \abs{E_j(\la;0,t)}
        \leq
        \re^{-2\epsilon z^2t}
        \leq
        \re^{-2\epsilon z^2\tau}.
    \]
    Therefore,
    \[
        \abs{ \int_{\text{ray of }\Gamma} E_j(\la;x,t)\phi(\la) \la^M \D\la }
    \]
    is bounded by
    \[
        \int_R^\infty \re^{-2\epsilon z^2\tau} \abs{\phi(\la(z))} \abs{\la(z)}^M \D z,
    \]
    with $\la(z)$ being the point $\la$ on that ray of $\Gamma$ with $\abs{\la}^2=z^2+\epsilon^2$.
    Because $\phi$ is bounded, this also converges.
\end{proof}

\begin{lem} \label{lem:ConvergenceD3IntegralsStronger}
    Fix $\tau>\tau'\geq0$.
    Suppose that, for some $N\in\NN$, $\theta\in\Lebesgue^1\cap\Lebesgue^\infty(\partial D)$ may be represented as the sum of functions
    \[
        \theta(\la) = \re^{\ri\la^3\tau'}\phi(\la) + \re^{\ri\la^3\tau}\varphi(\la) + \psi(\la),
    \]
    for $\psi,\phi$ as in lemma~\ref{lem:ConvergenceD3Integrals}, and $\varphi$ extensible continuously to $s$ such that it is holomorphic inside $s$ with
    \begin{equation} \label{eqn:ConvergenceD3IntegralsStronger.varphiLimit}
        \lim_{C\to\infty}\max_{\substack{\la\in s:\\\abs\la=C}} \abs{\varphi(\la)}=0.
    \end{equation}
    Define $E_j(\la;x,t) = \exp[\ri\nu(\alpha^j\la) x - \ri\la^3t]$.
    % If $(j,x)\in\{0\}\times(-\infty,0)\cup\{1,2\}\times(0,\infty)$, then, for all $t\in[\tau',\tau]$,
    If
    \begin{align*}
        (x,t) &\in (-\infty,0]\times(\tau',\tau]\setminus\{(0,\tau)\} & \mbox{for } j &= 0, \mbox{ or} \\
        (x,t) &\in [0,\infty)\times(\tau',\tau]\setminus\{(0,\tau)\} & \mbox{for } j &\in\{1,2\},
    \end{align*}
    then
    \begin{equation} \label{eqn:ConvergenceD3IntegralsStronger.integral}
        \PV\int_{\partial D} E_j(\la;x,t) \theta(\la) \D\la
        =
        \PV\,\left\{ \int_{\partial D} \psi(\la) +  \int_{\Gamma}\re^{\ri\la^3\tau'}\phi(\la) + \int_{\gamma}\re^{\ri\la^3\tau}\varphi(\la) \right\} E_j(\la;x,t) \D\la.
    \end{equation}
    Further, for each $M\in\{0,1,\ldots,N-2\}$, all $\xi>0$, and all $(\sigma',\sigma)$ obeying $\tau'<\sigma'<\sigma<\tau$, the integrals
    \begin{equation} \label{eqn:ConvergenceD3IntegralsStronger.integralM}
        \int_{\partial D} E_j(\la;x,t) \la^M \psi(\la) \D\la,
        \quad
        \int_{\Gamma} E_j(\la;x,t) \la^M \re^{\ri\la^3\tau'} \phi(\la) \D\la,
        \quad\mbox{and}\quad
        \int_{\gamma} E_j(\la;x,t) \la^M \re^{\ri\la^3\tau} \varphi(\la) \D\la
    \end{equation}
    converge absolutely uniformly in
    \[
        (x,t)\in(-\infty,-\xi]\times[\sigma',\tau]\cup(-\infty,0]\times[\sigma',\sigma]
        \qquad\mbox{for } j=0,
    \]
    and absolutely uniformly in
    \[
        (x,t)\in[\xi,\infty)\times[\sigma',\tau]\cup[0,\infty)\times[\sigma',\sigma]
        \qquad\mbox{for } j \in\{1,2\}.
    \]
\end{lem}

\begin{proof}[Proof of lemma~\ref{lem:ConvergenceD3IntegralsStronger}]
    In the case that $\varphi=0$, this is a slight weakening of lemma~\ref{lem:ConvergenceD3Integrals}, except that a factor of $\re^{\ri\la^3\tau'}$ has been extracted from $\phi(\la)$ and the $t$ domains adjusted accordingly
    Therefore, we need only prove the claims that relate to $\varphi$.

    For $\la\in s$,
    \[
        \re^{\ri\la^3\tau}E_j(\la;x,t)
        =
        \re^{\ri\alpha^j\la x} \re^{\ri\la^3(\tau-t)}
        =
        \begin{cases}
            \re^{\ri \mu(\la) x'}\times[\mbox{bounded}] & \mbox{useful when } x\neq 0, \\
            [\mbox{bounded}]\times\re^{\ri\la^3(\tau-t)} & \mbox{useful when } t < \tau,
        \end{cases}
    \]
    for some $x'>0$ and $\mu(\la)\in\clos\CC^+$ with $\abs{\mu(\la)}=\abs\la$.
    Therefore, given $(x,t)\neq(0,\tau)$, the product $\re^{\ri\la^3\tau}E_j(\la;x,t)$ provides a kernel appropriate to application of Jordan's lemma.
    The contour deformation described in equation~\eqref{eqn:ConvergenceD3IntegralsStronger.integral} simplifies to
    \[
        \PV\int_{\partial D} \re^{\ri\la^3\tau} E_j(\la;x,t) \varphi(\la) \D\la
        =
        \PV\int_{\gamma} \varphi(\la) \re^{\ri\la^3\tau} E_j(\la;x,t) \D\la,
    \]
    which then follows from Cauchy's theorem and Jordan's lemma, via bound~\eqref{eqn:ConvergenceD3IntegralsStronger.varphiLimit}.

    It remains only to analyse the uniform convergence of the third of integrals~\eqref{eqn:ConvergenceD3IntegralsStronger.integralM} which, because the integrand is bounded, reduces to showing uniform convergence on the semiinfinite components of $\gamma$.
    As $\la\to\infty$ along either of these components, because
    \(
        \Re(\ri\alpha^j\la),
        \Re(\ri\la^3)
        \to-\infty,
    \)
    and $(x,t)\neq(0,\tau)$, it follows that $\re^{\ri\la^3\tau} E_j(\la;x,t)\to0$ exponentially, at least as fast as $E_j(\la;\pm\xi,0)$ or $\re^{\ri\la^3(\tau-\tau'')}$.
    Therefore, the integral converges uniformly.
\end{proof}

\subsection{To calculate the $t\to0$ limit of integrals} \label{ssec:techncialLemmata.tLimit}

\begin{lem} \label{lem:tLimitMainLemma}
    For fixed $a\in\RR$, define $E(\la;x,t) = \exp[\ri\la x - \ri(\la^3-a\la)t]$.
    Then, for all $x\neq0$ and $B>0$,
    \begin{equation} \label{eqn:tLimitMainLemma.limit1}
        \lim_{t\to0}\int_B^\infty E(\la;x,t) \frac1\la \D\la
        =
        \int_B^\infty E(\la;x,0) \frac1\la \D\la.
    \end{equation}
    Similarly, if $\nu$ is the function described by lemma~\ref{lem:generalNu}, then
    \begin{equation} \label{eqn:tLimitMainLemma.limit2}
        \lim_{t\to0}\int_R^\infty \re^{\ri\nu(\la)x-\ri\la^3t} \frac1\la \D\la
        =
        \int_R^\infty \re^{\ri\nu(\la)x} \frac1\la \D\la.
    \end{equation}
\end{lem}

\begin{proof}[Proof of lemma~\ref{lem:tLimitMainLemma}]
    % Addaptation of Chatziafratas's argument from~\cite{COT2024a}, unless I can find another reference.
    Limit~\eqref{eqn:tLimitMainLemma.limit2} follows from limit~\eqref{eqn:tLimitMainLemma.limit1} by change of variables and observation that $\nu^{-1}(\la)=\la+\lindecayla$ and $\nu'(\la)=1+\lindecayla$.
    For $t$ small enough, $\sgn(x+at)$ is constant, so we can change variables $\mu=\la\abs{x+at}$,
    \[
        \int_B^\infty E(\la;x,t) \frac1\la \D\la = \int_1^\infty \re^{\pm\ri\mu-\ri\mu^3\left(\frac t{\abs{x+at}^3}\right)} \frac1\mu \D\mu + \int_{B\abs{x+at}}^1 \re^{\pm\ri\mu-\ri\mu^3\left(\frac t{\abs{x+at}^3}\right)} \frac1\mu \D\mu,
    \]
    and the latter integral is a continuous function of $t$ close to $0$.
    Therefore, limit~\eqref{eqn:tLimitMainLemma.limit1} is equivalent to the same with $a=0$, $B=1$, and $x=\pm1$.
    Taking the complex conjugate only has the effect of changing the sign of $t$ and the coefficient of $\ri\mu$ in the exponent, so we need only show that
    \[
        \lim_{t\to0} \int_1^\infty \re^{\ri\la}\left(1-\re^{\ri\la^3t}\right)\frac1\la\D\la = 0.
    \]

    We write the integral of interest as $I_1(t)+I_2(t)+I_3(t)$, for
    \begin{align*}
        I_1(t) &= \int_1^{\frac1{2\sqrt{\abs t}}} \re^{\ri\la}\left(1-\re^{\ri\la^3t}\right)\frac1\la\D\la, \\
        I_2(t) &= \int_{\frac1{2\sqrt{\abs t}}}^{\frac1{\sqrt{\abs t}}} \re^{\ri\la}\left(1-\re^{\ri\la^3t}\right)\frac1\la\D\la, \\
        I_3(t) &= \int_{\frac1{\sqrt{\abs t}}}^\infty \re^{\ri\la}\left(1-\re^{\ri\la^3t}\right)\frac1\la\D\la.
    \end{align*}
    Changing variables $\la=\mu/\sqrt{\abs t}$ and using the notation $s = 1 / \sqrt{\abs t}$, we find
    \begin{equation} \label{eqn:tLimitMainLemma.proof1}
        I_2(t) = \int_{\frac12}^1 \re^{\ri\mu s} \frac1\mu\D\mu - \int_{\frac12}^1 \re^{\ri(\mu-\sgn(t)\mu^3) s} \frac1\mu\D\mu.
    \end{equation}
    The $s\to\pm\infty$ limits of the first integral on the right of equation~\eqref{eqn:tLimitMainLemma.proof1} are $0$ by the Riemann-Lebesgue lemma.
    Integrating by parts in the outer integral below,
    \begin{align*}
        \int_{\frac12}^1 \int_{\frac12}^\mu \re^{\ri(\rho\pm\rho^3) s} \D\rho \frac1{\mu^2} \D\mu
        &=
        \left[ -\frac1\mu \int_{\frac12}^\mu \re^{\ri(\rho\pm\rho^3) s} \D\rho \right]_{\mu=\frac12}^{\mu=1} + \int_{\frac12}^1 \re^{\ri(\mu\pm\mu^3) s} \frac1\mu\D\mu
        \\
        &= - \int_{\frac12}^1 \re^{\ri(\rho\pm\rho^3) s} + \int_{\frac12}^1 \re^{\ri(\mu\pm\mu^3) s} \frac1\mu\D\mu,
    \end{align*}
    hence the second integral on the right of equation~\eqref{eqn:tLimitMainLemma.proof1} can be reexpressed as
    \[
        \int_{\frac12}^1 \re^{\ri(\rho\pm\rho^3) s} \D\rho + \int_{\frac12}^1 \int_{\frac12}^\mu \re^{\ri(\rho\pm\rho^3) s} \D\rho \frac1{\mu^2} \D\mu.
    \]
    Now
    \(
        \abs{ \frac{\D^3}{\D\rho^3} (\rho\pm\rho^3) } = 6
    \),
    so the van der Corput lemma asserts the existence of a constant $c>0$ for which
    \[
        \int_{\frac12}^{\mu^\star} \re^{\ri(\rho\pm\rho^3) s} \D\rho \leq cs^{\frac{-1}3} = c t^{\frac16},
    \]
    with $\mu^\star\in\{1,\mu\}$.
    This establishes that $\lim_{t\to0}I_2(t)=0$.

    Integrating by parts,
    \begin{multline} \label{eqn:tLimitMainLemma.proof2}
        I_1(t)
        =
        \left[ \frac{\re^{\ri\la}}{\ri\la} \left( 1 - \frac{\re^{-\ri\la^3t}}{1-3\la^2t} \right) \right]_{\la=1}^{\la=\frac1{2\sqrt{\abs t}}}
        \\
        + \int_1^{\frac1{2\sqrt{\abs t}}} \frac{\re^{\ri\la}}{\ri\la^2} \left( 1 - \frac{\re^{-\ri\la^3t}}{1-3\la^2t} \right) \D\la
        - 6t\int_1^{\frac1{2\sqrt{\abs t}}} \frac{\re^{\ri\la-\ri\la^3t}}{(1-3\la^2t)^2}.
    \end{multline}
    The boundary terms evaluate to
    \[
        -2\ri\sqrt{\abs t}\re^{\ri/2\sqrt{\abs t}} \left( 1 - k\re^{-\ri/8\sqrt{\abs t}} \right) + \ri\re^\ri t \frac{3-\ri - t + \bigoh{t^2}}{1-3t},
    \]
    with $k=4$ or $k=4/7$ depending on the sign of $t$, so it has limit $0$ as $t\to0$.
    Bounds on $1-3\la^2t$ depend on the sign of $t$, but in either case it holds that $\frac14 \leq 1-3\la^2t \leq \frac74$.
    Using the lower bound to analyse the final term of equation~\eqref{eqn:tLimitMainLemma.proof2},
    \[
        \abs{ - 6t\int_1^{\frac1{2\sqrt{\abs t}}} \frac{\re^{\ri\la-\ri\la^3t}}{(1-3\la^2t)^2} }
        \leq
        6 \abs t \int_1^{\frac1{2\sqrt{\abs t}}} \frac{1}{(1/4)^2} \D\la
        = 96 \abs t \left( \frac1{2\sqrt{\abs t}} - 1 \right) \to 0,
    \]
    as $t\to0$.
    Pointwise in $\la\geq1$,
    \[
        \lim_{t\to0}\left( 1 - \frac{\re^{-\ri\la^3t}}{1-3\la^2t} \right) = 0,
    \]
    and the integrand of the penultimate term in equation~\eqref{eqn:tLimitMainLemma.proof2} is dominated by $5/\la^2$.
    Therefore, by Lebesgue's dominated convergence theorem, the penultimate term in equation~\eqref{eqn:tLimitMainLemma.proof2} also has limit $0$.
    We have shown that $\lim_{t\to0}I_1(t)=0$.

    With $\la\geq1/\sqrt{\abs t}$, $\abs{1-3\la^2t}\geq2\la^2\abs t \geq 2$.
    If we integrate by parts in the definition of $I_3(t)$, then we get the right side of equation~\eqref{eqn:tLimitMainLemma.proof2}, except with limits of integration and boundary limits $(\la=1/\sqrt{\abs t},\la\to\infty)$ instead of $(\la=1,\la=1/2\sqrt{\abs t})$.
    Once again, the bondary terms have limit $0$.
    The absolute value of the first integral is
    \[
        \abs{ \int_{\frac1{\sqrt{\abs t}}}^\infty \frac{\re^{\ri\la}}{\ri\la^2} \left( 1 - \frac{\re^{-\ri\la^3t}}{1-3\la^2t} \right) \D\la }
        \leq
        \int_{\frac1{\sqrt{\abs t}}}^\infty \frac1{\la^2} \left(1+\frac12\right) \D\la
        =
        \frac32\sqrt{\abs t} \to 0
    \]
    as $t\to0$.
    The absolute value of the second integral is
    \[
        \abs{ - 6t\int_{\frac1{\sqrt{\abs t}}}^\infty \frac{\re^{\ri\la-\ri\la^3t}}{(1-3\la^2t)^2} }
        \leq
        6 \abs t \int_{\frac1{\sqrt{\abs t}}}^\infty \frac1{\left(2\la^2\abs t\right)^2} \D\la = \frac12\sqrt{\abs t} \to 0.
    \]
    Hence also $\lim_{t\to0}I_3(t)=0$.
\end{proof}

\bibliographystyle{amsplain}
{\small\bibliography{dbrefs}}

\providecommand{\bysame}{\leavevmode\hbox to3em{\hrulefill}\thinspace}
\providecommand{\MR}{\relax\ifhmode\unskip\space\fi MR }
\providecommand{\MRhref}[2]{%
  \href{http://www.ams.org/mathscinet-getitem?mr=#1}{#2}
}
\providecommand{\href}[2]{#2}
\begin{thebibliography}{10}

\bibitem{ABS2022a}
S.~A. Aitzhan, S.~Bhandari, and D.~A. Smith, \emph{Fokas diagonalization of
  piecewise constant coefficient linear differential operators on finite
  intervals and networks}, Acta Appl. Math. \textbf{177} (2022), no.~2, 1--66.

\bibitem{AC2018a}
K.~Ammari and E.~Cr{\'e}peau, \emph{Feedback stabilization and boundary
  controllability of the {K}orteweg--de {V}ries equation on a star-shaped
  network}, SIAM Journal on Control and Optimization \textbf{56} (2018), no.~3,
  1620--1639.

\bibitem{PC2021a}
J.~Angulo~Pava and M.~Cavalcante, \emph{Linear instability criterion for the
  {K}orteweg--de {V}ries equation on metric star graphs}, Nonlinearity
  \textbf{34} (2021), no.~5, 3373.

\bibitem{AC2025a}
\bysame, \emph{Dynamics of the korteweg--de vries equation on a balanced metric
  graph}, Bulletin of the Brazilian Mathematical Society, New Series
  \textbf{56} (2025), no.~1, 5.

\bibitem{AM2024a}
J.~Angulo~Pava and A.~Mu{\~n}oz, \emph{Airy and schr$\backslash$" odinger-type
  equations on looping-edge graphs}, arXiv preprint arXiv:2410.11729 (2024),
  1--28.

\bibitem{AP2015a}
A.~Arancibia and M.~Plyushchay, \emph{Chiral asymmetry in propagation of
  soliton defects in crystalline backgrounds}, Phys. Rev. D \textbf{92} (2015),
  105009.

\bibitem{APSS2015a}
M.~Asvestas, E.~P. Papadopoulou, A.~G. Sifalakis, and Y.~G. Saridakis,
  \emph{The unified transform for a class of reaction-diffusion problems with
  discontinuous time dependent parameters}, Proceedings of the World Congress
  on Engineering, vol.~1, 2015.

\bibitem{BK2013a}
G.~Berkolaiko and P.~Kuchment, \emph{Introduction to quantum graphs},
  Mathematical surveys and monographs, American Mathematical Society,
  Providence, Rhode Island, 2013.

\bibitem{Bir1908b}
G.~D. Birkhoff, \emph{Boundary value and expansion problems of ordinary linear
  differential equations}, Trans. Amer. Math. Soc. \textbf{9} (1908), 373--395.

\bibitem{BC2008a}
J.~Bona and R.~Cascaval, \emph{Nonlinear dispersive waves on trees}, Canadian
  Applied Mathematics Quarterly \textbf{16} (2008), no.~1, 1--18.

\bibitem{BCCK2024a}
J.~L. Bona, A.~Chatziafratis, H.~Chen, and S.~Kamvissis, \emph{The linear
  {BBM}-equation on the half-line, revisited}, Letters in Mathematical Physics
  \textbf{114} (2024), no.~4, 91.

\bibitem{CJ1959a}
H.~S. Carslaw and J.~C. Jaeger, \emph{Conduction of heat in solids}, 2nd ed.,
  Oxford University Press, Oxford, 1959.

\bibitem{Cau2015a}
V.~Caudrelier, \emph{On the inverse scattering method for integrable {PDE}s on
  a star graph}, Communications in Mathematical Physics \textbf{338} (2015),
  893--917.

\bibitem{Cav2018a}
M.~Cavalcante, \emph{The {K}orteweg-de {V}ries equation on a metric star
  graph}, Zeitschrift f\"ur angewandte Mathematik und Physik \textbf{69}
  (2018), 12401--12422.

\bibitem{CCM2020a}
E.~Cerpa, E.~Cr{\'e}peau, and C.~Moreno, \emph{On the boundary controllability
  of the {K}orteweg--de {V}ries equation on a star-shaped network}, IMA Journal
  of Mathematical Control and Information \textbf{37} (2020), no.~1, 226--240.

\bibitem{ACCF2024a}
A.~Chatziafratis, E.~C. Aifantis, A.~Carbery, and A.~S. Fokas, \emph{Integral
  representations for the double-diffusivity system on the half-line},
  Zeitschrift f{\"u}r angewandte Mathematik und Physik \textbf{75} (2024),
  no.~2, 54.

\bibitem{ACF2024a}
A.~Chatziafratis, A.~S. Fokas, and E.~C. Aifantis, \emph{On {B}arenblatt's
  pseudoparabolic equation with forcing on the half-line via the {F}okas
  method}, ZAMM-Journal of Applied Mathematics and Mechanics/Zeitschrift
  f{\"u}r Angewandte Mathematik und Mechanik \textbf{104} (2024), no.~3,
  e202300614.

\bibitem{CGK2024a}
A.~Chatziafratis, L.~Grafakos, and S.~Kamvissis, \emph{Long-range instability
  of linear evolution {PDE} on semi-bounded domains via the fokas method},
  Dynamics of Partial Differential Equations \textbf{21} (2024), no.~2,
  97--169.

\bibitem{CK2024a}
A.~Chatziafratis and S.~Kamvissis, \emph{A note on uniqueness for linear
  evolution {PDE}s posed on the quarter-plane}, ` preprint arXiv:2401.08531
  (2024), 1--14.

\bibitem{CKS2023a}
A.~Chatziafratis, S.~Kamvissis, and I.~G. Stratis, \emph{Boundary behavior of
  the solution to the linear {K}orteweg-{D}e {V}ries equation on the half
  line}, Stud. Appl. Math. \textbf{150} (2023), 339--379.

\bibitem{CM2022a}
A.~Chatziafratis and D.~Mantzavinos, \emph{Boundary behavior for the heat
  equation on the half-line}, Mathematical Methods in the Applied Sciences
  \textbf{45} (2022), no.~12, 7364--7393.

\bibitem{ACFKM2025a}
A.~Chatziafratis, A.~Miranville, G.~Karali, A.~S. Fokas, and E.~C. Aifantis,
  \emph{Higher-order diffusion and {C}ahn--{H}illiard-type models revisited on
  the half-line}, Mathematical Models and Methods in Applied Sciences
  \textbf{35} (2025), no.~05, 1133--1197.

\bibitem{CO2024a}
A.~Chatziafratis and T.~Ozawa, \emph{New instability, blow-up and break-down
  effects for {S}obolev-type evolution {PDE}: asymptotic analysis for a
  celebrated pseudo-parabolic model on the quarter-plane}, Partial Differential
  Equations and Applications \textbf{5} (2024), no.~5, 30.

\bibitem{COT2024a}
A.~Chatziafratis, T.~Ozawa, and S.-F. Tian, \emph{Rigorous analysis of the
  unified transform method and long-range instabilities for the inhomogeneous
  time-dependent {S}chr{\"o}dinger equation on the quarter-plane},
  Mathematische Annalen \textbf{389} (2024), no.~4, 3535--3602.

\bibitem{Coo1931a}
J.~L. Coolidge, \emph{A treatise on algebraic plane curves}, Oxford University
  Press, Oxford, 1931.

\bibitem{CP2017a}
E.~Corrigan and R.~Parini, \emph{Type {I} integrable defects and finite-gap
  solutions for {KdV} and sine-{G}ordon models}, Journal of Physics A:
  Mathematical and Theoretical \textbf{50} (2017), no.~28, 284001.

\bibitem{DF2025a}
B.~Deconinck and M.~Farkas, \emph{Variable-coefficient evolution problems via
  the {F}okas method part {I}: Dissipative case}, Studies in Applied
  Mathematics \textbf{154} (2025), no.~1, e12800.

\bibitem{DPS2014a}
B.~Deconinck, B.~Pelloni, and N.~E. Sheils, \emph{Non-steady-state heat
  conduction in composite walls}, Proc. R. Soc. Lond. Ser. A Math. Phys. Eng.
  Sci. \textbf{470} (2014), no.~2165, 20130605.

\bibitem{DS2014a}
B.~Deconinck and N.~Sheils, \emph{Interface problems for dispersive equations},
  submitted for review, 2014.

\bibitem{DS2020a}
B.~Deconinck and N.~E. Sheils, \emph{The time-dependent {S}chr\"{o}dinger
  equation with piecewise constant potentials}, E. J. Appl. Math. \textbf{31}
  (2020), 57--83.

\bibitem{DSS2016a}
B.~Deconinck, N.~E. Sheils, and D.~A. Smith, \emph{The linear {K}d{V} equation
  with an interface}, Comm. Math. Phys. \textbf{347} (2016), 489--509.

\bibitem{DTV2014a}
B.~Deconinck, T.~Trogdon, and V.~Vasan, \emph{The method of {F}okas for solving
  linear partial differential equations}, SIAM Rev. \textbf{56} (2014), no.~1,
  159--186.

\bibitem{DF2023a}
M.~Farkas and B.~Deconinck, \emph{Solving the heat equation with variable
  thermal conductivity}, Applied Mathematics Letters \textbf{135} (2023),
  108395.

\bibitem{DF2024a}
\bysame, \emph{The explicit solution of linear, dissipative, second-order
  initial-boundary value problems with variable coefficients}, arXiv preprint
  arXiv:2405.20536 (2024), 1--53.

\bibitem{Fok2002b}
A.~S. Fokas, \emph{Integrable nonlinear evolution equations on the half-line},
  Comm. Math. Phys. \textbf{230} (2002), no.~1, 1--39.

\bibitem{Fok2008a}
\bysame, \emph{A unified approach to boundary value problems}, CBMS-SIAM, 2008.

\bibitem{FHM2016a}
A.~S. Fokas, A.~A. Himonas, and D.~Mantzavinos, \emph{The {K}orteweg--de
  {V}ries equation on the half-line}, Nonlinearity \textbf{29} (2016), no.~2,
  489.

\bibitem{FHM2017a}
\bysame, \emph{The nonlinear {S}chr{\"o}dinger equation on the half-line},
  Transactions of the American Mathematical Society \textbf{369} (2017), no.~1,
  681--709.

\bibitem{FP2001a}
A.~S. Fokas and B.~Pelloni, \emph{Two-point boundary value problems for linear
  evolution equations}, Math. Proc. Cambridge Philos. Soc. \textbf{131} (2001),
  521--543.

\bibitem{FS2016a}
A.~S. Fokas and D.~A. Smith, \emph{Evolution {P}{D}{E}s and augmented
  eigenfunctions. {F}inite interval}, Adv. Differential Equations \textbf{21}
  (2016), no.~7/8, 735--766.

\bibitem{HM2020a}
A.~A. Himonas and D.~Mantzavinos, \emph{Well-posedness of the nonlinear
  {S}chr{\"o}dinger equation on the half-plane}, Nonlinearity \textbf{33}
  (2020), no.~10, 5567.

\bibitem{HM2021a}
\bysame, \emph{The nonlinear {S}chr{\"o}dinger equation on the half-line with a
  {R}obin boundary condition}, Analysis and Mathematical Physics \textbf{11}
  (2021), no.~4, 157.

\bibitem{HM2022a}
\bysame, \emph{The {R}obin and {N}eumann problems for the nonlinear
  {S}chr{\"o}dinger equation on the half-plane}, Proceedings of the Royal
  Society A \textbf{478} (2022), no.~2265, 20220279.

\bibitem{HMY2019c}
A.~A. Himonas, D.~Mantzavinos, and F.~Yan, \emph{Initial-boundary value
  problems for a reaction-diffusion equation}, Journal of Mathematical Physics
  \textbf{60} (2019), no.~8, 081509:1--19.

\bibitem{HMY2019b}
\bysame, \emph{The {K}orteweg-de {V}ries equation on an interval}, Journal of
  Mathematical Physics \textbf{60} (2019), no.~5, 051507:1--26.

\bibitem{HMY2019a}
\bysame, \emph{The nonlinear {S}chr{\"o}dinger equation on the half-line with
  {N}eumann boundary conditions}, Applied Numerical Mathematics \textbf{141}
  (2019), 2--18.

\bibitem{Hop1919a}
J.~W. Hopkins, \emph{Some convergent developments associated with irregular
  boundary conditions}, Trans. Amer. Math. Soc. \textbf{20} (1919), 245--259.

\bibitem{Jac1915a}
D.~Jackson, \emph{Expansion problems with irregular boundary conditions}, Proc.
  Amer. Acad. Arts Sci. \textbf{51} (1915), no.~7, 383--417.

\bibitem{Lan1931a}
R.~E. Langer, \emph{The zeros of exponential sums and integrals}, Bull. Amer.
  Math. Soc. \textbf{37} (1931), 213--239.

\bibitem{LTC2022a}
T.~Lawrie, G.~Tanner, and D.~Chronopoulos, \emph{A quantum graph approach to
  metamaterial design}, Scientific Reports \textbf{12} (2022), 18006.

\bibitem{LP1967}
P.~D. Lax and R.~S. Phillips, \emph{Scattering theory}, Academic Press, New
  York and London, 1967.

\bibitem{Loc2000a}
J.~Locker, \emph{Spectral theory of non-self-adjoint two-point differential
  operators}, Mathematical Surveys and Monographs, vol.~73, American
  Mathematical Society, Providence, Rhode Island, 2000.

\bibitem{Loc2008a}
\bysame, \emph{Eigenvalues and completeness for regular and simply irregular
  two-point differential operators}, vol. 195, Memoirs of the American
  Mathematical Society, no. 911, American Mathematical Society, Providence,
  Rhode Island, 2008.

\bibitem{Man2022a}
D.~Mantzavinos, \emph{The {F}okas method for the well-posedness of nonlinear
  dispersive equations in domains with a boundary}, International Conference on
  Complex Systems and Applications Network Summer School, Springer, 2022,
  pp.~347--359.

\bibitem{MPSS2016a}
D.~Mantzavinos, M.~G. Papadomanolaki, Y.~G. Saridakis, and A.~G. Sifalakis,
  \emph{Fokas transform method for a brain tumor invasion model with
  heterogeneous diffusion in 1+ 1 dimensions}, Applied Numerical Mathematics
  \textbf{104} (2016), 47--61.

\bibitem{MNS2018a}
D.~Mugnolo, D.~Noja, and C.~Seifert, \emph{Airy-type evolution equations on
  star graphs}, Analysis \& {PDE} \textbf{11} (2018), 1625--1652.

\bibitem{MNS2024a}
\bysame, \emph{On solitary waves for the {K}orteweg--de {V}ries equation on
  metric star graphs}, arXiv:2402.08368 [math.AP], 2024.

\bibitem{Pap2011a}
G.~Papanicolaou, \emph{An example where separation of variable fails}, J. Math.
  Anal. Appl. \textbf{373} (2011), no.~2, 739--744.

\bibitem{Pel2004a}
B.~Pelloni, \emph{Well-posed boundary value problems for linear evolution
  equations on a finite interval}, Math. Proc. Cambridge Philos. Soc.
  \textbf{136} (2004), 361--382.

\bibitem{Pel2005a}
\bysame, \emph{The spectral representation of two-point boundary-value problems
  for third-order linear evolution partial differential equations}, Proc. R.
  Soc. Lond. Ser. A Math. Phys. Eng. Sci. \textbf{461} (2005), 2965--2984.

\bibitem{PS2013a}
B.~Pelloni and D.~A. Smith, \emph{Spectral theory of some non-selfadjoint
  linear differential operators}, Proc. R. Soc. Lond. Ser. A Math. Phys. Eng.
  Sci. \textbf{469} (2013), no.~2154, 20130019.

\bibitem{SSVW2015a}
C.~Schubert, C.~Seifert, J.~Voigt, and M.~Waurick, \emph{Boundary systems and
  (skew-) self-adjoint operators on infinite metric graphs}, Mathematische
  Nachrichten \textbf{288} (2015), no.~14-15, 1776--1785.

\bibitem{She2017a}
N.~E. Sheils, \emph{Multilayer diffusion in a composite medium with imperfect
  contact}, Applied Mathematical Modelling \textbf{46} (2017), 450--464.

\bibitem{DS2016a}
N.~E. Sheils and B.~Deconinck, \emph{Initial-to-interface maps for the heat
  equation on composite domains}, Studies in Applied Mathematics \textbf{137}
  (2016), no.~1, 140--154.

\bibitem{Smi2012a}
D.~A. Smith, \emph{Well-posed two-point initial-boundary value problems with
  arbitrary boundary conditions}, Math. Proc. Cambridge Philos. Soc.
  \textbf{152} (2012), 473--496.

\bibitem{STV2019a}
D.~A. Smith, T.~Trogdon, and V.~Vasan, \emph{Linear dispersive shocks},
  arXiv:1908.08716 [math.AP], 2019.

\bibitem{SE2020a}
Z.~Sobirov and M.~Eshimbetov, \emph{The {F}okas' unified transformation method
  for {A}iry equation on simple open star graph}, Bulletin of National
  University of Uzbekistan: Mathematics and Natural Sciences \textbf{3} (2020),
  438--447.

\bibitem{ZMGI2020a}
S.~Zheng, M.~H. Meylan, D.~Greaves, and G.~Iglesias, \emph{Water-wave
  interaction with submerged porous elastic disks}, Physics of Fluids
  \textbf{32} (2020), no.~4, 047106.

\end{thebibliography}

\end{document}